\documentclass[12pt,a4paper, reqno]{amsart}
\usepackage{amsfonts,amsmath,amssymb}
\usepackage{hyperref}
\usepackage{latexsym,fullpage,amsfonts,amssymb,amsmath,amscd,graphics,epic,enumerate}
\usepackage[all]{xy}
\usepackage{amssymb,amsthm,amsxtra}
\usepackage{color}
\usepackage{amscd}
\usepackage{amsthm}
\usepackage{amsfonts}
\usepackage{amssymb}
\usepackage{url}
\usepackage{bbm}
\usepackage{wasysym}
\usepackage{MnSymbol}

\theoremstyle{plain}% default
\newtheorem*{theorem*}{Theorem}
\newtheorem*{remark*}{Remark}
\newtheorem*{example*}{Example}
\newtheorem{lemma}{Lemma}[subsection]
\newtheorem{proposition}[lemma]{Proposition}
\newtheorem{corollary}[lemma]{Corollary}
\newtheorem{theorem}[lemma]{Theorem}

\newtheorem*{conjecture*}{Conjecture}

\newtheorem{prop}[lemma]{Proposition}

\theoremstyle{definition}
\newtheorem{definition}[lemma]{Definition}

\newtheorem{example}[lemma]{Example}

\theoremstyle{remark}
\newtheorem{remark}[lemma]{Remark}

\newcommand{\VW}{\bigdoublevee}
\newcommand{\VWd}{{\bigdoublevee}_d}
 \newcommand{\op}{\operatorname}

\newcommand{\Hom}{\operatorname{Hom}}

\newcommand{\Groth}{\operatorname{G}}
\newcommand{\Ind}{\operatorname{Ind}}
\newcommand{\id}{\operatorname{Id}}

\newcommand{\Ker}{\operatorname{Ker}}
\newcommand{\Ext}{\operatorname{Ext}}

\renewcommand{\gg}{\mathfrak{g}}
\newcommand{\hh}{\mathfrak{h}}

\newcommand{\bb}{\mathfrak{b}}

\renewcommand{\dim}{\mathrm{dim}}

 \def\fgl{\mathfrak{gl}} 
 \def\fp{\mathfrak{p}}
 
 \def\<{\langle}
  \def\>{\rangle}

  \DeclareMathOperator{\str}{str}
\DeclareMathOperator{\tr}{tr}

\newcommand{\End}{\mathrm{End}}

\def\quotient#1#2{%
    \raise1ex\hbox{$#1$}\Big/\lower1ex\hbox{$#2$}%
}

\newcommand{\InnaB}[1]{{{ #1} }}

\newcommand{\InnaC}[1]{{{#1}}}

\begin{document}

%\date{\today}
\title{Translation functors and decomposition numbers for the periplectic Lie superalgebra $\mathfrak{p}(n)$}

\author{M. Balagovic 
\and Z. Daugherty} 
\thanks{Research of Z.D. partially supported by National Science Foundation Grant DMS-1162010.} 
\author{ I. Entova-Aizenbud 
\and I. Halacheva
}
\thanks{Research of I. E-A partially supported by the MPI, Bonn, by UC Berkeley, and by ERC under grant No. 669655 (PI: David Kazhdan).}
\author {J. Hennig 
\and M. S. Im 
\and G. Letzter 
\and E. Norton 
\and V. Serganova 
\and C. Stroppel}
\thanks{Research of C. S. partially supported by the Hausdorff Center of Mathematics, Bonn.}
\address{M. B.: School of Mathematics, Statistics and Physics, Newcastle University, Newcastle upon Tyne NE1 7RU UK}
\email{martina.balagovic@newcastle.ac.uk} 
\address{Z. D.: Department of Mathematics, City College of New York, New York, NY 10031 USA}
\email{zdaugherty@gmail.com}
\address{I. E.: Department of Mathematics, Tel Aviv University, Tel Aviv, Israel}
\email{inna.entova@gmail.com}
\address{I. H.: Department of Mathematics and Statistics, Lancaster University, Lancaster, LA1 4YF UK}
\email{i.halacheva@lancaster.ac.uk}
\address{J. H.: Department of Math. and Statistical Sciences, University of Alberta, Edmonton, Alberta T6G 2G1 Canada}
\email{jhennig1@ualberta.ca}
\address{M. S. I.: Department of Mathematical Sciences, United States Military Academy, West Point, NY 10996 USA}
\email{meeseongim@gmail.com}
\address{G. L.: Department of Defense, Ft. George G. Meade, MD 20755 USA}
\email{gletzter@verizon.net}
\address{E. N.: Department of Mathematics, TU Kaiserslautern\\
67663 Kaiserslautern, Germany}
\email{norton@mathematik.uni-kl.de}
\thanks{Research of E. N. was supported by the MPI, Bonn.}
\address{V. S.: Department of Mathematics, University of California at Berkeley, Berkeley, CA 94720 USA}
\email{serganov@math.berkeley.edu}
\address{C. S.: Mathematisches Institut, Universitaet Bonn, Endenicher Allee 60,  53115 Bonn, Germany}
\email{stroppel@math.uni-bonn.de}
\date{}

\begin{abstract} We study the category $\mathcal{F}_n$ of finite-dimensional integrable representations of the periplectic Lie superalgebra $\mathfrak{p}(n)$.
We define an action of the Temperley--Lieb algebra with infinitely many generators and defining parameter $0$ on the category $\mathcal{F}_n$
by certain translation functors. We also introduce combinatorial tools, called weight diagrams and arrow diagrams for $\mathfrak{p}(n)$ resembling those for $\mathfrak{gl}(m|n)$. We discover two natural highest weight structures. Using the Temperley--Lieb algebra action and the combinatorics of weight and arrow diagrams,
we then calculate the multiplicities of irreducibles in standard and costandard modules 
and classify the blocks of $\mathcal{F}_n$. We also prove the surprising fact that indecomposable projective modules in this category are multiplicity-free.  
\end{abstract}

\maketitle
\setcounter{tocdepth}{2}
\tableofcontents
\section{Introduction}\label{sec:intro}

The simple Lie superalgebras $\mathfrak g$ over $\mathbb C$ were classified by Kac in 1977, \cite{K77}. These are divided into three groups: 
{\it basic} Lie superalgebras (which means the classical and exceptional series), the {\it strange} ones (often also called periplectic and queer) 
and the ones of {\it Cartan type}. The basic and strange Lie superalgebras are the ones whose even part $\mathfrak g_0$ is reductive and hence there is some hope 
to apply some classical methods in these cases. 
However, already the most natural question of computing characters of 
finite-dimensional irreducible representations of a simple Lie superalgebra turned out to be quite difficult due to the fact that not every finite-dimensional
representation is completely reducible. Using geometric methods (see \cite{Vsel}, \cite{PS}, \cite{GS}), and methods of 
categorification (see \cite{B03}, \cite{B04}, \cite{BS}, \cite{CLW}, \cite{ES1}), this problem was solved for all simple Lie superalgebras except the periplectic Lie 
superalgebra 
$\mathfrak{p}(n)$ defined in Section~\ref{sec:notation} below.\footnote{Strictly speaking  $\mathfrak{p}(n)$ is not simple, but it has the simple ideal of 
codimension $1$ consisting of matrices with zero supertrace which we could consider instead.}

In this paper we study the combinatorics and decomposition numbers of the category $\mathcal{F}_n$ of finite-dimensional representations of the algebraic supergroup
$P(n)$ with Lie superalgebra $\mathfrak p(n)$. 
By \cite{C} this category is a highest weight category (in the sense of \cite{CPS})
with standard objects given by so-called Kac modules. We describe this structure here in more detail stressing thereby the unusual behaviour specific to the periplectic case.

\begin{enumerate}
\item A peculiar fact is that the usual duality functor for this category (in contrast with
the classical category $\mathcal O$ or the category of finite-dimensional modules over basic Lie superalgebras) does not preserves simple objects.
\item 
Although projective modules are also injective and tilting, the combinatorics of standard modules is different from the combinatorics of costandard modules. In
particular, the standard and costandard modules with the same highest weight have different dimensions, (Lemma~\ref{lem:Kacdim}).  We call them therefore {\it thick} respectively {\it thin Kac modules}. A remarkable observation (Lemma~\ref{lem:ext}) is the existence of two highest weight structure with the thick and thin Kac modules interchanged. 
\item Another particular feature of this category is
that, although we have typical weights (for which indecomposable projectives are thick Kac modules and irreducibles are thin Kac modules) the indecomposable projective modules are never simple; see Theorems~\ref{thrm:main-mult2} and  \ref{cor:KL_coeff}. This fact  is related to the existence of a non-trivial Jacobson radical
in $\mathcal{U}(\mathfrak p(n))$, see \cite{Ser}.

\end{enumerate}
Key ingredients in studying the categories of finite-dimensional modules over basic and queer Lie superalgebras
were the existence of a large center in the universal enveloping algebra $\mathcal{U}(\mathfrak g)$, and so-called translation functors given by tensoring
with the natural representation (respectively its dual) followed by the projection onto a block.
In the case of $\mathfrak{p}(n)$, the center of $\mathcal{U}(\mathfrak {p}(n))$ is however trivial, \cite{G}. 
This results in the fact that there are only very few blocks in the category and therefore, one has to adjust the definition of translation functors accordingly to get finer information. 
The key step is the observation that, although  $\mathcal{U}(\mathfrak {p}(n))$ does not have a quadratic Casimir element, one can use the canonical embedding
$\mathfrak{p}(n)\subset\mathfrak{gl}(n|n)$ and the fact that $\mathfrak{gl}(n|n)$ is the direct sum of the adjoint and the coadjoint
representations of $\mathfrak{p}(n)$, see \eqref{dual},  to construct a $\mathfrak{p}(n)$-invariant element $\Omega\in\mathfrak {p}(n)\otimes\mathfrak{gl}(n|n)$.
This element $\Omega$ acts on $M\otimes V$ for any $\mathfrak{p}(n)$-module $M$ and $V=\mathbb{C}^{n|n}$, the vector representation of $\mathfrak{gl}(n|n)$.
A translation functor is then given by tensoring with $V$ followed by projection onto a generalized eigenspace of $\Omega$.\\

The main goal of our paper is to provide a detailed analysis of the combinatorics of the category $\mathcal{F}_n$ and the underlying highest weight structure. On the way we introduce weight diagrams in the spirit of \cite{BS} as a useful combinatorial tool which allows to easily compute the multiplicities of 
standard modules in indecomposable projective modules  and of simple modules in (co)standard modules. The surprising fact is that not only 
are these multiplicities at most $1$ (Theorems~\ref{thrm:main-mult2} and \ref{cor:KL_coeff}), but even the indecomposable projective modules are multiplicity-free (Corollary~\ref{cor:dim_hom_indec}), i.e. the dimension of the homomorphism space between
two indecomposable projective modules  is at most $1$. 

A standard fact of $\mathcal{F}_n$ is that projective modules are at the same time injective, and so in particular they have both a filtration by thick and by thin Kac modules.  Although the category is preserved under taking the dual of a representation, this duality is {\it not} fixing the simple objects, but permutes them in an interesting way.  In Section~\ref{ssec:duality_simple} we determine the highest weight of the dual to a simple module which at the same time gives us finer information about the structure of modules. 

Along the way we show (Theorem~\ref{ssec:categorical_action}) that the translation functors 
induce an action of the Temperley--Lieb algebra $\op{TL}_\infty=\op{TL}_\infty(q+q^{-1})$ attached to the infinite symmetric group $S_{\infty}$ on the category $\mathcal{F}_n$, where $q=\pm i$ is a primitive fourth root of unity. 
As far as we know this is the first instance of a categorical Hecke algebra action at roots of unity using abelian categories. This is in contrast with the approach of \cite{CK}, where homotopy categories of complexes are used, and in contrast to  \cite{EQ} where the Schur-Weyl dual quantum group for $\mathfrak{sl}_2$ at $q=i$ is categorified using the concept of p-dg categories.  As an application of this categorical Temperley-Lieb algebra action we deduce in Theorem~\ref{thrm:blocks} a description of the blocks in $\mathcal{F}_n$ with the action of the translation functors  (Corollary~\ref{thrm:blocks}). \\

This paper is the first part of a WINART group project which took part in March 2016. In the second part, \cite{second}, we introduce the {\it affine VW supercategory} $s\VW$ and the  {\it affine VW superalgebras} (or odd affine Nazarov-Wenzl algebras) $s\VWd$ which describe the natural transformations between the translation functors in the spirit of \cite{BS}, \cite{ES}, but using now the element $\Omega$ from above. 

In fact, up to crucial signs, the relations in this algebra $s\VWd$ are exactly the ones from \cite[Section~4]{Nazarov}, \cite[Definition 2.1]{ES}. We prove it has a basis completely analogous to \cite[Section~2.2]{ES}. Its polynomial subalgebra (generated by elements as in \cite[(2.4), Lemma~11.5]{ES}) is defined using the action of the element $\Omega$ from above. The algebra $s\VWd$ can also be seen as a degenerate affine version of the odd or marked Brauer algebra studied in \cite{M}, \cite{KT}. \\

The affine VW-superalgebras were considered independently in \cite{CP}, where a basis theorem is formulated as well. We should also mention some overlap with \cite{Col}, which independently introduced the 
Casimir elements and Jucys--Murphy elements and classified blocks using totally different methods. In both papers, the authors use the term {\it affine periplectic Brauer algebra} for what we call {\it affine VW-superalgebra}. We prefer the second terminology, since (by making the parallel to the type $A$ situation) we are in fact dealing here with an affine VW-algebra (in the sense of  \cite{ES}) in the supersetting.  We see the construction here in parallel with Drinfeld's degenerate version of an affine Hecke algebra, \cite{Dri} in type $A$. The affine VW-algebras from \cite{ES} were originally introduced in \cite{Nazarov} and further studied in \cite{AMR} and also briefly as degenerate affine BMW-algebras in \cite{DRV}. Using partly our results, aspects of the representation theory in the non-affine case are studied in \cite{CE}.

\subsection{Acknowledgments}
We thank J. Comes and M. Gorelik for sharing their ideas with us and M. Ehrig, D. Tubbenhauer and in the referee for extremely useful remarks on earlier versions of this paper. We also thank the Banff Center and the organizers of the WINART workshop for the hospitality, and the Hausdorff Center of Mathematics in Bonn for support. 

\subsection{Structure of the paper}
In Section~\ref{sec:BGG}, we introduce thin and thick Kac modules.  We consider the two different resulting highest weight structures on $\mathcal{F}_n$ given by either taking thick Kac modules as standard objects and thin Kac modules as costandard objects or vice versa with the corresponding partial order on weights defined in Section~\ref{ssec:order_weights}. In each case we prove a BGG-type reciprocity (Section~\ref{ssec:BGG}), and define the reduced Grothendieck group for $\mathcal{F}_n$ (Section~\ref{ssec:reduced_GR_group}). We also give necessary conditions for extensions between simple modules in Section~\ref{ssec:extensions}.

In Section~\ref{sec:transl_funct}, we define translation functors on $\mathcal{F}_n$, using the endomorphism $\Omega$ of the endofunctor $_- \otimes V$ of $\mathcal{F}_n$ which is our replacement for the missing Casimir element. In particular we project onto generalized eigenspaces for $\Omega$ instead of (as in familiar situations) blocks given by central characters. This is crucial here, since it will turn out that there are only a few blocks altogether.

This endomorphism $\Omega$ is defined in Section~\ref{ssec:Omega}. We compute the actions of translation functors on Kac modules in Section~\ref{ssec:acKac}, and show that they categorically lift the Temperley-Lieb relations (Section~\ref{ssec:categorical_action}) by giving explicit natural transformations realizing the desired relations of functors. 

In Section~\ref{sec:weight_diagrams}, we introduce the notion of weight diagrams for dominant weights and  explain the combinatorics of the actions of translation functors on (co)standard objects in terms of these diagrams (Section~\ref{ssec:translation_fun_diag}) as well as the combinatorics of the duality (Proposition~\ref{prop:duality}).

Section~\ref{sec:KL} is devoted to the computation of the decomposition numbers. This requires the definition of arrow diagrams, given in Section~\ref{ssec:arc_diagrams}. 

In Section~\ref{sec:transl_functors_proj} we describe the action of translation functors on indecomposable projective modules. We show that the result is indecomposable or zero. 

We  prove that indecomposable projective modules are multiplicity-free in Section~\ref{sssec:proj_mult_free} and deduce results  concerning translations of simple modules.  Section~\ref{ssec:socles_Kac} contains a description of the socles of the standard modules and the cosocles of the costandard modules in terms of arrow diagrams.

Finally, Section~\ref{sec:blocks} gives a classification of blocks in the category $\mathcal{F}_n$, and a description of the actions of translation functors on these blocks.

\section{The periplectic Lie supergroup and its finite dimensional representations}\label{sec:notation}
Throughout this paper, we will work over the base field $\mathbb C$. By a {\it vector superspace} we mean a  $\mathbb{Z}/2\mathbb{Z}$-graded vector space $V=V_{\bar 0}\oplus V_{\bar 1}$. The {\it parity} of a homogeneous vector $v \in V$ will be denoted by $p(v) \in \mathbb{Z}/2\mathbb{Z}=\{\bar 0, \bar 1\}$. If the notation $p(v)$ appears in formulas we always assume that $v$ is homogeneous. Throughout let $n>0$ be a fixed positive integer.

\subsection{The periplectic Lie superalgebra}
Let $V$ be an $(n|n)$-dimensional vector superspace equipped with a non-degenerate odd symmetric form
\begin{eqnarray}
\label{beta}
\beta:V\otimes V\to\mathbb C,\quad \beta(v,w)=\beta(w,v), \quad\text{and}\quad \beta(v,w)=0\,\,\text{if}\,\,p(v)=p(w).
\end{eqnarray}

Then $\op{End}_{\mathbb C}(V)$ inherits the structure of a vector superspace from $V$. By $\gg$ we denote the Lie superalgebra $\mathfrak{p}(n)$ of all $X\in\operatorname{End}_{\mathbb C}(V)$ preserving $\beta$, i.e. satisfying  $$\beta(Xv,w)+(-1)^{p({X})p(v)}\beta(v,Xw)=0.$$

\begin{remark}\label{rmk:basis}
If we choose dual bases $v_1, v_2, \ldots, v_n$ in $V_{\bar{0}}$ and $v_{1'}, v_{2'},\ldots, v_{n'}$ in $V_{\bar{1}}$, then the matrix of $X\in \mathfrak{p}(n)$ has the form $\left(\begin{smallmatrix}A&B\\C&-A^t\end{smallmatrix}\right)$
where $A,B,C$ are $n\times n$ matrices such that $B^t=B,\, C^t=-C$. In fact, we can write an explicit homogeneous basis of $\mathfrak{g}$ using our chosen basis.  For this let $E_{i j}, E_{i' j}, E_{i j'}, E_{i' j'}$ ($1 \leq i, j \leq n$) be the corresponding unit matrices in $ \mathfrak{gl}(n|n)$. Using the elements
\begin{equation}
\label{ABC}
  A_{i j}^{\pm} := E_{i j} \pm E_{j' i'} , \quad \quad
  B_{i j}^{\pm} := E_{i j'} \pm E_{j i'},\quad\quad
  C_{i j}^{\pm} := E_{i' j} \pm E_{j'i}, 
  \end{equation}
we obtain a homogeneous basis of $\mathfrak{g}$ as 
\begin{equation}
\label{ABCbasis1}
\{A_{ij}^{-}\}_{1 \leq i, j \leq n} \cup \{B_{ij}^+\}_{1 \leq i < j \leq n} \cup \{B_{ii}^+\}_{1 \leq i \leq n} \cup \{C_{ij}^-\}_{1 \leq i < j \leq n}.
\end{equation} 
\end{remark}

 As a vector space, $\fgl(n|n) = \fp(n) \oplus \fp(n)^\perp,$  where the complement is taken to be the dual with respect to the {\it supertrace form} on $\mathfrak{gl}(n|n)$ defined as
\begin{eqnarray}
\label{dual}
\<x,y\> = \str(xy), & \text{ where } & \str\begin{pmatrix} A&B\\C&D \end{pmatrix} = \tr(A) - \tr(D).
\end{eqnarray} 
The basis dual to \eqref{ABCbasis1} is given in Remark~\ref{dualbasis} below. For a representation $W$ of $\mathfrak{g}$ we denote by $W^*$ the dual representation with $x.f(w)=-(-1)^{p(f)p(w)}f(xw)$ for $x\in\mathfrak{g}$, $w\in W$, $f\in W^*$.  In particular, there is an  isomorphism of $\gg$-modules
\begin{eqnarray}
\label{eta}
\eta:\; V^*& \cong& V \otimes \Pi \mathbb{C}
\end{eqnarray}
induced by the form $\beta$ from \eqref{beta}. Here $\Pi \mathbb{C}$ denotes the $1$-dimensional (trivial) representation in degree $\overline{1}$.

\begin{lemma}
\label{lem:decoSE} Consider the  decomposition $V\otimes V=S^2 V\oplus\Lambda^2V$, where $S^2V$ is the symmetric and $\Lambda^2V$ is the exterior power of $V$. 
Then $S^2V$ and $\Lambda^2V$ are indecomposable $\mathfrak g$-modules and $c=\sum_{i}(v_i\otimes v_{i'}-v_{i'}\otimes v_{i})\in\Lambda^2 V$ spans the unique trivial 
submodule in $V\otimes V$. Moreover, the algebra $\operatorname{End}_{\gg}(V\otimes V)$ is 3-dimensional with basis the identity and 
\begin{eqnarray}
\label{sande}
s: v\otimes w\mapsto (-1)^{p(v)p(w)} w\otimes v&\text{and}&e={\beta^*}\circ\beta: 1\mapsto c.
\end{eqnarray}
\end{lemma}
\begin{remark} Note that  $e=e\circ s=-s\circ e$ and $\operatorname{End}_{\gg}(V\otimes V)$ is by the above decomposition isomorphic to the algebra of lower triangular $2\times 2$-matrices via $s\mapsto \left(\begin{smallmatrix}1&0\\0&-1\end{smallmatrix}\right)$ and $e\mapsto\left(\begin{smallmatrix}0&0\\1&0\end{smallmatrix}\right)$.
\end{remark}
\begin{proof} This is proved for instance in \cite[Section 6.1]{M}.
\end{proof}

%\begin{proof} Let us note that $\Pi(S^2V)$ is isomorphic to the adjoint representation via the formula:
%$$T_{v\cdot w}(x)=\beta(w,x)v+(-1)^{p(v)p(w)}\beta(v,x)w, \quad v,w,x\in V.$$
%
%For $n\geq 3$, we know from the structure theory of Lie superalgebras, \cite{K77}, that $[\gg,\gg]$ is a simple ideal of codimension $1$,
%hence via the above isomorphism the kernel of $S^2V\to \mathbb C$ is an irreducible $\gg$-module. 
%The isomorphism \eqref{eta} induces an isomorphism $\Lambda^2V\simeq (S^2 V)^*$, hence by duality $\Lambda^2V$ is also indecomposable with
%trivial submodule given by the image of $\mathbb C\xrightarrow{\beta^*}V\otimes V$. This implies all the statements of the lemma for $n\geq 3$.
%
%For $n=1,2$ lemma follows by direct calculations. If $n=1$, then $\gg$ consists of matrices of the form $\left(\begin{smallmatrix} a&b\\0&-a \end{smallmatrix}\right)$
%and the adjoint representation is obviously indecomposable of length $2$. If $n=2$ then the adjoint representation has length $3$ with simple socle
%given by the matrices
%$\left(\begin{smallmatrix}A&B\\0&-A^t\end{smallmatrix}\right)$, with $\operatorname{tr}A=0$, and the quotient by the socle is a $(1|1)$-dimensional indecomposable
%$\gg$-module.
%\end{proof}

Note that $\mathfrak g_{\bar 0}$ is isomorphic to $\mathfrak{gl}(n)$ and $\gg_{\bar 1}$ decomposes as $$\gg_{\bar 1} = \mathfrak g_{-1} \oplus \mathfrak g_1$$ where
$\mathfrak g_1=\{\left(\begin{smallmatrix}0&B\\0&0\end{smallmatrix}\right)\}$ and $\mathfrak g_{-1}=\{\left(\begin{smallmatrix}0&0\\C&0\end{smallmatrix}\right)\}$. Thus, 
$\mathfrak g$ has a $\mathbb Z$-grading $\mathfrak g=\mathfrak g_{-1}\oplus \mathfrak g_0\oplus \mathfrak g_1$. It is given by the adjoint action of the element
 \begin{eqnarray}
 \label{theh}
 h&:=&\frac{1}{2}\operatorname{diag}(1,1,\ldots,1,-1,-1,\ldots,-1)\in\gg_0=\gg_{\overline{0}}.
 \end{eqnarray}

\subsection{The  category \texorpdfstring{$\mathcal{F}_n$}{F} of finite-dimensional integrable representations}
By $\mathcal{F}_n$ we denote the abelian category of finite-dimensional representations of the corresponding supergroup $G$, i.e., finite-dimensional
$\gg$-modules integrable over $G_0\cong GL(n)$. We will denote by $\Pi$ the {\it parity switching functor} $- \otimes {\mathbb C}^{(0|1)}$ 
(that is tensoring with the odd trivial representation  $\Pi \mathbb{C}$ on the right).

By definition, the morphisms in $\mathcal{F}_n$ are {\it even} $G_0$-morphisms (otherwise $\mathcal{F}_n$ would not be abelian),
i.e., $\operatorname{Hom}_{\mathcal{F}_n}(X,Y)$ is a vector space and not a vector superspace.
It will be convenient however to consider the vector superspace
$$\operatorname{Hom}_{\gg}(X,Y)=\operatorname{Hom}_{\mathcal{F}_n}(X,Y)\oplus \operatorname{Hom}_{\mathcal{F}_n}(X,\Pi Y)$$
and set 
$$\dim\operatorname{Hom}_{\gg}(X,Y)=\dim\operatorname{Hom}_{\mathcal{F}_n}(X,Y)+ \dim\operatorname{Hom}_{\mathcal{F}_n}(X,\Pi Y).$$
In this way we also define the Jordan--H\"{o}lder multiplicities, $[X: L]$ will denote the total number of simple subquotients isomorphic to $L$ or $\Pi L$.

We fix  the standard Cartan subalgebra $\hh$ of diagonal matrices in $\gg_0$ with its standard dual basis $\{\varepsilon_1,\dots,\varepsilon_n\}$ and denote by $\Delta=\Delta(\gg_{-1})\cup\Delta(\gg_{0})\cup\Delta(\gg_{1})$ the set of roots divided according to the $\mathbb{Z}$-grading. Then weights of modules in $\mathcal{F}_n$ are of the form 
\begin{eqnarray}
\label{defdom}
\lambda&=&(\lambda_1,\dots,\lambda_n)=\sum_{i=1}^n\lambda_i\varepsilon_i,\quad \lambda_i\in\mathbb Z.
\end{eqnarray}
A weight $\lambda$ is (integral) {\it dominant} if and only if $\lambda_1\geq\lambda_2\geq\dots\geq\lambda_n$. 
We denote by $V(\lambda)$ the simple $\gg_0$-module with highest weight $\lambda$ with respect to the fixed Borel $\bb_0$ of $\gg_0$. 
% {\color{red}Dear Vera, I am confused. Don;t we have to define $\bb_0$ as the lower diagonal matrices to get later the correct Borel??}
% \Vera{no, it is correct this way, otherwise dominance condition changes, For $\bb$ we take two options
% $\bb_0\oplus\gg_1$ and $\bb_0\oplus\gg_{-1}$}

Denote the set of dominant weights by $\Lambda_n$. The simple objects of $\mathcal{F}_n$, up to isomorphism and parity switch, will be parametrized by the set $\Lambda_n$, with the simple module $L(\lambda)$ having the highest weight $\lambda$ with respect to the Borel subalgebra $\bb_0 \oplus \gg_{-1}$.
\small
We use the following abbreviations
\begin{equation}
\label{omegarho}
\quad |\lambda| = \sum_i \lambda_i,\quad\omega = \sum_{i=1}^n\varepsilon_i, \quad \rho=\sum_{i=1}^n (n-i)\varepsilon_i,\quad \bar\lambda:=\lambda+\rho,\quad c_\lambda=\{\bar\lambda_1,\dots,\bar\lambda_n\}\subset \mathbb Z,
\end{equation}
\normalsize
$$\gamma=\sum_{\alpha\in\Delta^+(\gg_{-1})}\alpha=\sum_{i=1}^n(1-n)\varepsilon_i= (1-n)\omega,\quad
\tilde\gamma=\sum_{\alpha\in\Delta^+(\gg_{1})}\alpha=\sum_{i=1}^n(n+1)\varepsilon_i= (n+1)\omega.$$

Note that $\lambda$ is dominant if and only if $\bar\lambda$ is a strictly decreasing sequence of integers. 

\section{Kac modules and BGG reciprocity}\label{sec:BGG}

We introduce now the thin and thick Kac modules as induced modules, where we use the usual induction and coinduction functors for a Lie super subalgebra  $\mathfrak{a}\subset\mathfrak{g}$ forming the following pairs of adjoint functors with the restriction functor $\op{Res}^\mathfrak{g}_\mathfrak{a}$
\begin{eqnarray}
\label{Indcoind}
(\operatorname{Ind}^{\mathfrak{g}}_\mathfrak{a},\op{Res}^\mathfrak{g}_\mathfrak{a})
&\text{and}&
(\op{Res}^\mathfrak{g}_\mathfrak{a}, \operatorname{Coind}^{\mathfrak{g}}_\mathfrak{a})
\end{eqnarray}
and the functors $H^0(\mathfrak{a},_-)$ and  $H_0(\mathfrak{a},_-)$ on $\mathcal{F}_n$,  of taking invariants and coinvariants,  see e.g. \cite{Fuks} for more details.

\subsection{Thin and Thick Kac modules}\label{ssec:Kac_modules}
Let $\lambda$ be a dominant weight. We define the {\it thin Kac module} corresponding to $\lambda$ as
$$\nabla(\lambda)=\Pi^{n(n-1)/2}\operatorname{Ind}^{\gg}_{\gg_0\oplus\gg_1}V(\lambda-\gamma)\simeq \operatorname{Coind}^{\gg}_{\gg_0\oplus\gg_1}V(\lambda).$$
Note that $\nabla(\lambda)$ is a free and cofree $\mathcal{U}(\gg_{-1})$-module and we have
$$H_0(\gg_{-1},\nabla(\lambda))=\Pi^{n(n-1)/2} V(\lambda-\gamma),\quad H^0(\gg_{-1},\nabla(\lambda))=V(\lambda).$$

Similarly, the {\it thick Kac module} corresponding to $\lambda$ is defined as
$$\Delta(\lambda)=\operatorname{Ind}^{\gg}_{\gg_0\oplus\gg_{-1}}V(\lambda)\simeq \Pi^{n(n+1)/2}\operatorname{Coind}^{\gg}_{\gg_0\oplus\gg_{-1}}V(\lambda+\tilde{\gamma}).$$
Furthermore, $\Delta(\lambda)$ is a free and cofree $\mathcal{U}(\gg_{1})$-module and we have
$$H_0(\gg_{1},\Delta(\lambda))=V(\lambda),\quad H^0(\gg_{1},\Delta(\lambda))= \Pi^{n(n+1)/2} V(\lambda+\tilde\gamma).$$
 
The thick Kac module $\Delta(\lambda)$ is a highest weight module with highest weight $\lambda$ with respect to Borel subalgebra $\bb_0\oplus\gg_{-1}$, hence 
by a standard argument, it
has unique simple quotient $L(\lambda)$, see for example \cite{K77}. Note that by Frobenius reciprocity $L(\lambda)$ coincides with the socle of
$\nabla(\lambda)$.

\begin{lemma}\label{lem:Kacdim} The dimensions of the thin and the thick Kac modules are given by
$$\dim\nabla(\lambda)=2^{n(n-1)/2}\dim V(\lambda),\quad \text{  and  } \quad \dim\Delta(\lambda) =2^{n(n+1)/2}\dim V(\lambda).$$ 
\end{lemma}

\begin{proof} We have the following isomorphisms of $\gg_0$-modules:
$$\nabla(\lambda)\simeq \Lambda(\gg_{-1}^*)\otimes V(\lambda),\quad\text{and}\quad \Delta(\lambda)\simeq \Lambda(\gg_{1})\otimes V(\lambda).$$
Since $\dim \gg_{-1}=n(n-1)/2$ and $\dim\gg_1=n(n+1)/2$, the statement follows. 
\end{proof}

\begin{example}
With respect to the Borel subalgebra $\bb_0 \oplus \gg_{-1}$, the highest weight of the natural $\gg$-module $V$ is (odd) $-\varepsilon_n$. Thus, we have
$V\simeq \Pi L(-\varepsilon_n)$. 
\end{example}

\subsection{BGG reciprocity}\label{ssec:BGG}
The category $\mathcal{F}_n$ has enough projective and injective objects, since
$\operatorname{Ind}^\gg_{\gg_0}V(\lambda)$ is both projective and injective (cf. \cite{BKN}). We will denote the full subcategory of projectives in $\mathcal{F}_n$ by $\mathcal P_n$. For any dominant weight $\lambda$ we denote by $P(\lambda)$ the projective cover and by $I(\lambda)$ the injective hull of $L(\lambda)$. 
Furthermore, every projective has a standard filtration (that means with subquotients isomorphic to thick Kac modules) and a costandard filtration (that means with subquotients isomorphic to thin Kac  modules)
since $\operatorname{Ind}^\gg_{\gg_0}V(\lambda)$ admits such filtrations and then also any direct summand, because of the following Ext-vanishing result. Moreover this result implies that the multiplicity $(P : \Delta(\lambda))$ or $(P : \nabla(\lambda))$ of how often $\Delta(\lambda)$ or $\nabla(\lambda)$, respectively, appears as a subquotient for $P$ is independent of the choice of the filtration.  The following lemmas are standard general results, see e.g. \cite[Section~4]{C} for our special case.

\begin{lemma}\label{lem:ext} For any dominant $\lambda$ and $\mu$ we have
$$\dim\Hom_\gg(\Delta(\lambda),\nabla(\mu))=\delta_{\lambda,\mu},\quad \text{  and  } \quad \Ext^1_{\mathcal{F}_n}(\Delta(\lambda),\nabla(\mu))=0,$$ and similarly,
$$\dim\Hom_\gg(\nabla(\lambda),\Delta(\mu))=\delta_{\lambda-2\omega,\mu},\quad \text{  and  } \quad \Ext^1_{\mathcal{F}_n}(\nabla(\lambda),\Delta(\mu))=0.$$
\end{lemma}
%\begin{proof} Applying adjunction/Frobenius reciprocity \eqref{Indcoind}, we have:
%\begin{eqnarray*}
%\Hom_\gg(\Delta(\lambda),\nabla(\mu))&\cong&\Hom_{\gg_0\oplus\gg_{-1}}(V(\lambda),\nabla(\mu))\cong\Hom_{\gg_0}(V(\lambda),H^0(\gg_{-1},\nabla(\mu)))\\
%&\cong &\Hom_{\gg_0}(V(\lambda),V(\mu)) \cong \delta_{\lambda,\mu} \mathbb{C}, \quad\mbox{  and  }\\
%\Ext_{\mathcal{F}_n}^1(\Delta(\lambda),\nabla(\mu))&\cong&\Ext^1_{\gg_0\oplus\gg_{-1}}(V(\lambda),\nabla(\mu))\cong\Ext^1_{\gg_0}(V(\lambda),V(\mu)) \cong0.
%\end{eqnarray*}
%The last equality follows from the fact that the category $\op{Rep}^{fd}(\gg_0)$ of finite-dimensional representations of $\gg_0$ is a semisimple category. Moreover, 
%\begin{eqnarray*}
% \Hom_\gg( \nabla(\lambda),\Delta(\mu))&\cong&\Hom_{\gg_0\oplus\gg_{1}}(V(\lambda - \gamma),\nabla(\mu))\cong
%\Hom_{\gg_0}(V(\lambda - \gamma),H^0(\gg_{1}, \nabla(\mu)))\\
% &\cong&\Hom_{\gg_0}(V(\lambda - \gamma),V(\mu +\tilde \gamma )) \cong \delta_{\lambda-2\omega,\mu} \mathbb{C}.
%\end{eqnarray*}
%Again, the last equality follows from the fact that $\op{Rep}^{fd}(\gg_0)$ is a semisimple category.
%\end{proof}
\begin{remark}
\label{par}
 Keeping track of the parity, the map $\Delta(\lambda) \to \nabla(\lambda)$ is even, while the map $\nabla(\lambda) \to \Delta(\lambda -2\omega)$ has parity $(-1)^n$.
\end{remark}

\begin{corollary}[BGG-reciprocity]\label{cor:BGG} For any  dominant weights $\lambda, \mu$  the following holds:
$$(P(\lambda):\Delta(\mu))=[\nabla(\mu):L(\lambda)], \;\;\;\; (P(\lambda):\nabla(\mu+2\omega))=[\Delta(\mu):L(\lambda)].$$
\end{corollary}
%\begin{proof}We have the equalities
%\begin{eqnarray*}
% [\nabla(\mu):L(\lambda)]&=&\dim \Hom_\gg(P(\lambda),\nabla(\mu))=\sum_\nu(P(\lambda):\Delta(\nu))\dim \Hom_\gg(\Delta(\nu),\nabla(\mu))\\&=&(P(\lambda):\Delta(\mu)),
%\end{eqnarray*}
%and moreover
%\begin{eqnarray*}
%[\Delta (\mu):L(\lambda)]&=&\dim \Hom_\gg(P(\lambda), \Delta (\mu)) =\sum_\nu (P(\lambda): \nabla(\nu))\dim \Hom_\gg(\nabla(\nu),\Delta(\mu))\\
% &=&\sum_\nu (P(\lambda): \nabla(\nu)) \delta_{\nu-2\omega, \mu} = (P(\lambda):\nabla(\mu + 2\omega)),
%\end{eqnarray*}
%and so we are done.
%\end{proof}

The following crucial observation will be used throughout the paper.
\begin{proposition}\label{lem:injective_hull_simple}
 Let $\lambda$ be dominant. Then $\Pi^n P(\lambda - 2 \omega)$ is the injective hull of $L(\lambda)$.
\end{proposition}

\begin{proof}
 Let $I(\lambda)$ be the injective hull of $L(\lambda)$. It satisfies:
 $(I(\lambda):\nabla(\mu)) = [\Delta(\mu):L(\lambda)].$
 This is proved analogously to the BGG reciprocity above.
 Yet $I(\lambda)$ is an indecomposable projective module, and the multiplicities in its filtration by thin Kac modules coincide with that of $P(\lambda - 2\omega)$. Thus $I(\lambda) \cong P(\lambda - 2\omega)$. Hence it is isomorphic to $P(\lambda - 2 \omega)$ or its parity shift. The proposition follows then from Remark~\ref{par}.
\end{proof}

\subsection{Ordering on weights}\label{ssec:order_weights}
We define a partial order on the dominant weights $\Lambda_n$ as follows: we say that $\mu \geq \lambda$ if $\mu_i \leq \lambda_i$ for each $i$. 

\begin{lemma}
\label{lem:syl2}
 If $(P(\lambda): \Delta(\mu)) \neq 0$ (equivalently, $[\nabla(\mu): L(\lambda)] \neq 0$) then $\mu \geq \lambda$.
 
 Similarly, if $(P(\lambda): \nabla(\mu+2\omega)) \neq 0$ (equivalently, $[\Delta(\mu): L(\lambda)] \neq 0$) then $\mu \geq \lambda$.
\end{lemma}
This lemma means that the above order gives a highest-weight structure on the category $\mathcal{F}_n$, with $\Delta(\lambda)$ as standard objects. 
\begin{proof}
 Recall from Corollary~\ref{cor:BGG} that $(P(\lambda): \Delta(\mu)) \neq 0$ implies $[\nabla(\mu): L(\lambda)] \neq 0$. 
 Thus, $$[\op{Res}^{\gg}_{\gg_0} \nabla(\mu): V(\lambda)] = [\Lambda(\Lambda^2(V_{\overline{0}})) \otimes V(\mu): V(\lambda)] \neq 0.$$ The $\gg_0$-weights  in $\Lambda(\Lambda^2(V_{\overline{0}}))$ are sums of $\{\varepsilon_i +\varepsilon_j \mid i \neq j\}$. So $\lambda = \mu + \sum_i a_i \varepsilon_i$, $a_i \geq 0$, as required. Similarly, by Corollary~\ref{cor:BGG} $(P(\lambda): \nabla(\mu+2\omega)) \neq 0$ implies $[\Delta(\mu): L(\lambda)] \neq 0$, and hence $$[\op{Res}^{\gg}_{\gg_0} \Delta(\mu): V(\lambda)] = [\Lambda(S^2(V_{\overline{0}})) \otimes V(\mu): V(\lambda)] \neq 0.$$ The $\gg_0$-weights in $\Lambda(S^2(V_{\overline{0}}))$ are sums of $\{\varepsilon_i +\varepsilon_j \mid i , j\}$. So $\lambda = \mu + \sum_i a_i \varepsilon_i$, $a_i \geq 0$, as required. 
\end{proof}

\begin{remark}\label{rmk:2_highest_weight_struct}
 The category $\mathcal{F}_n$ is a {\it  highest weight category} in the sense of Cline, Parshall and Scott (see \cite{CPS}). It has two natural highest weight category structures. The first structure, which we will mostly use, with the partial order on $\Lambda_n$ given above, has thick Kac modules $\Delta(\lambda)$ as standard modules, and thin Kac modules $\nabla(\lambda)$ as costandard modules. The second  structure has $\Delta(\lambda)$ as costandard modules, and $\nabla(\lambda)$ as standard modules. Note that the standard module corresponding to $L(\tau)$ would not be $\nabla(\tau)$, but rather $\nabla(\lambda)$, where $\lambda$ is obtained from $\tau$ using the procedure described in Proposition~\ref{prop:socle_thick_kac}. Note however,  that there is no ``duality'' contravariant endofunctor on $\mathcal{F}_n$ interchanging the thick and the thin Kac modules.
\end{remark}

\subsection{Typical weights}\label{ssec:typical_weights}
A dominant weight $\lambda=(\lambda_1,\dots,\lambda_n)$ is {\it typical} if $\lambda_1>\ldots>\lambda_n$.
The following result can be found in \cite{K78}.

\begin{lemma}\label{lem:typical} We have: $\Delta(\lambda)\simeq P(\lambda)$ and $\nabla(\lambda)\simeq L(\lambda)$ if and only if
$\lambda$ is typical.
\end{lemma}

\subsection{Tilting modules}\label{ssec:tilting_modules}
An object $X\in\mathcal{F}_n$ is {\it tilting} if it has both a filtration by thin and by thick Kac modules. Note that this is equivalent to
$X$ being free over $\mathcal{U}(\gg_1)$ and $\mathcal{U}(\gg_{-1})$. Also, if $X$ is tilting then $X^*$ is tilting.  By Lemma~\ref{lem:ext} a direct summand of a tilting object is again tilting. 

\begin{lemma}\label{lem:tilt} In the category $\mathcal{F}_n$, an object $X$ is tilting if and only if $X$ is projective.
\end{lemma}
\begin{proof} If $X$ is projective, then it is tilting as we explained above. 
To prove the opposite, assume $X$ is tilting. First note that $\Delta(\lambda)\otimes \nabla(\mu)\simeq \operatorname{Ind}^\gg_{\gg_0}(V(\lambda)\otimes V(\mu))$ is projective for any dominant weight $\lambda,\mu$. 
Since $X$ has filtrations by thin respectively by thick Kac modules, we conclude that $X\otimes X$ is projective and therefore $X\otimes X^*\otimes X\cong X\otimes X\otimes X^*$ is projective. 
Since the the counit map followed by the unit map defines a morphism $X\rightarrow X\otimes X^*\otimes    X \rightarrow X$ equal to the identity map, $X$ is isomorphic to a direct summand of
$X\otimes X^*\otimes X$. Thus, $X$ is projective as well. 
\end{proof}

\subsection{Duality for Kac modules}\label{ssec:duality_Kac_modules}
\begin{lemma}\label{lem:dual_Kac}
 We have $\Delta(\lambda)^* = \Delta(-w_0\lambda - \tilde{\gamma})$ and $\nabla(\lambda)^* = \nabla(-w_0\lambda + \gamma)$, where $w_0$ is the longest element in the Weyl group of $\gg_0$.
\end{lemma}

 In other words, $\Delta(\lambda)^* = \Delta(\mu-2\omega)$, and $\nabla(\lambda)^* = \nabla(\mu)$ where $\mu + \rho = - w_0 (\lambda+ \rho)$.

\begin{proof}
 By definition, for every $\gg$-module $M$, 
 \begin{eqnarray*}
\Hom_{\gg}(M, \Delta(\lambda)^*) &\cong &\Hom_{\gg}(M \otimes \Delta(\lambda), \mathbb{C}) \cong \Hom_{\gg_0 \oplus \gg_{-1}}(\op{Res} M \otimes V(\lambda), \mathbb{C}) \\
&\cong&\Hom_{\gg_0 \oplus \gg_{-1}}(\op{Res} M,  V(\lambda)^*) \cong \Hom_{\gg}( M, \op{Coind}^{\gg}_{\gg_0 \oplus \gg_{-1}} V(\lambda)^*).
 \end{eqnarray*}
Thus, $$\Delta(\lambda)^* \cong \op{Coind}^{\gg}_{\gg_0 \oplus \gg_{-1}} V(\lambda)^* \cong \op{Coind}^{\gg}_{\gg_0 \oplus \gg_{-1}} V(-w_0 \lambda) \cong  \Delta(-w_0 \lambda - \tilde \gamma).$$
Similarly, 
\begin{eqnarray*}
\Hom_{\gg}( \nabla(\lambda)^*, M) &\cong \Hom_{\gg}(\mathbb{C}, M \otimes \nabla(\lambda)) \cong \Hom_{\gg_0 \oplus \gg_{1}}(\mathbb{C}, \op{Res} M \otimes V(\lambda))  \\
&\cong\Hom_{\gg_0 \oplus \gg_{1}}( V(\lambda)^*, \op{Res} M ) \cong \Hom_{\gg}(\op{Ind}^{\gg}_{\gg_0 \oplus \gg_{1}} V(\lambda)^*,  M).
 \end{eqnarray*}
Thus $\nabla(\lambda)^* \cong \op{Ind}^{\gg}_{\gg_0 \oplus \gg_{1}} V(\lambda)^* \cong \op{Ind}^{\gg}_{\gg_0 \oplus \gg_{1}} V(-w_0 \lambda) \cong \nabla(-w_0 \lambda + \gamma).$
Hence the claim follows.
\end{proof}

\subsection{Extensions of simples}\label{ssec:extensions}
\begin{proposition}\label{prop:extension} Let $\lambda,\mu$ be dominant weights and  $h$ as in \eqref{theh}. Let
$$0\to L(\mu)\to X\to L(\lambda)\to 0$$
be a non-trivial extension. Then we have the following two possibilities
\begin{enumerate}[i.)]
\item either $\lambda(h)<\mu(h)$ and $X$ is a quotient of $\Delta(\lambda)$ or of its parity switch; or 
\item  $\mu(h)<\lambda(h)$ and $X$ is a submodule of $\nabla(\mu)$ or of its parity shift.
\end{enumerate}
\end{proposition}
\begin{proof} 
First, we note that if $\lambda(h)=\mu(h)$, then $H^0(\gg_{-1},X)\cong V(\lambda)\oplus V(\mu)$. 
Indeed, $t:=\lambda(h)=\mu(h)$ is the lowest eigenvalue of $h$. Therefore the $h$-eigenspace with eigenvalue $t$ is 
a $\gg_0$-submodule annihilated by $\gg_{-1}$.

Then $\mathcal{U}(\gg)V(\lambda)=\mathcal{U}(\gg_1)V(\lambda)$ is a proper submodule of
$X$. Hence the extension must be trivial.

Assume now that $\lambda(h)<\mu(h)$, then $V(\lambda)\subset H^0(\gg_{-1},X)$. Since the extension is non-split, we have $X=\mathcal{U}(\gg)V(\lambda)$.
Thus, by adjointness \eqref{Indcoind}, $X$ is a quotient of $\Delta(\lambda)$.

Finally, let us assume that $\mu(h)<\lambda(h)$. Then consider the dual extension
$$0\to L(\lambda)^*\to X^*\to L(\mu)^*\to 0,$$
and use  $H^0(\gg_{1},L(\mu)^*)=V(-w_0(\mu))=V(\mu)^*$ and  $H_0(\gg_{1},L(\lambda)^*)=V(-w_0(\lambda))=V(\lambda)^*.$ Since we have $-w_0(\lambda)(h)<-w_0(\mu)(h)$, we obtain that $V(\mu)^*$ is a submodule in $H^0(\gg_{1},X)$.
Hence by Frobenius reciprocity, we obtain that $X^*$ is a quotient of the induced module $N:=\mathcal{U}(\gg)\otimes_{\gg_0\oplus\gg_1}V(\mu)^*$.
Dualizing again, we get that $X$ is a submodule of $N^*$. On the other hand,
$N^*$ is isomorphic to $\nabla(\mu)$. Hence the proposition follows.
\end{proof}

\subsection{Reduced Grothendieck group}\label{ssec:reduced_GR_group}
Let $\Groth_n$ denote the {\it reduced Grothendieck group} of
$\mathcal{F}_n$. By this we mean the usual Grothendieck group quotient by the relation $[\Pi X]=[X]$.
Denote by $\Groth_n(\nabla)$ and $\Groth_n(\Delta)$ the subgroups of $\Groth_n$ generated by the thin and thick Kac modules respectively and 
consider the pairing
$\langle -,-\rangle:\Groth_n(\Delta)\times \Groth_n(\nabla) \to \mathbb Z$ given by
\begin{eqnarray}
\label{bil}
\langle [M],[N]\rangle:=\dim\Hom_{\gg}(M,N).
\end{eqnarray}
Then $\{[\Delta(\lambda)]\}$ and $\{[\nabla(\lambda)]\}$ are dual bases.

Consider the full subcategories of modules with filtrations by thin and thick Kac modules, respectively. 
A module $X$ lies in both subcategories if and only if $X$ is tilting and therefore projective by Lemma~\ref{lem:tilt}. These subcategories give us two groups $\Groth_n (\Delta)$ and $\Groth_n (\nabla)$, both mapping into $\Groth_n$. 
If we denote by $\Groth^\oplus(\mathcal P_n)$, the reduced split Grothendieck group of the full (additive) subcategory of projective modules in $\mathcal{F}_n$, then the obvious inclusion maps fit into a commutative square 
 
$$\begin{CD}
\Groth^\oplus(\mathcal P_n)@>>>\Groth_n(\Delta)\\
@VVV@VVV\\
\Groth_n(\nabla)@>>>\Groth_n.
\end{CD}$$
The restriction of $\langle -,- \rangle$ to $\Groth^\oplus(\mathcal P_n)\times \Groth^\oplus(\mathcal P_n)$ satisfies the relation
$$\langle [P],[Q]\rangle=\langle [Q\otimes T],[P]\rangle,$$
where $T$ is the one-dimensional $\gg$-module with highest weight $2\omega$, \cite[Lemma 9.4]{Serqr}. This follows from Proposition~\ref{lem:injective_hull_simple}, and we obtain for any dominant weights $\lambda,\mu$ the following.
\begin{corollary} It holds $\dim\Hom_\gg(P(\lambda),P(\mu))=\dim\Hom_\gg(P(\mu+2\omega),P(\lambda)).$ 
 \begin{proof}
Since $\dim \Hom_\gg(P(\lambda),P(\mu))=\dim\Hom_\gg(P(\lambda),I(\mu+\omega))=[P(\lambda):L(\mu+2\omega)]
= \dim\Hom_\gg(P(\mu+2\omega),P(\lambda))$, the claim follows.
\end{proof}
\end{corollary}
\section{Translation functors and the fake Casimir element}\label{sec:transl_funct}

\subsection{Endomorphism of the functor \texorpdfstring{$_-\otimes V$}{(times V)}}\label{ssec:Omega}
Consider the following endofunctor of $\mathcal{F}_n$,
\begin{eqnarray}
\Theta'&=&_-\otimes V:\quad\mathcal{F}_n\longrightarrow\mathcal{F}_n.
\end{eqnarray}
We would like to investigate the direct summands of this functor. In order to do so, we introduce a crucial  natural endotransformation $\Omega$ of $\Theta'$.

Recall the even non-degenerate invariant 
supersymmetric form \eqref{dual} on  $\mathfrak{gl}(n|n)$, and consider the involutive anti-automorphism
$\sigma:\mathfrak{gl}(n|n) \to \mathfrak{gl}(n|n)$ defined as
$$\left(\begin{smallmatrix}A&B\\C&D\end{smallmatrix}\right)^\sigma:=\left(\begin{smallmatrix}-D^t&B^t\\-C^t&-A^t\end{smallmatrix}\right).$$
Then $\gg=\fp(n)\subset\mathfrak{gl}(n|n)$ is precisely given by all elements fixed by $\sigma$ and 
$$\gg':=\{x\in\mathfrak{gl}(n|n)\,|\,x^\sigma=-x\}=\fp(n)^\perp.$$

Observe that $\gg$ and $\gg'$ are maximal isotropic subspaces with respect to the form $\<_-,_-\>$ from \eqref{dual} and hence this form defines a non-degenerate
$\gg$-invariant pairing $\gg\otimes\gg'\to\mathbb C$. 
\begin{definition}
We pick now a $\mathbb{Z}$-homogeneous basis $\{X_i\}$ in $\gg$ and the basis $\{X^i\}$ in $\gg'$ such that
$\left\<X^i,X_j\right\>=\delta_{ij}$ and define the {\it fake Casimir element}
\begin{eqnarray}
\label{Omega}
\Omega:=2\sum_i X_i\otimes X^i\in\gg\otimes\gg'\subset \gg\otimes\mathfrak{gl}(n|n).
\end{eqnarray}
\end{definition}
\begin{remark}
\label{dualbasis}
 Consider the basis of $\gg$ from Remark~\ref{rmk:basis}. Then the dual basis is $$\left\{\tfrac{1}{2}A_{ji}^{+}\right\}_{1 \leq i, j \leq n} \cup \;\left\{ -\tfrac{1}{2} C_{ji}^+\right\}_{1 \leq i < j \leq n}\cup\; \left\{-\tfrac{1}{4} C_{ii}^+\right\}_{1 \leq i \leq n}\cup\; \left\{\tfrac{1}{2} B_{ji}^-\right\}_{1 \leq i < j \leq n}.$$ 
\end{remark}

\begin{definition}
Now, given a $\gg$-module $M$, let $\Omega_M: M \otimes V \to M \otimes V$ be the linear map defined as
 $$\Omega_M(m \otimes v) = 2\sum_i (-1)^{p(X_i) p(m)} X_im \otimes X^iv$$
 for homogeneous $m \in M$, $v \in V$.
\end{definition}

\begin{lemma}\label{lem:commute}
 The morphisms $\Omega_M$ define an endomorphism of the functor $\Theta'$.
\end{lemma}
\begin{proof}
 For any homogeneous $y \in \gg$, we have
$$[y \otimes 1 + 1 \otimes y,  X_i \otimes X^i] = [ y, X_i] \otimes X^i + (-1)^{p(X_i)p(y)} X_i \otimes [y, X^i].$$
By expanding $[y, X_i]$, $ [y, X^i]$ in the bases $\{X_j\}$ respectively  $\{X^j\}$ of $\mathfrak{gl}(n|n)$, we obtain:
$$ [y, X_i] = \sum_j \left\<X^j,[y, X_i] \right\> X_j, \;\;  [y, X^i] = \sum_j \left\< [y, X^i],X_j \right\> X^j.$$ 

Thus  the non-degeneracy and invariance of the trace form \eqref{dual} implies
\begin{eqnarray*}
&& [y \otimes 1 + 1 \otimes y, \sum_i  X_i \otimes X^i]\\
&=&  \sum_i  [ y, X_i] \otimes X^i + (-1)^{p(X_i)p(y)} X_i \otimes [y, X^i]\\
 & =&  \sum_{i, j} \left\<X^j,[y, X_i]\right\>  X_j \otimes X^i +\sum_{i, j} (-1)^{p(X_i)p(y)} \left\< [y, X^i],  X_j \right\> X_i \otimes  X^j \\
 & =&  \sum_{i, j} \left\<X^j,[y, X_i]\right\> X_j \otimes X^i -\sum_{i, j} \left\< [X^i,y],  X_j \right\> X_i \otimes  X^j\\
 &=&  \sum_{i, j} \left\<X^j,[y, X_i]\right\>  X_j \otimes X^i -\sum_{i, j} \left\< X^i,  [y,X_j] \right\> X_i \otimes  X^j \\
  &= &\sum_{i, j} \left\<X^j,[y, X_i]\right\> X_j \otimes X^i -\sum_{i, j} \left\< X^j,  [y,X_i] \right\> X_j \otimes  X^i \quad=\quad 0.
\end{eqnarray*}
This implies that the map $\Omega_M$ commutes with the action of $\gg$ on $M \otimes V$ for any $\gg$-module $M$, as required.
\end{proof}

\begin{definition}
Let $0\leq p<q\leq d$ be integers and $M$ a $\gg$-module. Define the linear maps $\Omega_{p\,q,M}:M\otimes V^{\otimes d}\to M\otimes V^{\otimes d}$ by 
$$\Omega_{p\,q}:=2\sum_i 1\otimes\dots\otimes X_i\otimes 1\otimes\dots\otimes X^i\otimes 1\otimes \dots\otimes 1,$$
where $X_i$ is applied to the $p$-th tensor factor and $X^i$ is applied to the $q$-th tensor factor (numbered from $0$ to $d$). They define an endomorphism of the endofunctor $_-\otimes V^{\otimes d}$ on the category of vector superspaces and we can consider, for $1\leq p\leq d$, the endomorphisms 
\begin{eqnarray}
\label{yse}
y_p &:= &\sum_{k=0}^{p-1} \Omega_{k\, p}, \quad\text{in particular}\quad {y_1}_V=s+e:V\otimes V\rightarrow V\otimes V,
\end{eqnarray}
where $s(x\otimes y)=(-1)^{p(x)p(y)}y\otimes x$ is the super swap and $e$ projects onto the unique trivial module by applying first $\beta$ and then the inclusion given in Lemma~\ref{lem:decoSE}.
\end{definition}
Using the decomposition $V\otimes V=S^2 V\oplus\Lambda^2V$ we have $e(\Lambda^2 V)=0$ and $e(S^2V)\subset\Lambda^2V$.
As a consequence we have $e\circ s=-s\circ e=e$. 
\begin{proposition}\label{prop:jm} 
The operators $y_1, y_2, \dots, y_{d}$ are pairwise commuting endomorphisms of the functor $_-\otimes V^{\otimes d}:\mathcal{F}_n\longrightarrow \mathcal{F}_n$. 
\end{proposition}
\begin{proof}
Let $\Delta: \mathcal{U}(\gg)\to \mathcal{U}(\gg)\otimes\mathcal{U}(\gg)$ denote the comultiplication, and $\Delta^q:\mathcal{U}(\mathfrak{g})\mapsto \mathcal{U}(\mathfrak{g})^{\otimes d}$ the iterated comultiplication.   Note that, with our dual bases $\{X_i\}$, $\{X^i\}$,
$$y_q=\sum_i\Delta^{q-1}(X_i)\otimes X^i\otimes 1\otimes\dots\otimes 1.$$
We have shown in Lemma~\ref{lem:commute} that for any $x\in\gg$ and
any $\gg$-module $M$ the operators $\Delta(x)$ and $\Omega$ commute in $M\otimes V$. Then it follows easily that $y_p $ is an endomorphism of $\gg$-modules for $1\leq 1\leq d$. Moreover, $\Delta^{q-1}(X_i)$ commutes with $\Omega_{a\,b}$ for $a,b\leq q$.
Now for $1\leq p\leq d$ with $p<q$ we get
$$[y_p,y_q]=\sum_{k=1}^{p}[\Omega_{k,p+1},\Delta^{q-1}(X_i)]\otimes X^i\otimes 1\otimes\dots\otimes 1=0.$$
The statement is proved.
\end{proof}

Note that $\Theta'$ is an exact functor; its left and right adjoint are isomorphic to $\Pi \Theta'$ where
$\Pi$ is the parity switch functor, since $V\cong \Pi V^*$ via \eqref{eta}.
\begin{definition}
For any $k\in\mathbb C$ we define a functor $\Theta'_k:\mathcal{F}_n\to\mathcal{F}_n$ as the functor $\Theta'=_-\otimes V$ followed by the projection onto the generalized $k$-eigenspace for $\Omega$, i.e., 
\begin{eqnarray}
\label{thetak}
\Theta'_k(M):=\bigcup_{m>0}\Ker (\Omega -k\operatorname{Id})^m_{|_{M\otimes V}}
\end{eqnarray}
and set $\Theta_k:=\Pi^k\Theta'_k$ in case $k\in\mathbb{Z}$. 
\end{definition}
\begin{lemma}\label{lem:exact}
 The functors $\Theta'_k$, $k\in\mathbb{C}$, and $\Theta_k$, $k\in\mathbb{Z}$ are exact.
\end{lemma}

\begin{proof}
This follows directly from \eqref{thetak} and the fact that $_-\otimes V $ is an exact functor.
\end{proof}

In fact all occurring eigenvalues for $\Omega$ are integral, that means we obtain the following. 

\begin{proposition}\label{prop:theta_integer}
If $k\notin\mathbb{Z}$, then $\Theta'_k=0$. Therefore, $\Theta'=\bigoplus_{k\in\mathbb{Z}} \Theta'_k$.
\end{proposition}
This will be proved in Section~\ref{proof:theta_integer}.

\subsection{Some useful properties of \texorpdfstring{$\Omega$}{Omega}}\label{ssec:Omega_prop}
It will be convenient to write $\Omega =\Omega_0+\Omega_1+\Omega_{-1}$ where
$$\Omega_0=2\sum_{\{i\mid X_i\in\gg_0\}} X_i\otimes X^i,\quad \Omega_1=2\sum_{\{i\mid X_i\in\gg_1\}} X_i\otimes X^i,\quad \Omega_{-1}=2\sum_{\{i\mid X_i\in\gg_{-1}\}} X_i\otimes X^i.$$
We denote by $C\in \mathcal{U}(\gg_0)$ and $\tilde{C}\in\mathcal{U}(\mathfrak{gl}(n|n))$ the respective Casimir elements.

\begin{lemma}\label{lem:usprop} 
\begin{enumerate}[1.)]
\item\label{itm:usprop1} If $M$ is a $\gg$-module, then 
\begin{eqnarray*}
\Omega_{\pm 1}(H^0(\gg_{-1},M)\otimes V_{\overline{1}})&=\;0\;=&\Omega_{\pm 1}(H^0(\gg_{1},M)\otimes V_{\overline{0}}).
\end{eqnarray*}
\item\label{itm:usprop2} For every $\gg_0$-module $M$, $m\in M$ and $v\in V$ homogeneous,  we have
$$\Omega_0(m\otimes v)=\tfrac{1}{2}(-1)^{p(v)}\left(C(m\otimes v)-C(m)\otimes v-m\otimes C(v)\right).$$
\item\label{itm:usprop3} Consider the element $2\sum_i X_i X^i$ in the universal enveloping algebra $\mathcal{U}(\mathfrak{gl}(n|n))$. Then
$$2\sum_i X_iX^i=\tilde C-\mathbf{1}_n,$$ where $\mathbf{1}_n\in\mathfrak{gl}(n|n)$ is the identity matrix. 
\item\label{itm:usprop4} The element $2\sum_i X_iX^i$ acts on the natural module $V$ as $-\id$.
\end{enumerate}
\end{lemma}
\begin{proof} Part \ref{itm:usprop1}.) is straightforward: any $X_i \in \gg_{-1}$ acts trivially on the $\gg_{-1}$-invariants $H^0(\gg_{-1},M)$ of $M$, so $\Omega_{-1}(H^0(\gg_{-1},M)\otimes V_{\overline{1}})=0$. On the other hand, for any $X_i \in \gg_{1}$, we have that $X^i$ is of the form $\left(\begin{smallmatrix}0&0\\C&0\end{smallmatrix}\right)$ for some $C: V_{\overline{0}} \to V_{\overline{1}}$, and thus $X^i$ acts trivially on $V_{\overline{1}}$. This implies $\Omega_{\pm 1}(H^0(\gg_{-1},M)\otimes V_{\overline{1}})=0$. Similarly for the second equality.
To prove \ref{itm:usprop2}.) we first define an involution on $\mathfrak{gl}(n|n)_{\overline{0}}$ by   
%\begin{eqnarray*}
$X=\left(\begin{smallmatrix}A&0\\0&D\end{smallmatrix}\right)
\mapsto\hat X:=\left(\begin{smallmatrix}A&0\\0&-D\end{smallmatrix}\right). $
%\end{eqnarray*}
Note that for homogeneous $v\in V$ we have
$\hat Xv=(-1)^{p(v)}Xv.$ For simplicity we assume that we have a basis $\{X_j\}_{j\in J}$ for $\gg_0$ induced from \eqref{ABCbasis1}. 
Consider the elements $X^j \in \mathfrak{gl}(n|n), j \in J$ given in Remark~\ref{dualbasis}. Then $\{2\hat X^j\}_{j\in J}$ is a dual basis for $\gg_0$, and $C = 2\sum_j X_j \hat X^j. $\\
Now, $2 X_j m\otimes \hat{X^j} v= 2 (-1)^{p(v)} X_j m \otimes X^j v.$ Thus 
\begin{eqnarray*}
C(m\otimes V)&=& Cm\otimes v+ m\otimes Cv+2\sum_{j\in J} X_j m\otimes 2\hat X^j v\\
&=&Cm\otimes v+m\otimes Cv+4(-1)^{p(v)} \sum_{X_j\in\gg_0} X_j m\otimes X^j v \\
&= &C m\otimes v+m\otimes Cv+2(-1)^{p(v)} \Omega_0(m\otimes v)
\end{eqnarray*}
This proves \ref{itm:usprop2}.).

To prove \ref{itm:usprop3}.) recall that by definition
$\tilde C=\sum_iX_iX^i+(-1)^{p(X_i)}X^iX_i$. Using the relation
$X_iX^i=(-1)^{p(X_i)}X^iX_i+[X_i,X^i]$
we obtain
$$2\sum_l X_i X^i=\tilde C+\sum_i [X_i,X^i].$$
Note that $\sum_i [X_i,X^i]\subset\gg'$ and, moreover, $[x,\sum_i [X_i,X^i]]=0$ for every $x\in\gg$. 
Since $\gg'$ is the coadjoint $\gg$-module, every $\gg$-invariant 
vector in $\gg'$ is
proportional to $\mathbf{1}_n \in \gg' \subset \mathfrak{gl}(n|n)$. Hence $\sum_i[X_i,X^i]=t \mathbf{1}_n$ for some $t \in \mathbb C$. It remains to find $t$. For this we use the invariant supertrace form \eqref{dual} and the grading element $h$ from \eqref{theh}:
\begin{eqnarray*}
\left\<h,\sum_i [X_i,X^i]\right\>&=&\sum_i\left\<[h,X_i],X^i\right\>.
\end{eqnarray*}
For $X_i\in\gg_0$ we have $[h,X_i]=0$ and for $X_i\in\gg_{\pm 1}$ we have $[h,X_i]=\pm X_i$. Thus, we obtain
\begin{eqnarray*}
\left\<h,\sum_i [X_i,X^i]\right\>\;=\;-\dim\gg_1+\dim\gg_{-1}\;=\;-n.
\end{eqnarray*}
Since $\left\<h,\mathbf{1}_n \right\>=n$, we get $t=-1$ and hence proved \ref{itm:usprop3}.).

Finally, to show \ref{itm:usprop4}.), we just recall that by definition $\tilde C$ acts by zero on $V$.
\end{proof}

\subsection{The action of \texorpdfstring{$\Theta_i$}{translation functors} on Kac modules}\label{ssec:acKac}
We would like to decompose $\nabla(\lambda)\otimes V$ and $\Delta(\lambda)\otimes V$ into a direct sum of generalized eigenspaces with respect to
$\Omega$, and determine the occurring eigenvalues. Throughout this section, we set $\Delta(\mu ) =\nabla (\mu) = 0$ whenever $\mu$ is not a dominant weight.

\begin{lemma}\label{lem:filtration} The tensor products $\nabla(\lambda)\otimes V$ and $\Delta(\lambda)\otimes V$ have filtrations 
\begin{eqnarray}\label{eq:filtr}
\{0 \}&=&N_0 \subset N_1 \subset N_2 \subset \ldots \subset N_{2n} = \nabla(\lambda)\otimes V\nonumber,\\
\{0\} &=&M_0 \subset M_1 \subset M_2 \subset\ldots\subset M_{2n} = \Delta(\lambda)\otimes V,
\end{eqnarray}

where 
$$N_i / N_{i-1} \cong \begin{cases}
\nabla(\lambda+\varepsilon_{n-i+1}), &\text{ if } i\leq n, \\
                            \Pi \nabla(\lambda-\varepsilon_{i-n}), &\text{ if } i>n,
\end{cases} \;\;\; \text{ and } M_i / M_{i-1} \cong \begin{cases}
\Pi \Delta(\lambda-\varepsilon_{n-i+1}), &\text{ if } i\leq n, \\
                             \Delta(\lambda+\varepsilon_{i-n}), &\text{ if } i>n.
\end{cases}
$$ 
\end{lemma}
\begin{proof} We will deal with the case of $\Delta(\lambda)\otimes V$, the other case is similar. We use 
$$\Delta(\lambda)\otimes V=\operatorname{Ind}^\gg_{\gg_0\oplus\gg_{-1}}(V(\lambda)\otimes V)$$
and the exact sequence of $\gg_0\oplus\gg_{-1}$ modules
$$0\to V(\lambda)\otimes V_{\overline{1}}\to  V(\lambda)\otimes V\to  V(\lambda)\otimes V_{\overline{0}}\to 0.$$
On the other hand, recall the well-known identities for $\gg_0=\mathfrak{gl}(n)$-modules
\begin{eqnarray}\label{eq:translation_kac}
  V(\lambda)\otimes V_{\overline{1}}=\bigoplus_{i=1}^n \Pi V(\lambda-\varepsilon_i)&\text{and}&
  V(\lambda)\otimes V_{\overline{0}}=\bigoplus_{i=1}^n V(\lambda+\varepsilon_i),\quad
\end{eqnarray}
where we replace $V(\lambda\pm\varepsilon_i)$ by $0$ if $\lambda\pm\varepsilon_i$ is not dominant.
Thus, applying the exact induction functor $\Ind$ to \eqref{eq:translation_kac}, we obtain the exact sequence
$$0\to \bigoplus_{i=1}^n \Pi\Delta (\lambda-\varepsilon_i)\to \Delta (\lambda)\otimes V \to \bigoplus_{i=1}^n \Delta(\lambda+\varepsilon_i)\to 0,$$
which implies the statement.
\end{proof}

Next we notice that for each term of the filtration \eqref{eq:filtr} we have
$$H_0(\gg_{-1},\nabla(\lambda+\varepsilon_i))\subset H^0(\gg_{1},\nabla(\lambda+\varepsilon_i))\subset V(\lambda-\gamma)\otimes V_{\overline{0}},$$
$$H^0(\gg_{-1},\nabla(\lambda-\varepsilon_i))\subset V(\lambda)\otimes V_{\overline{1}},$$
$$H_0(\gg_{1},\Delta(\lambda-\varepsilon_i))\subset H^0(\gg_{-1},\Delta(\lambda-\varepsilon_i))\subset V(\lambda)\otimes V_{\overline{1}},$$
$$H^0(\gg_{1},\nabla(\lambda+\varepsilon_i))\subset V(\lambda+\tilde\gamma)\otimes V_{\overline{1}}.$$

In order to calculate the eigenvalue of $\Omega$ on each term of the filtration it suffices to calculate the eigenvalue of $\Omega_0$ on
$$H_0(\gg_{-1},\nabla(\lambda+\varepsilon_i)),\,H^0(\gg_{-1},\nabla(\lambda-\varepsilon_i)),\,
H_0(\gg_{1},\Delta(\lambda-\varepsilon_i)),\, H^0(\gg_{1},\Delta(\lambda+\varepsilon_i)).$$
Now we use Lemma~\ref{lem:usprop} \eqref{itm:usprop1}, \eqref{itm:usprop2}.  
Recall that the eigenvalue of $C$ on $V(\mu)$ is $(\mu+2\rho,\mu)$.

Thus, the eigenvalue of $\Omega$ on $\nabla(\lambda+\varepsilon_i)\cap \left( V(\lambda-\gamma)\otimes V_{\overline{0}}\right)$ is
\begin{equation*}
\tfrac{1}{2}\left((\lambda-\gamma+\varepsilon_i+2\rho,\lambda-\gamma+\varepsilon_i)-(\lambda-\gamma+2\rho,\lambda-\gamma)-(\varepsilon_1+2\rho,\varepsilon_1)\right),
\end{equation*}
which is equal to
\begin{equation*}
{\tfrac{1}{2}} \left(2(\varepsilon_i,\lambda-\gamma)+2(\rho,\varepsilon_i-\varepsilon_1)+(\varepsilon_i,\varepsilon_i)-(\varepsilon_1,\varepsilon_1) \right)=
\lambda_i+n-i=\bar{\lambda}_i.
\end{equation*}

Similarly, the eigenvalue of $\Omega$ on $\nabla(\lambda-\varepsilon_j)\cap  \left(V(\lambda)\otimes V_{\overline{1}}\right)$ is
\begin{equation*}
\tfrac{1}{2}\left((\lambda+2\rho,\lambda)+(-\varepsilon_n+2\rho,-\varepsilon_n)-(\lambda-\varepsilon_j+2\rho,\lambda-\varepsilon_j)\right)=
\bar{\lambda}_j.
\end{equation*}

Altogether we obtain:
\begin{lemma}\label{lem:small} Let $\lambda$ be a dominant weight. Fix $k \in \mathbb Z$. Then there is at most one $1\leq j\leq n$ such that $\bar{\lambda}_j=k$.
\begin{enumerate}[i.)]
\item If $\bar{\lambda}_j\neq k$ for all $1\leq j\leq n$, then $\Theta'_k(\nabla(\lambda))=0$.
\item If $\bar{\lambda}_j=k$, then $\Theta'_k(\nabla(\lambda))$ can be described by the exact sequence
$$0\to \nabla(\lambda+\varepsilon_j)\to \Theta'_k(\nabla(\lambda))\to  \Pi\nabla(\lambda-\varepsilon_j)\to 0.$$
\end{enumerate}
\end{lemma}

For the thick Kac modules the picture is more interesting. Namely, the element $\Omega$ acts on $\Delta(\lambda-\varepsilon_j)\cap  \left(V(\lambda)\otimes V_{\overline{1}}\right)$ as multiplication by the scalar
\begin{eqnarray}
\label{formula1}
{\tfrac{1}{2}}\left((\lambda+2\rho,\lambda)+(-\varepsilon_n+2\rho,-\varepsilon_n)-(\lambda-\varepsilon_j+2\rho,\lambda-\varepsilon_j)\right)\;=\;\bar{\lambda}_j,
\end{eqnarray}
and on  $\Delta(\lambda+\varepsilon_i)\cap  \left(V(\lambda+\tilde\gamma)\otimes V_{\overline{0}}\right)$ by the scalar
\begin{eqnarray}
\label{formula2}
&&{\tfrac{1}{2}}\left((\lambda+\tilde\gamma+\varepsilon_i+2\rho,\lambda+\tilde\gamma+\varepsilon_i)-(\lambda+\tilde\gamma+2\rho,\lambda+\tilde\gamma)-
(\varepsilon_1+2\rho,\varepsilon_1)\right)\nonumber\\
&=&\lambda_i+2+n-i\;=\;\bar{\lambda}_i+2.
\end{eqnarray}
\begin{lemma}
 The subquotients $\Delta(\lambda - \varepsilon_i)$ and $\Delta(\lambda + \varepsilon_j)$ of $\Delta(\lambda) \otimes V$ have the same $\Omega$-eigenvalue if and only if $j=i+1$ and $\lambda_i=\lambda_{i+1}+1$.
\end{lemma}
\begin{proof} By \eqref{formula1} and \eqref{formula2} we must have $\overline{\lambda}_i=\overline{\lambda}_{i+1}+2$ which is equivalent to $\lambda_i=\lambda_{i+1}+1$.
\end{proof}

\begin{lemma}\label{lem:big} Let $\lambda$ be a dominant weight. Fix $k \in \mathbb Z$.
\begin{enumerate}[i.)]
\item If $\bar{\lambda}_j\neq k,k-2$ for all $j\leq n$, then $\Theta'_k(\Delta(\lambda))=0$.
\item If $\bar{\lambda}_j=k$ and $\lambda_{j+1}\neq \lambda_j-1$, then $\Theta'_k(\Delta(\lambda))=\Delta(\lambda-\varepsilon_j)$.
\item If $\bar{\lambda}_j=k-2$ and $\lambda_{j-1}\neq \lambda_j+1$, then $\Theta'_k(\Delta(\lambda))=\Delta(\lambda+\varepsilon_j)$.
\item If $\bar{\lambda}_j=k$ and $\lambda_{j+1}=\lambda_j-1$, then $\Theta'_k(\Delta(\lambda))$ fits into an exact sequence
$$0\to \Pi\Delta(\lambda-\varepsilon_j)\to \Theta'_k(\Delta(\lambda))\to  \Delta(\lambda+\varepsilon_{j+1})\to 0.$$
\end{enumerate}
\end{lemma}

We finish this section with the proof of the statement $\Theta'_k = 0$ for $k \not\in \mathbb{Z}$:
\begin{proof}[Proof of Proposition \ref{prop:theta_integer}]\label{proof:theta_integer}
 By Lemma \ref{lem:big}, $\Theta'_k(\Delta(\lambda)) =0$ for any $k \not\in \mathbb{Z}$ and $\lambda \in \Lambda_n$. Since $\Theta'_k$ is exact, this implies that $\Theta'_k(L(\lambda)) =0$ for any $k \not\in \mathbb{Z}$, i.e. $\Theta'_k$ is zero on all simple modules. Using exactness of $\Theta'_k$ once again, we conclude that $\Theta'_k = 0$ for $k \not\in \mathbb{Z}$.
\end{proof}

\subsection{Adjunctions and Temperley-Lieb algebra relations}\label{ssec:act}
We establish now basic properties of the translation functors. Thanks to Proposition~\ref{prop:theta_integer} we can assume from now on that $k \in \mathbb{Z}$ whenever we consider $\Theta'_k $ or $\Theta_k$.
\begin{proposition}
\label{itm:TL1} The functor $\Theta_k$ is left adjoint to $ \Theta_{k-1}$ and right adjoint to $\Theta_{k+1}$.
\end{proposition}

Before we prove the proposition, we use our previous results to check dimension formulas which would be implied by the adjunctions.
 \begin{example}
 Consider a thick Kac module $\Delta(\lambda)$, and indices $k, j$ such that $\bar{\lambda}_j=k$ and $\lambda_{j+1}=\lambda_j-1$. We established in Lemma~\ref{lem:big} an exact sequence
$$0\to \Pi \Delta(\lambda-\varepsilon_j)\to \Theta'_k(\Delta(\lambda))\to  \Delta(\lambda+\varepsilon_{j+1})\to 0,$$ and similarly, 
there is an injective map $\Pi \Delta(\lambda)\to  \Theta'_{k-1} \Delta(\lambda+\varepsilon_{j+1})$. Hence
$$\dim \Hom_{\mathfrak{g}} (\Theta'_k(\Delta(\lambda)), \Delta(\lambda+\varepsilon_{j+1})) = 1 = \dim \Hom_{\mathfrak{g}} (\Pi \Delta(\lambda), \Theta'_{k-1}\Delta(\lambda+\varepsilon_{j+1}))$$ as predicted by Proposition~ \ref{itm:TL1}.
 \end{example}

For the proof of Proposition~ \ref{itm:TL1} we first discuss the left and right adjoint functors to the functor $\Theta' = _-\otimes V$. The right adjoint is the functor $_-\otimes V^*$. It is isomorphic to the left adjoint functor, although the isomorphism is non-trivial. Indeed, the left adjoint is given by tensoring with the {\it right dual} ${}^*V$; in our setting, there is an isomorphism $\eta:{}^*V\cong V^*$, but the isomorphism is non-trivial, see e.g. \cite[XIV.2.2, XIV]{Kassel} for more details.

Now consider the natural transformation $\Omega$ of $\Theta'$ from Section~\ref{ssec:Omega}. It induces transformations $\Omega^t:\,_-\otimes V^* \to\, _-\otimes V^*$ and ${}^t \Omega:\,_-\otimes{}^*V\to\,_-\otimes{}^*V$ of the functors adjoint to $\Theta'$ on either side; for instance the natural transformation $\Omega^t$ is defined for $M \in \mathcal{F}_n$ as the following composition (where $\op{ev}$ and $\op{coev}$ denote the usual evaluation $f\otimes v \mapsto f(v)$ respectively coevaluation $1 \mapsto \sum v_i\otimes f_i$ maps, where $f_i$ is given by $f_i(v_j)=\delta_{i\,j}$),
\begin{equation}
\label{syl1}
M \otimes V^* \xrightarrow{\id \otimes \op{coev}} M \otimes  V^* \otimes V \otimes V^* \xrightarrow{\Omega_{M \otimes V^*} \otimes \id} M \otimes  V^* \otimes V \otimes V^* \xrightarrow{\id \otimes \op{ev} \otimes \id} M \otimes  V^*.
\end{equation}

Hence the right (resp. left) adjoint functor to $\Theta'_k$ is the direct summand of $_-\otimes V^*$ (resp. $_-\otimes {}^*V$) corresponding to the generalized eigenvalue $k$ of $\Omega^t$ (resp. of ${}^t \Omega$). 

\begin{remark}
 Notice however that the isomorphism $\eta$ between the left and right adjoint functor to $_-\otimes V$ does not need to  intertwine $\Omega^t$ and ${}^t \Omega$. Hence the left and right adjoint functors to $\Theta'_k$ would not necessarily be isomorphic, and in fact they are not!
\end{remark}

\begin{proof}[Proof of Proposition~ \ref{itm:TL1}]
Let us evaluate $\Omega^t_M$ on $m \otimes f$, where $m\in M$, $f \in V^*$ are 
homogeneous. We calculate the first two maps in the composition 
\eqref{syl1}:%\todo{correct?}
\begin{eqnarray}
\label{twomaps}
m \otimes f&\longmapsto&\sum_j m \otimes f  \otimes v_j\otimes v^j
\;\longmapsto\; 2\sum_{i, j} (-1)^{ p(m \otimes f)p(X^i)}  X_i( m \otimes f) \otimes X^i v_j \otimes v^j,\quad\quad
\end{eqnarray}
for $\{ v_j\}$ a homogeneous basis of $V$ with dual basis $\{v^j\}$ of $V^*$.  Then the result \eqref{twomaps} equals
\small
\begin{equation*}
2\sum_{i, j} \left((-1)^{ p(m \otimes f)p(X_i)} X_im \otimes f \otimes X^iv_j 
\otimes v^j+ (-1)^{p(X_i) p(m \otimes f)+p(m) p(X_i)} m \otimes X_if \otimes 
X^iv_j \otimes v^j\right).
\end{equation*}
\normalsize
Applying finally $\id_M \otimes \op{ev} \otimes \id_{V^*}$  we obtain 
\small 
\begin{eqnarray*} 
2\sum_{i, j}  (-1)^{p(X_i) p(m \otimes f)} X_i m \otimes f( X^i v_j) v^j +  2\sum_{i,  j}   m \otimes (-1)^{p(X_i) p(m \otimes f)+p(m) p(X_i)}  \left( X_if\right) (X^iv_j)  v^j.
\end{eqnarray*}
Hence altogether we obtain 
\begin{eqnarray}
\Omega^t_M(m \otimes f)
&=&  -2\sum_{i, j} (-1)^{p(X_i) p(m \otimes f) + p(X^i) p(f)} X_im \otimes (X^if)(v_j) v^j\\
&&\quad\quad+\;2\sum_{i,  j}   m \otimes (-1)^{p(X_i) p(m \otimes f)+p(m) p(X_i)}   \left( X_if \right) (X^iv_j) v^j\nonumber\\
&=& -2\sum_{i} (-1)^{p(X_i) p(m)} X_im \otimes X^if -  2\sum_{ j}   m \otimes  f \left( \sum_i X_i X^i v_j \right) v^j\nonumber\\
&=&-2\left( \sum_{i}X_i \otimes X^i \right)m \otimes f +  m \otimes  f,\label{omegat}
\end{eqnarray}
where for the last equality we used  Lemma~\ref{lem:usprop}\ref{itm:usprop4}.).

On the other hand, recall the isomorphism $\eta: V^* \cong V \otimes \Pi \mathbb{C}$ from \eqref{eta}. 
It is easy to see that the elements $X^i$ of $ \mathfrak{gl}(n|n)$ satisfy  (note the sign appearing): 
$$ \xymatrix{ &{ V^*} \ar[d]_{\eta}  \ar[r]^{X^i} &{V^*} \ar[d]^{\eta} \\  &{ V \otimes \Pi \mathbb{C}} \ar[r]_{-X^i}  &{V \otimes \Pi \mathbb{C}} }.$$

Formula~\eqref{omegat} and the definition of $\Omega$ imply that the following diagram commutes 
$$\xymatrix{ &{M \otimes V^*} \ar[d]_{\id \otimes \eta}  \ar[rrr]^{{\Omega}^t_M} &&&{M \otimes V^*} \ar[d]^{\id \otimes \eta} \\  &{M \otimes V \otimes \Pi \mathbb{C}} \ar[rrr]^{\Omega_M \otimes \id_{\Pi \mathbb{C}} +\id} &&&{M \otimes V \otimes \Pi \mathbb{C}}}$$
Thus $\bar{\Omega}^t_M= \Omega_M \otimes \id_{\Pi \mathbb{C}} +\id$ and therefore the eigenspace corresponding to the eigenvalue $k$ of ${\Omega}^t_M$ coincides with the eigenspace corresponding to the eigenvalue $k-1$ of $\Omega_M$. Similarly one shows that the eigenspace corresponding to the eigenvalue $k+1$ of ${}^t{\Omega}_M$ coincides with the eigenspace corresponding to the eigenvalue $k$ of $\Omega_M$. Hence the proposition follows. 
\end{proof}

Since the functors $\Theta'_k$, for $k\in\mathbb{Z}$, are exact by Lemma~\ref{lem:exact}, they induce $\mathbb Z$-linear operators on the Grothendieck group $\Groth_n$, which we denote by $\theta'_k$. 
Note that the subgroups $\Groth_n(\Delta)$, $\Groth_n(\nabla)$ and $\Groth^\oplus(\mathcal P_n)$ are $\theta'_k$-stable by the Lemmas~\ref{lem:small}, \ref{lem:big} and \ref{lem:tilt}. 

We denote  by $\mathbb{U}=\mathbb C^\mathbb{Z}$ the vector space with fixed basis $\{u_i\,|\,i\in\mathbb Z\}$ indexed by $\mathbb{Z}$. Consider the Lie algebra $\mathfrak{sl}(\infty)$ of all linear operators in $\mathbb C^\mathbb{Z}$ of finite rank. This Lie algebra has
Chevalley generators $e_i=E_{i-1,i},\, f_i=E_{i,i-1}$ for $i \in \mathbb Z$ subject to the defining Serre relations of the $A_{\infty}$-Dynkin diagram. Let $\Lambda^n(\mathbb{U})$ be the $\mathbb{Z}$ lattice in the $n$th exterior vector space spanned by the standard basis of wedges in the $u_i$. 

We define the $\mathbb Z$-linear maps $\Phi$ and $\Phi^\vee$ as follows:
\begin{eqnarray*}
&\Phi:\;\Groth_n(\Delta)\to \Lambda^n(\mathbb{U})\quad\text{and}\quad\Phi^\vee:\;\Groth_n(\nabla)\to\Lambda^n(\mathbb {U})&\\
&\Phi[\Delta(\lambda)]=\Phi^\vee[\nabla(\lambda)]=u_{\bar\lambda_1}\wedge\ldots\wedge u_{\bar\lambda_n}.&
\end{eqnarray*}
Note that $\Phi$ and $\Phi^\vee$  are in fact isomorphisms of abelian groups.

\begin{theorem}\label{thrm:TL_action_Kac} For any $k \in \mathbb Z$, we have the following equalities:
\begin{eqnarray}
\label{PhiPhicheck}
\Phi\circ \theta'_k=(e_{k}+f_{k-1})\circ\Phi,\quad \text{and}\quad\Phi^\vee\circ \theta'_k =(e_{k}+f_{k+1})\circ\Phi^\vee.
\end{eqnarray}
\end{theorem}
\begin{proof} This follows immediately from Lemma~\ref{lem:small} and Lemma~\ref{lem:big}.
\end{proof}

\begin{remark} With the identifications  $\Phi$ and $\Phi^\vee$ of $\Groth_n(\Delta)$ respectively $\Groth_n(\nabla)$ with $\Lambda^n(\mathbb U)$ we obtain from \eqref{PhiPhicheck} directly
\begin{eqnarray*}
\langle [\Theta'_k M],[N]\rangle&=&\langle[M],[\Theta'_{k-1} N]\rangle
\end{eqnarray*}
where $\langle \cdot, \cdot \rangle$ is the pairing defined in Section~\ref{ssec:reduced_GR_group}.
This corresponds to the fact that $\Theta'_{k-1}$ is right adjoint to $\Pi \Theta'_k$ by Proposition~ \ref{itm:TL1} below.
\end{remark}

The $\theta'_k$ satisfy the defining relations for the Temperley-Lieb algebra of type $A_\infty$:
\begin{corollary}\label{cor:TL_action_Gr}
The operators $\theta'_k$ on $\Groth_n$ satisfy the following relations for any $k, j \in \mathbb Z$, $|k-j| >1$:
  \begin{equation}\label{eq:TL}
   {\theta'_k}^2 = 0, \;\;
  \theta'_k \theta'_j = \theta'_j \theta'_k, \;\;
   \theta'_k \theta'_{k \pm 1} \theta'_k = \theta'_k.
  \end{equation}

\end{corollary}
\begin{remark}
 The relations \eqref{eq:TL} are the defining relations for the Temperley-Lieb algebra $\op{TL}_{\infty}(q + q^{-1})$ on infinitely many generators. For a definition of the Temperley-Lieb algebra, see for instance \cite[Section~ 4]{Str}. In contrast to the situation described therein, in our case the parameter $q = \pm i$ is a primitive fourth root of unity.
\end{remark}
\begin{proof} The relations from Theorem~\ref{thrm:TL_action_Kac} still hold on $\Groth_n(\Delta)$ respectively $\Groth_n(\nabla)$ and thus also on $\Groth^\oplus(\mathcal P_n)$.
Using the natural pairing $\Groth^\oplus(\mathcal P_n)\times \Groth_n\to\mathbb Z$ induced from \eqref{bil} and Proposition~\eqref{itm:TL1}
we obtain the same relations on $\Groth_n$.
\end{proof}

\subsection{Categorical action}\label{ssec:categorical_action}
In fact, we will show that the Temperley-Lieb relations \eqref{eq:TL} hold even in a categorical version:
\begin{theorem}\label{thrm:TL_action_categor}
\label{itm:TL2} 
There exists natural isomorphisms of functors (for $k,j\in\mathbb{Z}$)
  \begin{eqnarray}
   \Theta_k^2 &\cong& 0,\label{first}\\  
   \psi_{k, k \pm 1}=\InnaB{{\Theta_k}}\op{adj}:\quad \Theta_k \Theta_{k \pm 1} 
\Theta_k &\stackrel{\cong}\Longrightarrow&   \Theta_k,\label{second}\\
   \psi_{k,j}=(j-i)1\otimes s+\id: \quad\quad\quad\Theta_k \Theta_j &\stackrel{\cong}\Longrightarrow&  \Theta_j \Theta_k \; \text{   if   } \; |k-j| >1,
  \end{eqnarray}
  where $\op{adj}$ denotes the respective adjunction morphism given by Proposition~\ref{itm:TL1}.\label{third}
 \end{theorem}

\begin{proof}
The first isomorphism \eqref{first} follows directly from Corollary~\ref{cor:TL_action_Gr}. 

By Proposition~\ref{itm:TL1}, there exist adjunction morphisms $ 
\overline{\op{adj}}:\Theta_{k + 1}  \Theta_k \Rightarrow \id $ (that is the counit of the adjunction) and $\op{adj}:\id \Rightarrow \Theta_{k - 1} \Theta_k$ 
(which is the unit of the adjunction). The corresponding natural transformation 
${\Theta_k}\op{adj}$ (respectively $\Theta_k\overline{\op{adj}}$) is given by injective respectively 
surjective maps by the adjunction axioms. Hence by Corollary~\ref{cor:TL_action_Gr} 
these are isomorphisms; thus \eqref{second} holds.

To prove the existence of the isomorphisms $\psi_{k,j}$, let $M\in\mathcal{F}_n $ and define the linear endomorphisms $\Omega_{1\,2}, \Omega_{1\,3}, \Omega_{2\,3}$ of $M\otimes V\otimes V$ by 
$$\Omega_{1\,2}=2\sum_i X_i\otimes X^i\otimes 1,\quad \Omega_{1\,3}=2\sum_i X_i\otimes 1\otimes X^i, \quad \Omega_{2\,3}=2 \sum_i 1\otimes X_i\otimes X^i.$$
By definition, $\Theta'_k\Theta'_j M$ is a submodule in $M\otimes V\otimes V$ which is a generalized eigenspace of $\Omega_{1\,2}$ with eigenvalue $j$ and
of $\Omega_{1\,3}+\Omega_{2\,3}$ with eigenvalue $k$. 

Recall that $\Omega_V=s+e$, see \eqref{yse} and  Lemma~\ref{lem:decoSE}.
Consider the operator $1\otimes e: M\otimes V\otimes V\to M\otimes V\otimes V$. By the proof of Proposition~\ref{itm:TL1} we obtain
$\Theta'_i \Theta'_j M \subset \Ker (1\otimes e)$ unless $i-1=j$: indeed, $1\otimes e: \Theta'_i \Theta'_j M \to M\otimes V\otimes V$ factors 
through $1\otimes e: \Theta'_i \Theta'_j M \to \Pi M$, and by Proposition~\ref{itm:TL1}, this gives a map $\Theta'_j M \to \Theta'_{i-1} M$; such a map is non-zero only if $i-1 = j$. Consider the natural transformation $$\psi_{ij}:=(j-i)1\otimes s+\id:\Theta'_i\Theta'_j 
\Longrightarrow - \otimes V\otimes V.$$
We claim that the image of $\psi_{ij}$ lies in $\Theta'_j \Theta'_i M$, inducing a natural transformation $$\psi_{ij}: \Theta'_i \Theta'_j \Longrightarrow  \Theta'_j \Theta'_i$$ and thus a 
natural transformation $$\psi_{ij}: \Theta_i \Theta_j \Longrightarrow  \Theta_j 
\Theta_i  $$ as well. 

%Indeed they are nonzero when applied to thin Kac modules \Cath{and 
% therefore isomorphisms by Proposition~\ref{smallKac}.}
We use the relations:
\begin{eqnarray*}
\Omega_{1\,2} \circ (1 \otimes s)  &=& (1 \otimes s) \circ(\Omega_{1\,3} + \Omega_{2\,3}) - \id - 1\otimes e,\\
(\Omega_{1\,3} + \Omega_{2\,3})\circ (1 \otimes s)&=&(1 \otimes s) \circ \Omega_{1\,2}+\id+1\otimes e,
\end{eqnarray*}
to obtain the identities
\begin{eqnarray*}
\Omega_{1\,2} \psi_{ij} &=&(j-i) (1 \otimes s) \circ(\Omega_{1\,3} + \Omega_{2\,3}) + \Omega_{1\,2} - ( j-i)\id+(j-i)(1\otimes e),\\
(\Omega_{1\,3} + \Omega_{2\,3})\psi_{ij} &=& ( j-i) (1 \otimes s)\circ \Omega_{1\,2}+\Omega_{1\,3}+\Omega_{2\,3}+( j-i)\id+(j-i)(1\otimes e).
\end{eqnarray*}
We rewrite the above relations in the form
\begin{eqnarray*}
(\Omega_{1\,2}-i\id) \psi_{ij} =\psi_{ij}(\Omega_{1\,3} + \Omega_{2\,3}-i\id ) + (\Omega_{1\,2} -j\id)-(\Omega_{1\,3} + \Omega_{2\,3}-i\id )+(j-i)(1\otimes e),\\
(\Omega_{1\,3} + \Omega_{2\,3}-j\id)\psi_{ij} =\psi_{ij}(\Omega_{1\,2} -j\id )-(\Omega_{1\,2} -j\id)+(\Omega_{1\,3} + \Omega_{2\,3}-i\id )+(j-i)(1\otimes e).
\end{eqnarray*}
Since the operators $\Omega_{1\,2} - j \id $ and $\Omega_{1\,3} + \Omega_{2\,3} - i\id$ commute and are nilpotent on $\Theta'_i \Theta'_j M$ and 
$(1\otimes e)|_{\Theta'_i \Theta'_j M}=0$, we conclude that $\Omega_{1,2} -i\id $ and and $\Omega_{1\,3} + \Omega_{2\,3} - j\id$ are nilpotent on $\psi_{ij}\Theta'_i \Theta'_j M$.
Hence the image $\psi_{ij}:\Theta'_i \Theta'_j M$ lies in $\Theta'_j \Theta'_i 
M$.

For $i \neq j \pm 1$, an easy computation shows that the 
map $$(j-i)1\otimes s+\id:M \otimes V\otimes V
\Longrightarrow M \otimes V\otimes V$$ is injective for any $M$. 
Indeed, consider the basis $\{v_i\}_{i \in I}$ of $V$ (see 
Remark~\ref{rmk:basis}), where $I = \{1, \ldots, n\} \sqcup \{1', \ldots, 
n'\}$. The element $\sum_{k, l \in I} m_{k, l} \otimes v_k \otimes v_l \in M 
\otimes V \otimes V$ is sent to $$\sum_{k, l \in I} ((-1)^{p(k)p(l)} (i-j) 
m_{l, k} + m_{k, l}) \otimes v_k \otimes v_l.$$ If this is zero, then for every 
$k, l$ we have $-(-1)^{p(k)p(l)} (i-j) m_{l, k}= m_{k, l}$, and hence 

$(i-j)^2 m_{k, l}= m_{k, l}$ for every $k, l$. If $i-j \neq \pm 1$, we conclude that 
$m_{k, l} =0$ for every $k, l$.

\InnaB{Hence 
$\psi_{ij}$ is an injective natural transformation, and therefore 
an isomorphism by Corollary~\ref{cor:TL_action_Gr}.}
\end{proof}

\section{Combinatorics: Weight diagrams, translations and duality}\label{sec:weight_diagrams}
\subsection{From weights to weight diagrams}\label{ssec:weight_diagrams_def}
For $\lambda$ a dominant weight we define the map 
\begin{eqnarray*}
f_\lambda:\mathbb{Z} \to \{0,1 \}\quad\text{as}\quad
f_\lambda(i)&=&
\begin{cases}1& \text{if $i\in c_{\lambda}$},\\
0& \text{if $i\notin c_{\lambda}$}.
\end{cases}
\end{eqnarray*}
where $c_{\lambda}:=\{\bar{\lambda_i}\,|\,i=1, \ldots, n\}$ is as in \eqref{omegarho}. The corresponding  {\it weight diagram} $d_\lambda$ is the  labeling of the integer line by symbols $\bullet$ (``black ball'') and $\circ$ (``empty'') such that $i$ has label $\bullet$ if $f(i) =1$, and label $\circ$ otherwise. 
\begin{example}
Let $n=4$. Then for $\lambda = 0$, the weight diagram is 
  $$ \xymatrix{  \ldots &\underset{-1}{\circ}  &\underset{0}{\bullet} &\underset{1}{\bullet}  &\underset{2}{\bullet}  &\underset{3}{\bullet}  &\underset{4}{\circ}  &\underset{5}{\circ} &\underset{6}{\circ}  &\underset{7}{\circ} &\ldots} $$
  whereas for  $\lambda = \rho$, the weight diagram is 
  $$\ldots \xymatrix{ &\underset{-1}{\circ}  &\underset{0}{\bullet} &\underset{1}{\circ}  &\underset{2}{\bullet}  &\underset{3}{\circ}  &\underset{4}{\bullet}  &\underset{5}{\circ} &\underset{6}{\bullet}  &\underset{7}{\circ} &\ldots} $$
  (in both diagrams all remaining positions are labeled by $\circ$).
\end{example}

\begin{remark}
The following properties are easy to verify.
\begin{enumerate}[1.)]
\item Typical weights (i.e. when $P(\lambda) = \Delta(\lambda)$) correspond precisely to the weight diagrams without two neighboring black balls.
\item There are two possible partial orders on the weights, corresponding to the choice of either thick Kac modules or thin Kac modules as the standard objects 
in $\mathcal{F}_n$, as mentioned in Remark~\ref{rmk:2_highest_weight_struct}. In both orders, if $\lambda \leq \mu$ then $\lambda_i \geq \mu_i$. In terms of diagrams, this means that the $i$-th black ball in $d_{\lambda}$ 
(counted from left) lies further to the right of the $i$-th black ball of $d_{\mu}$.
\end{enumerate}
\end{remark}
Obviously this induces a bijection between the set of dominant weights of $\gg=\mathfrak{p}(n)$, the set of  maps $f:\mathbb{Z} \to \{0,1 \}$ such that $\sum_i f(i) =n$, and the set of weight diagrams with exactly $n$ black balls.

\subsection{Translation functors in terms of weight diagrams}\label{ssec:translation_fun_diag} 
We have the following description for the action of $\Theta'_i$ on the thick and thin Kac modules.
By convention, any appearing diagram which is not defined, e.g., due to lack of black balls to be moved, corresponds to the zero module. 

Translation of thick Kac modules corresponds to moving black balls to position $k-1$, whereas translation of thin Kac modules corresponds to moving black balls away from $k$:
\begin{proposition}[Translation of thick Kac modules] \label{bigKac} Let $k\in\mathbb{Z}$. Then 
\begin{enumerate}[i.)]
\item $\Theta'_k\Delta(\lambda) \cong \Delta(\mu'')$ if $d_{\lambda}$ looks as follows at positions $k-2, k-1, k$ with $d_{\mu''}$ displayed underneath (all other positons agree in the two weight diagrams):
 \begin{align*} 
  d_\lambda = \xymatrix{  &\underset{k-2}{\bullet}  &\underset{k-1}{\circ}  &\underset{k}{\circ}   } \\
 d_{\mu''}= \xymatrix{  &\underset{k-2}{\circ}  &\underset{k-1}{\bullet}  &\underset{k}{\circ}    }
\end{align*}

\item $\Theta'_k \Delta(\lambda) = \Pi\Delta(\mu')$  if $d_{\lambda}$ looks as follows at positions $k-2, k-1, k$ with $d_{\mu'}$ displayed underneath (all other positions agree in the two weight diagrams): 
 \begin{align*} 
  d_\lambda = \xymatrix{  &\underset{k-2}{\circ}  &\underset{k-1}{\circ}  &\underset{k}{\bullet}   } \\
 d_{\mu'}=\xymatrix{  &\underset{k-2}{\circ}  &\underset{k-1}{\bullet}  &\underset{k}{\circ}   } 
\end{align*}

\item In case $d_\lambda$ looks at positions $k-2,k-1,k$ as below, there is a short exact sequence $$0 \to \Pi\Delta(\mu') \to \Theta'_k \Delta(\lambda) \to \Delta(\mu'') \to 0 $$ where $d_{\mu''}$ and $d_{\mu'}$ are obtained from $d_\lambda$  by moving one black ball to position $k-1$ (from position $k-2$ respectively position  $k$) as follows:
\begin{align*} 
  d_\lambda = \xymatrix{  &\underset{k-2}{\bullet}  &\underset{k-1}{\circ}  &\underset{k}{\bullet}   } \\
d_{\mu'}=\xymatrix{  &\underset{k-2}{\bullet}  &\underset{k-1}{\bullet}  &\underset{k}{\circ}   }\\
  d_{\mu''}= \xymatrix{  &\underset{k-2}{\circ}  &\underset{k-1}{\bullet}  &\underset{k}{\bullet}    } 
\end{align*}
\item  $\Theta'_k\Delta(\lambda) =0$ in all other cases.
\end{enumerate}
\end{proposition}
\begin{proof}
This is just a reformulation of Lemma~\ref{lem:big}.
\end{proof}

\begin{proposition}[Translation of thin Kac modules]  \label{smallKac} Let $k\in\mathbb{Z}$. Then 
\begin{enumerate}[i.)]
 \item $\Theta'_k \nabla(\lambda) = \nabla(\mu'')$ if  $d_{\lambda}$ looks as follows at positions $k-1, k, k+1$ with $d_{\mu''}$ displayed underneath:
 \begin{align*} 
  d_\lambda= \xymatrix{  &\underset{k-1}{\bullet}  &\underset{k}{\bullet}  &\underset{k+1}{\circ}   }  \\
 d_{\mu''}=\xymatrix{  &\underset{k-1}{\bullet}  &\underset{k}{\circ}  &\underset{k+1}{\bullet}   }
\end{align*}
\item  $\Theta'_k \nabla(\lambda)=\Pi\nabla(\mu')$ if  $d_{\lambda}$ looks as follows at positions $k-1, k, k+1$ with $d_{\mu'}$ displayed underneath:
\begin{align*} 
 d_{\lambda}= \xymatrix{  &\underset{k-1}{\circ}  &\underset{k}{\bullet}  &\underset{k+1}{\bullet}   }  \\
  d_{ \mu'}= \xymatrix{  &\underset{k-1}{\bullet}  &\underset{k}{\circ}  &\underset{k+1}{\bullet}   }
\end{align*}
\item In case $d_\lambda$ looks locally at positions $k-1,k,k+1$ as below, there is a short exact sequence
$$0 \to \nabla(\mu'') \to \Theta'_{\InnaC{k}} \nabla(\lambda) \to  \Pi\nabla(\mu') \to 0 $$ where $d_{\mu'}$ and $d_{\mu''}$ are obtained from $d_\lambda$  by moving one black ball away from position $k$ (to position $k-1$ respectively $k+1$) as follow:
\begin{align*} 
  d_\lambda= \xymatrix{  &\underset{k-1}{\circ}  &\underset{k}{\bullet}  &\underset{k+1}{\circ}   }  \\
   d_{\mu'} = \xymatrix{  &\underset{k-1}{\bullet}  &\underset{k}{\circ}  &\underset{k+1}{\circ}  }\\
    d_{\mu''}=\xymatrix{  &\underset{k-1}{\circ}  &\underset{k}{\circ}  &\underset{k+1}{\bullet}    
     }
\end{align*}
\item  $\Theta'_k\nabla(\lambda) =0$ in all other cases.
\end{enumerate}
\end{proposition}
\begin{proof}
This is just a reformulation of Lemma~\ref{lem:small}.
\end{proof}

%{\bf Thin modules}: 
%\begin{enumerate}[i.)]
% \item $\Theta'_i \nabla(\lambda)=0$ if position $i$ of $d_{\lambda}$ is empty, or if both positions $i-1, i+1$ are full.
%
%%(the rest of the positions in these diagrams coincide).
% \item if positions $i-2, i-1, i$ of $d_{\lambda}, d_{\mu}$ are of the form:
% 
%%(the rest of the positions in these diagrams coincide).
%\item Otherwise, we have a short exact sequence 
%$$0 \to \nabla(\mu') \to \Theta'_i \nabla(\lambda) \to  \Pi\nabla(\mu'') \to 0 $$ where $\mu'$ is the diagram corresponding to 
%moving one black ball from box $i$ to box $i+1$, and $\mu''$ is the diagram corresponding to moving one black ball from box $i$ to box $i-1$:
%\begin{align*} 
%  \lambda= \xymatrix{  &\underset{i-1}{\circ}  &\underset{i}{\bullet}  &\underset{i+1}{\circ}   }  \\
% \mu'=\xymatrix{  &\underset{i-1}{\circ}  &\underset{i}{\circ}  &\underset{i+1}{\bullet}   }  \\
%   \mu'' = \xymatrix{  &\underset{i-1}{\bullet}  &\underset{i}{\circ}  &\underset{i+1}{\circ}   }
%\end{align*}
%(the rest of the positions in these diagrams coincide).
%
%\end{enumerate}

\subsection{Duality for simple modules}\label{ssec:duality_simple}
The goal of this subsection is to explain the effect of the duality functor on simple modules, more precisely we will show the following. 

\begin{proposition}[Duality formula]\label{prop:duality} There is an isomorphism $L(\lambda^\sharp)\simeq \Pi^m L(\lambda)^*$, where 
$m=\tfrac{|\lambda^\sharp|+|\lambda|}{2}$ in the notation of \eqref{omegarho}.
\end{proposition}

To explain the notation used we fix the lexicographic ordering on the set $(i,j)$, $1\leq i<j\leq n$, that means 
\begin{equation}
\label{pairs}
(1,2)<(1,3)<\dots<(1,n)<(2,3)<\dots<(n-1,n).
\end{equation}

For a given dominant weight $\lambda$ define $\lambda^\sharp$ by the following rule. 

Set $\nu=\lambda$ and $(a,b)=(1,2)$ and enumerate the black balls in $d_\nu$ by $1,\ldots, n$
from right to left and let $p_i$ be the position of the $i$th black ball. Then let $d_{m_{a,b}(\nu)}$ be the diagram obtained from $\nu$, by moving the $a$th and $b$th black ball one position to the right if possible; that is in formulas (with necessarily $f(p_a+1)=f(p_b+1)=0$)
\begin{eqnarray}
f_{m_{a,b}(\nu)}(i)=
\begin{cases}
0&\text{ if $i=p_a,p_b$},\\
1& \text{ if $i=p_a+1,p_b+1$},\\
f_\nu(i)&\text{otherwise}.
\end{cases}
&\quad\text{and}&f_{m_{a,b}(\nu)}=f_{\nu}\quad\text{otherwise}.
\end{eqnarray}
Repeat this procedure with the resulting weight (now $m_{1,2}(\lambda)$) and the next pair $(a,b)$ (now $(1,3)$) from \eqref{pairs} until there is now such pair left. Let $\lambda^\dagger$ be the resulting weight. Define finally $\lambda^\sharp$ be the weight obtained from $\lambda^\dagger$ by reflecting $d_{\lambda^\sharp}$ at $\tfrac{n-1}{2}$. 

\begin{example}
\mbox{}
 \begin{enumerate}[1.)]
  \item For the natural $\gg$-module $V = L(- \varepsilon_n)$, we have $V^* \cong \Pi L(- \varepsilon_n)$ by \eqref{eta}. Let us illustrate Proposition~\ref{prop:duality} for $n=4$. We have $\lambda =  -\varepsilon_4$, and we obtain 
 \begin{eqnarray*}
\quad d_\lambda& =& \xymatrix{   \ldots&\underset{-2}{\circ}    &\underset{-1}{\bullet}  &\underset{0}{\circ}  &\underset{1}{\bullet}  &\underset{2}{\bullet} &\underset{3}{\bullet} &\underset{4}{\circ} &\underset{5}{\circ} &\ldots} \\
 d_{\lambda^\dagger}&=& \xymatrix{     \ldots&\underset{-2}{\circ}     &\underset{-1}{\circ}  &\underset{0}{\bullet}  &\underset{1}{\bullet}  &\underset{2}{\bullet} &\underset{3}{\circ} &\underset{4}{\bullet} &\underset{5}{\circ}&\ldots} 
 \end{eqnarray*}
 Note that our rule is only nontrivial when we apply it the first time, i.e. for $(a,b)=(1,4)$.
By reflecting this at $\tfrac{n-1}{2}=\tfrac{3}{2}$ we obtain 
 \begin{eqnarray*}
 \quad d_{\lambda^\sharp}&=&\xymatrix{     \ldots&\underset{-2}{\circ}  &\underset{-1}{\bullet}  &\underset{0}{\circ}  &\underset{1}{\bullet}  &\underset{2}{\bullet} &\underset{3}{\bullet} &\underset{4}{\circ} &\underset{5}{\circ} &\ldots} 
 \end{eqnarray*}
  \item If $\lambda$ is typical (i.e. $d_\lambda$ has no adjacent black balls), then $m_{{a,b}}$ is nontrivial for all pairs from \eqref{pairs} and $\lambda^\sharp=-w_0(\lambda)+(1-n)\omega$. Indeed, 
first we move all $\bullet$'s $n-1$ positions to the right to obtain $d_{\lambda^\dagger}$, and then reflect with respect to $\tfrac{n-1}{2}$. The diagram of $\lambda^\sharp$ is then the mirror to that
of $\lambda$ with respect to the reflection at $0$.
\item For instance $\lambda=(10,8,4,3,1)$ gives the values  $\lambda^\dagger=(14,12,8,7,5)$ and $\lambda^\sharp=(-10, -8, -4, -3,-1)$. 
  \item On the other hand, if $\lambda=k\omega$, see \eqref{omegarho} for some $k\geq 0$, then all $m_{a,b}$ are trivial and therefore $\lambda^\dagger=\lambda$ and $\lambda^\sharp=-\lambda$. 
 \end{enumerate}
\end{example}

\begin{proof}[Proof of Proposition~\ref{prop:duality}] 
We use the isomorphism of $\gg_0$-modules:
\begin{equation}\label{fordual}
H^0(\gg_{-1},L(\lambda)^*)\simeq (H^0(\gg_1, L(\lambda)))^*.
\end{equation}

First, we will prove that 
\begin{equation}\label{dual2}
H^0(\gg_{1},L(\lambda))=V(\lambda^\dagger).
\end{equation} 
Consider the two following Borel subalgebras of $\gg$, the standard one 
$\bb=\bb_0\oplus\gg_{-1}$ and its opposite $\bb'=\bb_0\oplus\gg_{1}$. Then obviously $\lambda$ is the highest weight of $L(\lambda)$ with respect to $\bb$ and we have to show that
$\lambda^\dagger$  is the highest weight of $L(\lambda)$ with respect to $\bb'$. 
We use the odd reflection methods, introduced in \cite{PS94} for non-contragredient superalgebras. 
Observe that odd roots of $\bb'$ are of form $\varepsilon_i+\varepsilon_j$ for all $i\leq j\leq n$. Order these roots
by setting $\varepsilon_i+\varepsilon_j<\varepsilon_{i'}+\varepsilon_{j'}$ if $i<i'$ or $i=i'$ and $j<j'$ and 
enumerate them  according to this order $\alpha_1<\dots<\alpha_{n(n+1)/2}$.
Define the sequence of Borel subalgebra $\bb^0,\ldots,\bb^{n(n+1)/2}$ by setting
$\bb^0:=\bb$ and defining $\bb^k$ by adding the root $\alpha_k$ to $\bb^{k-1}$ and removing the root
$-\alpha_k$ if $\alpha_k$ is invertible. Note that $\bb^{n(n+1)/2}=\bb'$.
Let $\lambda^k$ denote the highest weight of $L(\lambda)$ with respect to $\bb^k$, in particular, $\lambda^0=\lambda$.
Then we have the following recursive formula:
\begin{enumerate}[i.)]
\item If $\alpha_k=2\varepsilon_i$ is not invertible, then $\lambda^k=\lambda^{k-1}$;
\item If $\alpha_k=\varepsilon_i+\varepsilon_j$, $i\neq j$ is invertible we have
$$\lambda^k=\begin{cases} \lambda^{k-1}&\text{if  }\lambda^{k-1}_i=\lambda^{k-1}_j,\\ \lambda^{k-1}+\alpha_{k}&\text{if  }\lambda^{k-1}_i\neq\lambda^{k-1}_j.\end{cases}$$
\end{enumerate}
Translating this condition to the language of weight diagrams, we obtain $\lambda^{n(n+1)/2}=\lambda^\dagger$. The proof of (\ref{dual2}) is complete.

To finish the proof recall that if $w_0$ is the longest element of the Weyl group then
$V(\mu)^*=V(-w_0(\mu))$. Hence using (\ref{fordual}) we obtain that the $\bb$-highest weight of $L(\lambda)^*$
equals $-w_0(\lambda^\dagger)$. Since $-w_0(\rho)=\rho-(n-1)\omega$, 
we have $-w_0(\mu)+\rho=-w_0(\mu+\rho)-(n-1)\omega$ for all dominant $\mu$. 
In the language of diagrams this means that the diagram of the highest weight of $L(\lambda)^*$  is obtained from the diagram of $\lambda^\dagger$
by the symmetry with respect to $\tfrac{n-1}{2}$. Hence it equals $\lambda^\sharp$ and the statement is proven.
\end{proof}

\section{Computation of decomposition numbers and multiplicity formulas}\label{sec:KL}
For simplicity, we disregard in this section the parity switch; this means that we will not distinguish between $\Pi M$ and $M$ for $M \in \mathcal{F}_n $.

% Idea: In general, I think we should consider operators $\Theta_{\vec{I}} := \Theta_{i_1+2} \ldots \Theta_{i_1+k+2}$ where $i_1, \ldots, i_1 + k$ is a sequence such that $f_{\lambda}(i) = 1$ for $i = i_1, \ldots, i_1 + k$, and $f_{\lambda}(i) = 0$ for $i = i_1 -1, i_1 + k+1$. For a fixed $\lambda$, taking the product $\Theta_{\vec{I}}$ over all ``disjoint'' sequences of black balls $I$ in $\lambda$, we should get something which has a direct summand $P(\lambda + \omega)$. This should lead us to the required multiplicities. I hope to work it out at least for $\lambda = 0$ tomorrow.
\subsection{Multiplicity formulas for \texorpdfstring{$P(0)$}{P(0)}}\label{ssec:KL_P_0}

% \begin{definition}
%  We say that a weight $\lambda$ is {\it thick-balanced around $l$} if $f_{\lambda}(l)=1$ and $f_{\lambda}(i+l) +f_{\lambda}(l-i) = 1$ for any $i = 1,2,\ldots,n-1$. The last condition means that the diagram of $\lambda$ is ``anti-symmetric'' with respect to position $l$. We say that a weight $\lambda$ is simply {\it thick-balanced} if it is thick-balanced around $0$.
% \end{definition}

% \begin{definition}
% We say that a weight $\lambda$ is {\it thin-balanced around $l$} if $f_{\lambda}(l)=0$ and $f_{\lambda}(l+i) +f_{\lambda}(l-i) = 1$ for any $i = 1,2,\ldots,n$. The last condition means that the diagram of $\lambda$ is ``anti-symmetric'' with respect to position $l$.
% 
%  We say that a weight $\lambda$ is just {\it thin-balanced} if it is thin-balanced around $1$.
% \end{definition}

Consider the projective cover $P(0)$ of the trivial module. We compute the multiplicities $(P(0):\Delta(\lambda))$ 
of thick Kac modules in $P(0)$. With the notation from \eqref{omegarho} let $U$ and $U^{-1}$ be the $1$-dimensional $\gg$-modules with highest weight $\omega$ and $-\omega$ respectively. Before we state the general result, we give some examples.

\begin{example}
 Let $n=2$. We claim that 
  \begin{eqnarray}
  \label{syl5}
   [P(0)]  = [\Delta(0)] +[\Delta(-\omega)].
   \end{eqnarray}
 Consider the weight diagram for the trivial module:
 \begin{eqnarray*}
  d_0&=&
    \xymatrix{ \dots&\underset{-2}{\circ}  &\underset{-1}{\circ}  &\underset{0}{\bullet}  &\underset{1}{\bullet}  &\underset{2}{\circ} &\underset{3}{\circ} &\ldots  } 
 \end{eqnarray*}
 Write $[P(0)] = [\Delta(0)] + \sum_{\mu > 0} c_{0, \mu} [\Delta(\mu)]$ in $\Groth_n$, where $c_{0, \mu}=(P(0):\Delta(\mu))$ are the multiplicities.
 We apply $\Theta_2 \Theta_3$ and obtain (ignoring parity)  from Proposition~\ref{bigKac}$$[\Theta_2 \Theta_3 \Delta(0)] = [\Delta(\omega)] +[\Delta(0)],$$ and for $\mu >0$ we have: $[\Theta_2 \Theta_3 \Delta(\mu)] = 0$ if $f_{\mu}(1) = 0$:
 $$ \xymatrix{  &\underset{-1}{?}  &\underset{0}{?}  &\underset{1}{\circ}  &\underset{2}{\circ} &\underset{3}{\circ}   } $$
 and $[\Theta_2 \Theta_3 \Delta(\mu)] = \Delta(\mu)$ if $f_{\mu}(1) = 1$:
 $$ \xymatrix{  &\underset{-1}{?}  &\underset{0}{\circ}  &\underset{1}{\bullet}  &\underset{2}{\circ} &\underset{3}{\circ}   } $$
 
 Thus $[\Theta_2 \Theta_3 P(0)] = [\Delta(\omega)] +[\Delta(0)] + \sum_{\mu > 0, f_{\mu}(1) = 1} c_{0, \mu} [\Delta(\mu)]. $ This is a projective module, and $\omega$ is minimal among the weights $'\mu$ such that $\Delta(\mu)$ appears in a $\Delta$-filtration of $\Theta_2 \Theta_3P(0)$, so $P(\omega)$ must be a  summand  by Lemma~\ref{lem:syl2}. On the other hand, 
 \begin{eqnarray}
 \label{syl3}
 [P(\omega)] = [P(0) \otimes U] = [\Delta(\omega)] + \sum_{\mu > 0} c_{0, \mu} [\Delta(\mu+\omega)].
\end{eqnarray}
Note that for any $\mu > 0$ such that $f_{\mu}(1) = 1$, $f_{\mu+\omega}(2) = 1$, so $[\Delta(\mu+\omega)]$ does not appear in $\Theta_2 \Theta_3 [P(0)]$. Thus $c_{0, \mu} = 0$ in this case, and therefore 
 \begin{eqnarray}
  \label{syl4}
[\Theta_2 \Theta_3 P(0)] = [\Delta(\omega)] +[\Delta(0)] = [P(\omega)].
\end{eqnarray}
Tensoring with $U^{-1}$, we obtain from \eqref{syl3} the desired formula \eqref{syl5}.
 
\end{example}

\begin{example}
 A similar computation shows that for $n=3$, applying $\Theta_2 \Theta_3 \Theta_4$ to $P(0)$ and comparing with $P(\omega)$ gives $$[P(0)] = [\Delta(0)] +[\Delta(-2\omega)] + [\Delta(\mu_1)] + [\Delta(\mu_2)]$$
 where $\mu_1$ has the diagram 
  $$ \xymatrix{  &\underset{-3}{\circ} &\underset{-2}{\circ}   &\underset{-1}{\bullet}  &\underset{0}{\bullet}  &\underset{1}{\circ}  &\underset{2}{\bullet} &\underset{3}{\circ}   } $$
   and $\mu_2$ has the diagram
   $$ \xymatrix{    &\underset{-3}{\circ}  &\underset{-2}{\bullet}  &\underset{-1}{\circ}  &\underset{0}{\bullet}  &\underset{1}{\bullet}  &\underset{2}{\circ} &\underset{3}{\circ}   } $$

 \end{example}
 More generally we have the following multiplicity formulas:
\begin{theorem}[Decomposition numbers for $P(0)$]\label{thrm:KL_trivial_weight}\hfill

 \begin{enumerate}[1.)]
 \item The thick Kac modules appearing in the (thick) standard filtration of the projective module $P(0)$ are precisely the $\Delta(\lambda)$ where $\lambda$ satisfies:
 $f_{\lambda}(0)=1$ and $f_{\lambda}(i) +f_{\lambda}(-i) = 1$ for any $i = 1,2,\ldots,n-1$.
 \item Similarly, the thin Kac modules appearing in the (thin) costandard filtration of the projective module $P(0)$ are precisely $\nabla(\lambda)$ where $\lambda$ satisfies:
 $f_{\lambda}(1)=0$ and $f_{\lambda}(1+i) +f_{\lambda}(1-i) = 1$ for any $i = 1,2,\ldots,n$.
  \end{enumerate}
In particular, the multiplicities of the above standard and costandard modules are $1$. 
\end{theorem}

\begin{proof}
Consider the module $\Ind^\gg_{\gg_0}V(0)$. We first claim that  $P(0) \cong \Ind^\gg_{\gg_0}V(0)$. It is a projective module, and there is a canonical map $\Ind^\gg_{\gg_0}V(0) \to L(0)$.
To prove the claim it suffices to verify $\dim \Hom_{\gg} (\Ind^\gg_{\gg_0}V(0), L(\lambda) ) = \delta_{\lambda, 0}$. Recall that $$ \dim \Hom_{\gg} (\Ind^\gg_{\gg_0}V(0), L(\lambda) ) = \dim \Hom_{\gg_0} (V(0), \op{Res}^{\gg}_{\gg_0} L(\lambda) ).$$
If this dimension is positive, then  
\begin{eqnarray*}
\dim \Hom_{\gg_0} (V(0), \op{Res}^{\gg}_{\gg_0} \Delta(\lambda) )&\text{and}& \dim \Hom_{\gg_0} (V(0), \op{Res}^{\gg}_{\gg_0} \nabla(\lambda) )
\end{eqnarray*}
  are positive, which means that both 
\begin{eqnarray*}
[\Lambda (\Lambda^2 V_{\overline{0}}) \otimes V(\lambda): V(0)] , \quad\text{and}\quad&  [\Lambda (S^2 V_{\overline{0}}) \otimes V(\lambda): V(0)]
\end{eqnarray*}
  are positive. Yet the $\gg_0$-modules $\Lambda (\Lambda^2 V_{\overline{0}}), \Lambda (S^2 V_{\overline{0}})$ have only the common summand $V(0)$, appearing with multiplicity one, (see for example \cite[Chapter 2]{W}). This implies that $\lambda = 0$, and $P(0) \cong \Ind^\gg_{\gg_0}V(0)$. This proves our claim. Now, $$(P(0):\Delta (\lambda)) = ( \Ind^\gg_{\gg_0}V(0): \Delta (\lambda)) = [\Lambda(\gg_{-1}):V(\lambda)] = [\Lambda (\Lambda^2 V^*_0) : V(\lambda)].$$
This multiplicity is always equal to $0$ or $1$, and the latter happens 
precisely when $\lambda^{*}$ (the highest weight of $V(\lambda)^*$) 
satisfies 
the conditions in \ref{itm:KL_triv_weight_2}.) of Lemma \ref{armsandlegs} below (this is proved, for instance, 
in \cite[Chapter 2]{W}, and in 
\cite[Chapter I, Appendix A, Par. 7]{Macd}). By Lemma \ref{armsandlegs}, this 
is equivalent to the required condition for $\lambda$.
A similar argument works for thin Kac modules.
\end{proof}

Assume $\mu=(\mu_1,\ldots,\mu_n)$ is a dominant weight, and $\mu_i \geq 
0$ for $i = 1, \ldots, n$.
Then we define for $1\leq i\leq n$
\begin{eqnarray}
\op{arm}_i(\mu)=\mu_i - i+1&\text{and}& 
\op{leg}_i=\op{leg}_i(\mu)=\mu^{\vee}_i - i+1
\end{eqnarray}
 as long as this number is positive and call it the $i$-th {\it arm length} 
respectively {\it leg length}. 
 Here {$\mu$ is identified with a Young diagram, and $\mu^\vee$ denotes 
the transposed weight (that means the transposed Young diagram) defined as 
$\mu^\vee_i=|\{j\mid \mu_j\geq i\}|$. 
Clearly $\mu$ is uniquely determined by the two strictly decreasing sequences 
$\op{arm}_1, \op{arm}_2, \ldots, \op{arm}_r$ and $\op{leg}_1, \op{leg}_2, 
\ldots, \op{leg}_r$ observing that the length of these sequences agree. (In 
terms of Young diagrams this means that we describe the diagram by listing the 
number of boxes on and to the right of the diagonal in each row respectively 
strictly below the diagonal in each column.)

\begin{lemma} 
\label{armsandlegs}
Let $\mu$ be a dominant weight. Then the following are equivalent
\begin{enumerate}[\rm{(}I\rm{)}]
\item 
The set $B_{\mu}=\{ \mu_i 
-i +1 \mid 1\leq i\leq n\}$ contains $0$ and precisely one element from each 
pair $\pm j$, where $1\leq j\leq n-1$.
\label{itm:KL_triv_weight_1}
\item $\mu_i \geq 0$ for $i=1, 
\ldots, n$ and $\op{arm}_i + 1 = \op{leg}_i$ whenever $\op{arm}_i=\op{arm}_i(\mu)$ and 
$\op{leg}_i=\op{leg}(\mu)_i$ are defined (equivalently, $\mu_i+1 = 
\mu^{\vee}_i$). \label{itm:KL_triv_weight_2}
\item 
The set $\{ \mu^*_i+n-i \mid 1\leq i\leq n\}$ contains $0$ and 
precisely one element from each pair $\pm j$, where $1\leq j\leq n-1$.
\label{itm:KL_triv_weight_3}
\end{enumerate}
\end{lemma}
\begin{proof}
\InnaB{The equivalence \eqref{itm:KL_triv_weight_1} $\Leftrightarrow$ 
\eqref{itm:KL_triv_weight_3} follows directly from the equality $\mu_i = 
-\mu^*_{n+1-i}$ for $1 \leq i \leq n$.} We will now show \eqref{itm:KL_triv_weight_2} $\Rightarrow$ 
\eqref{itm:KL_triv_weight_1}. Since the sequence $(\mu_i -i +1)_i$ is strictly decreasing, and $\mu_i \geq 0$ for $i=1, 
\ldots, n$, we have $-n+1 
\leq \mu_n-n+1 \leq \mu_i - i+1 \leq \mu_1 \leq n-1.$
Hence $B_{\mu} \subset \{-n+1, -n+2, \ldots, n-1\}$. 
It remains to show that $\mu_i - i+1 + \mu_j -j +1 \neq 0$ for any $i \neq j$.

Let $i \neq j$. 
In case $ \mu_i <i$, $\mu_j <j$ we have $\mu_i -i + 1, \mu_j -j+1 \leq 0$, and 
 since these numbers are distinct, their sum is clearly not zero. 
Otherwise (without loss of generality) $\mu_i \geq i$. Then $\op{arm}_i$, 
$\op{leg}_i$ are defined, and the assumption $\op{arm}_i + 1 = \op{leg}_i$ 
implies $\mu^{\vee}_i = \mu_i +1$. 
Thus, 
\begin{eqnarray*}
\mu_i - i +1 + \mu_j - j+1 = \mu^{\vee}_i -j +\mu_j -i +1&\not=&0,
\end{eqnarray*}
since $\mu^{\vee}_i -j$, $\mu_j -i$ are integers, both non-negative if 
$\mu_j\geq i$, and both negative if $\mu_j<i$.

 For the converse, we compare the cardinality of the set of weights $\mu$ 
satisfying \eqref{itm:KL_triv_weight_1} respectively 
\eqref{itm:KL_triv_weight_2}. Indeed, the first set is in bijection with 
the collection $\mathcal{B}$ of 
sets $B_{\mu} \subset \{-n+1, -n+2, \ldots, n-1\}$ 
having $n$ elements, and such that if $j\not=0$ lies inside, then $-j$ does 
not (the bijection is via $\mu \mapsto B_{\mu} = \{ \mu_i - i+1\}_{i=1, \ldots, n}$). In particular, $0$ necessarily lies inside all such sets. 
 
 The second set is in bijection with the collection \InnaB{$\mathcal A$} of 
strictly decreasing sequences $(\op{arm}_1, 
\op{arm}_2, \ldots, \op{arm}_k)$ with  entries in the range $1, 2, \ldots, n-1$ 
and $0\leq k\leq n$ (denoting the number of boxes on the diagonal in the 
corresponding Young diagram). The bijection is via $\mu \mapsto A_{\mu} = (\op{arm}_1(\mu), 
\op{arm}_2(\mu), \ldots, \op{arm}_k(\mu))$ where $k$ is the number of boxes on the diagonal of the Young diagram of $\mu$ (that is, the largest value for which $\op{arm}_k(\mu)$ is defined). Both sets have cardinality $2^{n-1}$, and we have already proved that 
\eqref{itm:KL_triv_weight_2} $\Rightarrow$ 
\eqref{itm:KL_triv_weight_1}. Hence the statement of the Lemma follows.
\end{proof}

  \subsection{Arrow diagrams}\label{ssec:arc_diagrams}
 For a dominant weight $\lambda$ define the function $g_\lambda:\mathbb Z\to\{-1,1\}$
by setting $g_\lambda(i)=(-1)^{f_\lambda(i) +1}$; so $g_\lambda(i)=1$ if $i\in c_\lambda$, that means $d_\lambda$ has a black ball at the $i$-th position and $g_\lambda(i)=-1$ if $i\notin c_\lambda$.
For any $j<i$ set 
\begin{eqnarray*}
r^+(i,j)\;=\;\sum_{s=j}^{i-1}g_{\InnaC{\lambda}}(s)&{ and }& r^-(i,j)\;=\;-\sum_{s=j+1}^{i}g_{\InnaC{\lambda}}(s).
\end{eqnarray*}

For every $i\in c_\lambda$ define
$$\overset{\leftarrow i}{\blacktriangle}(\lambda)=\left\{j< i\,|\,r^+(i,j)=0,\,r^+(i,s)\geq 0\,\,\text{for all}\,j<s<i\right\}.$$
Also for every $j\notin c_\lambda$ define
$$\underset{j\dashrightarrow}{\blacktriangledown}(\lambda) = \left\{i> j\,|\, r^-(i, j)=0,\,r^-(s,j)\geq 0\,\,\text{for all}\,j<s<i \right\}.$$

To obtain the {\it arrow diagram} for $\lambda$ equip $d_\lambda$  with solid and dashed arrows, as follows: 
\begin{itemize}
 \item For every $i \in c_{\lambda}$ we draw a solid arrow from $i$ to every $j \in \overset{\leftarrow i}{\blacktriangle}(\lambda)$.
\item For every $j \notin c_{\lambda}$ we draw a dashed arrow from $j$ to every $i \in \underset{j\dashrightarrow}{\blacktriangledown}(\lambda)$.
\end{itemize}

Observe that $\underset{j\dashrightarrow}{\blacktriangledown}(\lambda)\subset c_\lambda$ and $\overset{\leftarrow i}{\blacktriangle}(\lambda)\cap c_\lambda=\emptyset.$
(To see the latter assume $k\in c_\lambda$ and $k\in \InnaC{\overset{\leftarrow i}{\blacktriangle}(\lambda)}$. Then $0=r^+(i,k)=1+r^+(i,k+1)>0$, which is a contradiction. Argue similarly for the first.) That means solid arrows always start at black balls and end at empty places, whereas dashed arrows start at empty places and end at black balls.

\begin{example}
 Let $n = 4$, $\lambda = (1,1,0,0)$. Below is the diagram of $\lambda$: each element of $c_{\lambda}$ is marked 
with a black ball, and for each $i \in c_{\lambda}$, the positions $j \in \overset{\leftarrow i}{\blacktriangle}(\lambda)$ are connected with $i$ by solid arrows. 
We also connect by dashed arrows all  $j\notin c_\lambda$ with $i \in \underset{j\dashrightarrow}{\blacktriangledown}(\lambda)$.
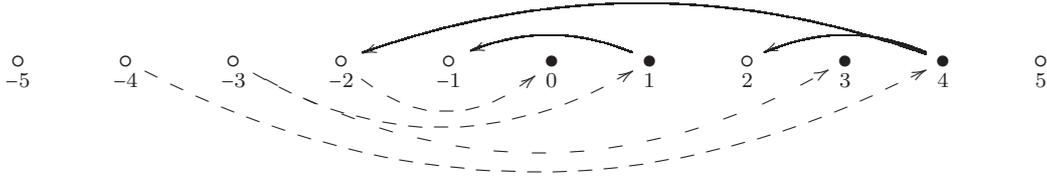
  $$ \xymatrix{&{} &{} &{} &{}&{}  &{}&{}&{}&{}&{}  &{} \\ &\underset{-5}{\circ} &\underset{-4}{\circ} \ar@{-->}@/_3.5pc/[rrrrurrrrd] &\underset{-3}{\circ} \ar@{-->}@/_2.1pc/[rrurrd] \ar@{-->}@/_2.9pc/[rrrurrrd] &\underset{-2}{\circ} \ar@{-->}@/_1.6pc/[rurd] &\underset{-1}{\circ}  &\underset{0}{\bullet}  &\underset{1}{\bullet} \ar@/_0.8pc/[luld]  &\underset{2}{\circ}  &\underset{3}{\bullet} &\underset{4}{\bullet} \ar@/_0.8pc/[luld] \ar@/_1.8pc/[lllullld] &\underset{5}{\circ} \\  &{} &{} &{} &{}&{}  &{}&{}&{}&{}&{}  &{} &{}&{}&{} } $$
 \vspace{1pt}
 For instance, $\overset{\leftarrow 4}{\blacktriangle}(\lambda)=\{-2,2\}$, whereas $\overset{\leftarrow 0}{\blacktriangle}(\lambda)=\emptyset$ and $\underset{\dashrightarrow -3}{\blacktriangle}(\lambda)=\{1,3\}$ 
\end{example}

\begin{lemma}\label{lem:comb_arc_diag1} Let $\lambda$ be a dominant weight.
\begin{enumerate}[1.)]
 \item\label{itm:comb_arc_diag1_2} Let $i_1, i_2\in  c_\lambda$, $i_1<i_2$ and $j_1 \in \overset{\leftarrow i_1}{\blacktriangle}(\lambda)$, $j_2 \in \overset{\leftarrow i_2}{\blacktriangle}(\lambda)$, then either $j_2<j_1$ or $j_2 > i_1$. In other words, two solid arrows can only intersect at a common source. In particular, ${\overset{\leftarrow i_1}{\blacktriangle}(\lambda)} \cap {\overset{\leftarrow i_2}{\blacktriangle}(\lambda)}=\emptyset$.
 \item\label{itm:comb_arc_diag1_4} Let $j_1, j_2 \notin  c_\lambda$, $j_1 < j_2$ and $i_1 \in \underset{j_1\dashrightarrow}{\blacktriangledown}(\lambda) $, $i_2 \in\underset{j_2\dashrightarrow}{\blacktriangledown}(\lambda) $, then either $i_2 < i_1$ or $i_1 < j_2$. In other words, two dashed arrows can only intersect at a common source.
 In particular, ${\underset{j_1\dashrightarrow}{\blacktriangledown}(\lambda)} \cap {\underset{j_2\dashrightarrow}{\blacktriangledown}(\lambda) }=\emptyset$.
\end{enumerate}
\end{lemma}
\begin{proof}  

We only show the first part, since the second is similar.  Assume that the two solid arrow intersect, that is we have $j_2<j_1<i_2$ such that 
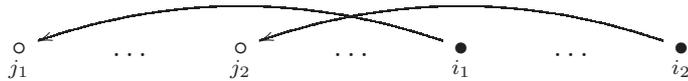
\begin{equation*}
\xymatrix{&{}&{} &{} &{} &{} &{} &{}\\&\underset{j_1}{\circ} &{\ldots} &\underset{j_2}{\circ} &{\ldots}  &\underset{i_1}{\bullet} \ar@/_1.3pc/[llulld] &{\ldots}   &\underset{i_2}{\bullet} \ar@/_1.3pc/[llulld]  } 
\end{equation*}
Then $r^+(i_1, j_2) \geq 0$, since $j_1 \leq j_2 < i_1$, $r^+(i_2, i_1+1) \geq 0$, hence  $r^+(i_2, j_2) = r^+(i_1, j_2) +1 + r^+(i_2, i_1+1) >1 \neq 0$. This contradicts $j_2 \in \overset{\leftarrow i_2}{\blacktriangle}(\lambda)$.
\end{proof}

We note that for every black ball in $d_{\lambda}$ there exists at least one arrow ending there:
\begin{lemma}\label{lem:comb_arc_diag2}
 For any $i \in c_{\lambda}$, there exists $j_i \notin c_{\lambda}$ such that $i \in  \underset{j_i\dashrightarrow}{\blacktriangledown}(\lambda)$.
\end{lemma}

\begin{remark}
 The index $j_i$ (hence the above arrow) is unique due to Lemma~\ref{lem:comb_arc_diag1}.
\end{remark}

\begin{proof}%[Proof of Lemma~\ref{lem:comb_arc_diag2}]
 Given $i \in c_{\lambda}$, consider the set $J_i=\{j \leq  i\mid \, r^-(i, j)=0\}\subset \mathbb{Z}$. Clearly $i \in J_i$, but $J_i$ has at least one other element: indeed, for $j \ll 0$, $r^-(i, j-1) < 0$, while $r^-(i, i-1) = 1$, so $J_i \setminus \{i\} \neq \emptyset$. 
 
 In fact, we claim that the set $J_i \setminus c_{\lambda}$ is not empty either. Indeed, for any $i' \in J_i \cap c_{\lambda}$ we have $J_{i'} \subset J_{i}$ (since $r^-(i, j) = r^-(i, i') + r^-(i', j)$ for any $j \leq i'$). Taking $i':= \min (J_i \cap c_{\lambda})$, we obtain $J_{i'} \subset J_{i}$, and $J_{i'}\setminus c_{\lambda} \neq \emptyset$, since $J_{i'} \cap c_{\lambda} = \{i'\}$. Thus $J_i \setminus c_{\lambda} \neq \emptyset$.  Let $j_i = \max (J_i \setminus c_{\lambda})$. We claim that $i \in\underset{j_i\dashrightarrow}{\blacktriangledown}(\lambda)$, i.e., that $r^-(s,j_i)\leq 0$ for all $j_i<s<i$. 
 
 Assume not, then there exists some $s$ such that $r^-(s,j_i)> 0$ and $j_i<s<i$. Since $r^-(i-1,j_i)< 0$, there exists $s'$ such that $ j_i <s< s'<i-1$ with $r^-(s', j_i+1) =0$, and thus $r^-(i, s') =0$. This implies $s' \in J_i $ and $s' > j_i$, which contradicts the choice of $j_i$.
\end{proof}

The following is an important tool for induction arguments. 
 \begin{lemma}
 \label{lem:change}
 Consider a dominant weight $\nu$ and $i$ such that $f_{\nu}(i) =1$, $f_{\nu}(i+1) = 0$. Let $\lambda$ be obtained from $\nu$ by moving a black ball from position $i$ to position $i+1$:
 \begin{eqnarray*}
d_\nu  = \xymatrix{&\underset{i}{\bullet} &\underset{i+1}{\circ}}&\quad\quad{\text and}\quad\quad&d_\lambda  = \xymatrix{&\underset{i}{\circ} &\underset{i+1}{\bullet}}
\end{eqnarray*}
Next, let $i_1$ be such that $i+1 \in \overset{\leftarrow i_1}{\blacktriangle}(\nu)$, $f_{\nu}(i_1)=1$ and $i_2$ be such that $i+2 \in \overset{\leftarrow i_2} {\blacktriangle}(\nu) {\sqcup \{i_2\}}$, $f_{\nu}(i_2)=1$  (if $i_1$ or $i_2$ do not exist, we set the corresponding value to be $\infty$). Then for $j \in c_{\lambda}$, we have
 \begin{eqnarray*}
  \overset{\leftarrow j}{\blacktriangle}( \lambda)& =& \begin{cases}
                        \overset{\leftarrow j}{\blacktriangle}(\nu) &\text{ if } j \notin \{i+1, i_1, i_2\}, \\
                       \emptyset &\text{ if } j= i+1, \\
                       \overset{\leftarrow i_2}{\blacktriangle}(\nu) \cup\{i\} \cup  \overset{\leftarrow i}{\blacktriangle}(\nu)& \text{ if } j = i_2 ,\\
                      \overset{\leftarrow i_1}{\blacktriangle}(\nu) \setminus \{i+1\} &\text{ if } j  = i_1. 
                     \end{cases}
\end{eqnarray*}
\end{lemma}
\begin{proof}
This follows directly from the definitions.
\end{proof}

We denote by $\blacktriangle(\lambda)$ the set of weight diagrams which are obtained from $d_\lambda$ by sliding some black balls along solid arrows in the arrow diagram for $\lambda$, and by $\blacktriangledown(\lambda)$ the set of  weight diagrams obtained by sliding some black balls (backwards) along dashed arrows. In formulas
\begin{eqnarray}
\label{Deftriang}
{\blacktriangle}(\lambda)&=&\left\{ \mu \in \Lambda_n \;\left|\; \forall i \in c_{\lambda}:\; f_{\mu} (i) + \sum_{j \in \overset{\leftarrow i}{\blacktriangle}(\lambda)} f_{\mu}(j) = 1\right.  \right \}\\
\label{Deftriangdown}
{\blacktriangledown}(\lambda)&=&\left\{ \mu \in \Lambda_n \;\left|\; \forall j \notin c_{\lambda}:\; 1 - f_{\mu}(i) + \sum_{i \in \underset{j\dashrightarrow}{\blacktriangledown}(\lambda)} (1-f_{\mu}(i)) = 1 \right.\right  \}
\end{eqnarray}

\begin{proposition}\label{prop:comb_arc_diag}
For any dominant weight $\lambda$ we have ${\blacktriangle}(\lambda)\cap {\blacktriangledown}(\lambda)
= \{\lambda\}$.
\end{proposition}

\begin{proof}
Clearly $\lambda\in {\blacktriangle}(\lambda)\cap {\blacktriangledown}(\lambda)$.  Let $\mu \in {\blacktriangle}(\lambda)\cap {\blacktriangledown}(\lambda)$. Assume $\mu \neq \lambda$.
 Then the weight diagram $d_\mu$  of $\mu$ is obtained from $d_\lambda$ by sliding some (at least one!) black balls along dashed arrows (since $\mu \in  {\blacktriangledown}(\lambda)$). 
 Consider such a dashed arrow $j_0 \dashrightarrow i_0$ of minimal length. That is, $f_{\lambda}(j_0) = 0$, $f_{\mu}(j_0) = 1$, $f_{\lambda}(i_0) = 1$, $f_{\mu}(i_0) = 0$, and $f_{\lambda}(s) = f_{\mu}(s)$ for any $j_0 < s <  i_0$.
That is, the arrow diagram of $\lambda$ and the weight diagram of $\mu$ are locally of the form
 \begin{eqnarray*}
 \xymatrix{&\underset{j_0}{\circ} \ar@{-->}@/_1.3pc/[rurd] &{\ldots}  &\underset{i_0}{\bullet} \\&{}&{} &{} } &\quad\text{  resp.  }\quad&\xymatrix{&\underset{j_0}{\bullet} &{\ldots}  &\underset{i_0}{\circ} }
 \end{eqnarray*}
 On the other hand, $\mu \in  {\blacktriangle}(\lambda)$, which means that $d_\mu$ was obtained by sliding some black balls via solid arrows in the diagram for $\lambda$. In particular,  
 \begin{eqnarray*}
\sum_{j \in  {\overset{\leftarrow i_0}{\blacktriangle}(\lambda)}} 
f_{\mu}(j) &=& f_{\mu}(i_0) + \sum_{j \in {\overset{\leftarrow 
\InnaB{i_0}}{\blacktriangle}(\lambda)}}f_{\mu}(j) = 1
\end{eqnarray*}
  (the first equality follows from $f_{\mu}(i_0) = 0$); hence there exists exactly one $j_0 \in  {\overset{\leftarrow i_0}{\blacktriangle}(\lambda)}$ such that $f_{\mu}(j) = 1$, while $f_{\lambda}(j) = 0$. 
 
 Also, since $\mu \in\blacktriangle(\lambda)$, the black ball at position $j_0$ in $d_\mu$ has been slid through a solid arrow in the diagram of $\lambda$. We denote the source of this arrow by $i$ (thus $f_{\mu}(i) =0$, $f_{\lambda}(i) = 1$). Recall that $f_{\lambda}(s) = f_{\mu}(s)$ for any $j_0 < s <  i_0$, so $j \leq j_0$, and $i \geq i_0$. If $i_0\not=i$ then the arrow diagram for $\lambda$ is locally of the form
 $$
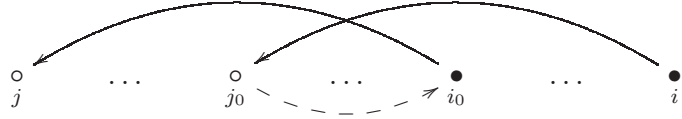
\xymatrix{&{}&{} &{} &{} &{}&{} &{} \\ &\underset{j}{\circ} &{\ldots}  &\underset{j_0}{\circ}  \ar@{-->}@/_1.2pc/[rurd] &{\ldots}  &\underset{i_0}{\bullet} \ar@/_2.3pc/[llulld] &{\ldots}  &\underset{i}{\bullet} \ar@/_2.3pc/[llulld] \\&{}&{} &{} &{} &{}&{} &{}  } $$
 and we obtain contradiction to Lemma~\ref{lem:comb_arc_diag1}. Hence $i_0=i$, and $d_\mu$ is of the form $$\xymatrix{&\underset{j}{\bullet} &{\ldots}  &\underset{j_0}{\bullet}  &{\ldots}  &\underset{i_0=i}{\circ}  &{\ldots}  &  } $$
 and $j_0, i_0$ are connected by both a solid and a dashed arrow in the arrow diagram for $\lambda$. Thus $r^-(i_0, j_0) =0 = r^+(i_0, j_0)$, which leads to the contradiction
 $-1 =r^-(i_0, j_0) -1 = r^+(i_0, j_0+1) = r^+(i_0, j_0) +1 = 1$
\end{proof}

\subsection{Multiplicity formulas and decomposition numbers}\label{ssec:KL_main}
\begin{theorem}\label{thrm:main-mult2} For $\lambda$ a dominant weight, the (thick and thin) Kac filtrations of the projective module $P(\lambda)$ give equalities in the reduced Grothendieck group $\mathcal{F}_n$ of the form
\begin{eqnarray}
\label{dec}
[P(\lambda)]=\sum_{\mu \in {\blacktriangle}(\lambda)} [\Delta(\mu)]&\quad\text{and}\quad& [P(\lambda)] =\sum_{\mu \in{\blacktriangledown}(\lambda)} [ \nabla(\mu +2\omega)].
\end{eqnarray}
\end{theorem}
\label{start}
\begin{remark} Note that Theorem~\ref{thrm:KL_trivial_weight} is Theorem~\ref{thrm:main-mult2} in the special case $\lambda=0$. 
\end{remark}
Before we proceed to prove the theorem, we state the most important direct consequence of Theorem~\ref{thrm:main-mult2} and Corollary~\ref{cor:BGG}.

\begin{theorem}[Decomposition numbers]\label{cor:KL_coeff}
  Let $\mu$ be a dominant weight. The following hold in the Grothendieck group of $\mathcal{F}_n$:
  \begin{eqnarray*}
  [\Delta (\mu)] = \sum_{\lambda\text{ s.t. }\mu \in {\blacktriangledown}(\lambda)} [ L(\lambda)], &\quad\text{and}\quad& [\nabla(\mu)] = \sum_{\lambda\text{ s.t. } \mu  \in{\blacktriangle}(\lambda)} [ L(\lambda)]
  \end{eqnarray*}
\end{theorem}

Together with Proposition~\ref{prop:extension} this implies
\begin{corollary} 
\label{cor:Exts}
If $\Ext^1(L(\lambda),L(\mu))\neq 0$, then either $\lambda\in{\blacktriangledown}(\mu)$ or $\mu\in{\blacktriangle}(\lambda)$. 
\end{corollary}

We introduce abbreviations for the right hand sides of the formulas \eqref{dec}. Denote 
\begin{eqnarray}
\label{abb}
[{\blacktriangle}(\lambda)] := \sum_{\mu \in {\blacktriangle}(\lambda)} [\Delta(\mu)] &\quad\text{  and   }\quad&[{\blacktriangledown}(\lambda)]:= \sum_{\mu \in {\blacktriangledown}(\lambda)} [\nabla(\mu)],
\end{eqnarray}
both considered as elements in the Grothendieck group. For the proof of Theorem~\ref{thrm:main-mult2} we need the following important fact (with $\theta_i$ as in Section~\ref{ssec:act}):

 \begin{prop}\label{prop:KL_mult_aux2}
Let $\nu$ be a dominant weight and $i$ such that $f_{\nu}(i) =1$, $f_{\nu}(i+1) = 0$. Let $d_\lambda$ be obtained from $d_\nu$ by moving a black ball from position $i$ to position $i+1$, which means
\begin{eqnarray*}
d_\nu  = \xymatrix{&\underset{i}{\bullet} &\underset{i+1}{\circ}}&\quad\quad{\text and}\quad\quad&d_\lambda  = \xymatrix{&\underset{i}{\circ} &\underset{i+1}{\bullet}}
\end{eqnarray*}
Then $\theta_{i+2}[{\blacktriangle}(\nu)]= [{\blacktriangle}(\lambda)]$ and $\theta_{i+2}[{\blacktriangledown}(\nu)]= [{\blacktriangledown}(\lambda)]$.

\end{prop}
\begin{proof}

First of all, observe that $c_{\lambda} = c_{\nu} \sqcup \{ i+1 \}\setminus \{i \}$.  As in Lemma~\ref{lem:change}, let $i_1$ be such that $i+1 \in \overset{\leftarrow i_1}{\blacktriangle}(\nu) $ and $i_2$ be such that $i+2 \in \overset{\leftarrow i_2}{\blacktriangle}(\nu) \sqcup \{i_2\}$ (if $i_1$ or $i_2$ do not exist, we set the corresponding value to be $\infty$). 

Then $ \overset{\leftarrow j}{\blacktriangle}( \lambda)= \overset{\leftarrow j}{\blacktriangle}( \nu)$ unless $j\in\{i+1,i_1,i_2\}$. Let now $\zeta\in{\blacktriangle}(\nu)$. We compute $[\Theta_{i+2}\Delta(\zeta)]$ and show that its standard summands lie in $[{\blacktriangle}(\lambda)]$.
\begin{enumerate}[i.)]
\item If neither $i$ nor $i+2$ occurs in $c_{\zeta}$, then $[\Theta_{i+2}\Delta(\zeta)] =0$ by Proposition~\ref{bigKac}. Similarly for the case $i+1 \in c_{\zeta}$.

Otherwise let $(a,b,c)\in\mathbb{Z}_{\geq0}^3$ be the positions of the black balls in $d_\zeta$ obtained from $d_\nu$ by sliding along arrows connected with $i,i_1, i_2$ respectively. In particular $b\not=i+1$.
\item If $i$ occurs in $c_{\zeta}$, but not $i+2$,
$$d_\zeta  = \xymatrix{&\underset{i}{\bullet} &\underset{i+1}{\circ} &\underset{i+2}{\circ} }$$
then $(a,b,c)=(i,b,c)$ with $c\not=i+2$ and then $[\Theta_{i+2}\Delta(\zeta)]=[\Delta(\zeta')]$ where $\zeta'$ corresponds to $(i+1,b,c)$.
\item If $i+2$ occurs in $c_{\zeta}$, but not $i$, 
$$d_\zeta  = \xymatrix{&\underset{i}{\circ} &\underset{i+1}{\circ} &\underset{i+2}{\bullet} }$$
then $(a,b,c)=(a,b,i+2)$ with $a\not=i$ and then $[\Theta_{i+2}\Delta(\zeta)]=[\Delta(\zeta')]$ where $\zeta'$ corresponds to $(i+1,b,a)$.
\item If $i$ and $i+2$ occur in $c_{\zeta}$, 
$$d_\zeta  = \xymatrix{&\underset{i}{\bullet} &\underset{i+1}{\circ} &\underset{i+2}{\bullet} }$$ then $(a,b,c)=(i,b,i+2)$ and then $[\Theta_{i+2}\Delta(\zeta)]=[\Delta(\zeta')]+[\Delta(\zeta'')]$ where $\zeta'$ corresponds to $(i+1,b,i+2)$ and $\zeta''$ corresponds to $(i+1,b,i)$. 
\end{enumerate} 
Note that the resulting triples are $(i+1,b,c)$ (from i) and iv)) and $(i,b,a)$ (from ii) and iv)). Hence, by varying $\zeta$, we obtain thanks to Lemma~\ref{lem:change} precisely the weights in ${\blacktriangle}(\lambda)$.  The first claim follows. 

The proof for the second part is similar, but easier. Let now $\zeta-2\omega \in{\blacktriangledown}(\nu)$, hence  $\zeta \in{\blacktriangledown}(\nu+2\omega)$. By assumption $d_{\nu+2\omega}$ has a black ball at position $i+2$, but not at $i+3$. Assume first that there is no black ball at position $i+1$ and let $i_1$ be the starting point of the dashed arrow $a$ ending in $i+2$:
\begin{eqnarray*}
 d_{\nu+2\omega}  = \xymatrix{&\underset{i_1}{\circ} \ar@{-->}@/_1.2pc/[rdur] &{\ldots} &\underset{i+1}{\circ} &\underset{i+2}{\bullet} &\underset{i+3}{\circ}}.
\end{eqnarray*}
\vspace{1pt}
Then $d_{\lambda+2\omega}$ has no black ball at positions $i+1$ and $i+2$, but at $i+3$ but the dashed arrow $a$ gets replaced by a dashed arrow from $i+1$ to $i+3$. Now if $i+2$ does not occur in $\zeta$ then  
$[\nabla(\zeta)]$ does not give any contribution when applying $\theta_{i+2}$, otherwise it gives exactly the sum $[\nabla(\zeta')]+[\nabla(\zeta'')]$ where the entry $i+2$ in $c_{\zeta}$ is replaced by $i+1$ in $c_{\zeta'}$ and respectively $i+3$ in $c_{\zeta''}$. The claim follows. 

Assume now that there is a black ball at position $i+1$ and let $i_1$ and $i_2$ be the starting points of the dashed arrows ending in $i+1$, respectively $i+2$:
\begin{eqnarray*}
 d_{\nu+2\omega}  = \xymatrix{&\underset{i_2}{\circ} \ar@{-->}@/_2.2pc/[rrdurrr] &{\ldots} &\underset{i_1}{\circ} \ar@{-->}@/_1.3pc/[rdur] &{\ldots} &\underset{i+1}{\bullet} &\underset{i+2}{\bullet} &\underset{i+3}{\circ}}.
\end{eqnarray*} 
\vspace{8pt}

Note that $i_2<i_1$ and that these two arrows get replaced in the arrow diagram for $\lambda+2\omega$ by two dashed arrows starting at $i_1$ and ending at $i+1$ respectively $i+3$. If we focus on the black balls for these two arrows, then $d_\zeta$ can have a black ball at positions $(i+1,i+2)$, $(i_1,i+2)$, $(i+1,i_2)$, and $(i_1, i_2)$. The last two options clearly give $\theta_i[\nabla(\zeta)]=0$ by Proposition~\ref{smallKac}, so we will only consider the first two options. Applying $\theta_{i+2}$ to $[\nabla(\zeta)]$ we obtain the sum of  $[\nabla(\zeta')]$'s where $\zeta'$ has instead black balls at positions $(i+1,i+3)$ in the first case, at positions $(i_1,i+1)$ and $(i_1,i+3)$ in the second case. But this gives exactly the claim. 
\end{proof}

\begin{proof}[Proof of Theorem~\ref{thrm:main-mult2}]
First we restrict ourselves to the case where $\lambda_i \geq 0$ for all $i$ (this is possible since $_-\otimes U: \mathcal P_n \to \mathcal P_n$ is an equivalence of categories, shifting the weights by $\omega$). In particular, $|\lambda|\geq0$ with equality exactly when $\lambda=0$.  Now let $|\lambda|>0$, then we can find some dominant weight $\nu$ and $i\in\mathbb{Z}$ with $\nu_i\geq 0$ satisfying the assumptions of Proposition~\ref{prop:KL_mult_aux2}. By induction on $|\lambda|$ we may assume that the theorem holds for $P(\nu)$. 
Since $\Theta_{i+2}$ sends projectives to projectives (see for instance Lemma~\ref{itm:TL1}), $\Theta_{i+2}P(\nu)$ is projective, hence $\Theta_{i+2}P(\nu)\cong \oplus_\gamma P(\gamma)^{\oplus {n_\gamma}}$ for some finite set of weights $\gamma$ and multiplicities $n_\gamma$. 
Moreover,  $[\Theta_{i+2}P(\nu)]= [{\blacktriangle}(\lambda)]=[{\blacktriangledown}(\lambda+2\omega)]$ by Proposition~\ref{prop:KL_mult_aux2}. If  $P(\gamma)$ occurs as a summand, then $\gamma\in{\blacktriangle}(\lambda)\cap {\blacktriangledown}(\lambda)$. But then Proposition~\ref{prop:comb_arc_diag} forces $\gamma=\lambda$ and $n_\gamma=1$. Hence  $\Theta_{i+2}P(\nu)\cong P(\lambda)$ and the Theorem~\ref{thrm:main-mult2}  follows.
\end{proof}

\section{Action of translation functors on indecomposable projectives}\label{sec:transl_functors_proj}
For simplicity, we disregard the parity in the following section; this means that we will not distinguish between $\Pi M$ and $M$ for $M \in \mathcal{F}_n $.
\subsection{Main result}
Our next goal is to prove the following surprising fact (which should be compared e.g. with \cite[Lemma 2.4]{BS}). 
\begin{theorem}\label{cor:indecomposable} For any dominant $\lambda$ and $i\in\mathbb Z$, $\Theta_i(P(\lambda))$ is either indecomposable projective or zero.
\end{theorem}

\subsection{Detailed analysis of translation functors applied to projectives}
\begin{lemma}\label{lem:simplemove} Let $\lambda$ be a dominant weight.
\begin{enumerate}[1.)]
\item\label{itm:simplemove1} If $f_\lambda(i-2)=1$ and $f_\lambda(i-1)=0$, that is, we have locally
$$d_\lambda = \xymatrix{  &\underset{i-2}{\bullet}  &\underset{i-1}{\circ}    } $$
 then $\Theta_i(P(\lambda))\cong P(\mu)$, where
$$f_\mu(j)=\begin{cases}f_{\lambda}(j)&\text{if $j\neq i-2,i-1$}\\f_{\lambda}(j)+1 (\mathrm{mod} \: 2)&\text{if $j= i-2,i-1$}\end{cases}.$$
That is, 
$$d_\mu = \xymatrix{  &\underset{i-2}{\circ}  &\underset{i-1}{\bullet}    } $$
\item\label{itm:simplemove2} If $f_\lambda(i-2)=0$ and $f_\lambda(i-1)=1$, that is, we have locally
$$d_\lambda = \xymatrix{  &\underset{i-2}{\circ}  &\underset{i-1}{\bullet}    } $$
then $\Theta_i(P(\lambda))=0$.
\end{enumerate}
\end{lemma}
\begin{proof}  In the proof of Theorem~\ref{thrm:main-mult2}, we established part \ref{itm:simplemove1}.). To prove \ref{itm:simplemove2}.) use that
$P(\lambda)\cong\Theta_i(P(\nu))$, where $d_\nu$ is obtained from $d_\lambda$ by moving the black ball at position $i-1$ to position $i-2$.
Then the statement follows from the relation $\Theta^2_i\cong0$.
\end{proof}

\begin{corollary}\label{cor:almosttrans} Assume $\nu_i\leq\lambda_i$ for $i=1,\ldots, n$, then there exists a sequence $j_1>\ldots>j_k$ of integers
such that $P(\lambda)\cong\Theta_{j_k}\cdots \Theta_{j_1}(P(\nu))$, and if we set $P(\nu^r):=\Theta_{j_r}\cdots \Theta_{j_1}(P(\nu))$, then $d_{\nu^\InnaC{r}}$ is obtained from $d_{\nu^{r-1}}$ by moving one 
black ball one position to the right.
\end{corollary}

\begin{lemma}\label{lem:hardmove}
 Let $\lambda$ be a dominant weight.
\begin{enumerate}[1.)]
 \item\label{itm:hardmove1}
 If $f_\lambda(i-2)=f_\lambda(i-1)=1$, that is we have locally
$$d_\lambda = \xymatrix{  &\underset{i-2}{\bullet}  &\underset{i-1}{\bullet}    } $$
then $\Theta_i(P(\lambda))\cong P(\mu)$, where $d_\mu$ is obtained from $d_\lambda$ by moving a black ball from $i-2$ to the (empty) position $j<i-2$ such that $r^+(i-1,j)=0$ and $j$ is maximal with such property. 
\item\label{itm:hardmove2} Let $f_\lambda(i-2)=f_\lambda(i-1)=0$, that is we have locally
$$d_\lambda = \xymatrix{  &\underset{i-2}{\circ}  &\underset{i-1}{\circ}    } $$
If for all $j \geq i$ we have $r^{-}(j,i-2)\neq 0$, then $\Theta_i(P(\lambda))= 0$. Otherwise, pick  $j\geq i$ minimal such that
$r^{-}(j,i-2)=0$. We have $\Theta_i(P(\lambda))\cong P(\mu)$, where $d_\mu$ is obtained from $d_\lambda$ by moving the black ball from $j$ to $i-1$.
\end{enumerate}\end{lemma}
\begin{remark}
  In part \ref{itm:hardmove1}.) of Lemma~\ref{lem:hardmove}, $d_\mu$ is obtained from $d_\lambda$ by moving a black ball from position $i-2$ to the empty position $j$ which has been connected to position $i-1$ in the arrow diagram of $d_{\lambda}$ by the shortest solid arc:
 \begin{eqnarray*}
&& \xymatrix{&\underset{j}{\circ} &{\ldots}  &\underset{i-2}{\bullet}  &\underset{i-1}{\bullet}  \ar@/_0.8pc/[lulld] &\underset{i}{?}   }\quad\text{ with}\\
 d_\mu= &&\xymatrix{&\underset{j}{\bullet}  &{\ldots}   &\underset{i-2}{\circ}  &\underset{i-1}{\bullet}  &\underset{i}{?}   } 
  \end{eqnarray*}
  In part \ref{itm:hardmove2}.) of the Lemma, $d_\mu$ is obtained from $d_\lambda$ by moving a black ball from position $j$ to the position $i-1$, where $j$ has been connected to position $i-2$ in the arrow diagram of  $d_{\lambda}$ by the shortest dashed arrow. (Notice the symmetry with part \ref{itm:hardmove1}.)).
  If $j$ does not exist then $\Theta_i P(\lambda) = 0$.  
  Otherwise, 
  \begin{eqnarray*}
   d_\mu = &&\xymatrix{  &\underset{i-2}{\circ}  &\underset{i-1}{\bullet}  &\underset{i}{\circ}  &{\ldots} &\underset{j}{\circ}   } \quad\text{and}\\
d_\lambda=&&   \xymatrix{  &\underset{i-2}{\circ} \ar@{-->}@/_1.8pc/[rrurrd] &\underset{i-1}{\circ}  &\underset{i}{\circ}  &{\ldots} &\underset{j}{\bullet}   }\\
 \end{eqnarray*}
 Alternatively, one can check that if $i$ is empty, it is the target of a solid arrow in the  arrow diagram of $d_{\lambda}$ whose source is $j$. 
\end{remark}

\begin{proof}[Proof of Lemma~\ref{lem:hardmove}]
We start with \ref{itm:hardmove1}.). First of all, observe that the requirement on position $j$ to be empty is superfluous: indeed, $j$ is maximal such that $r^+(i-1, j) = 0$, and $r^+(i-1, i-2) = 1$ (since $f_{i-2} =1$), so $r^+(i-1, s) > 0$ for any $j < s<i-1$. Thus $j \in \overset{\leftarrow i-1}{\blacktriangle}(\lambda)$, and so $j \notin c_{\lambda}$. Denote
\begin{eqnarray*}
u_i(\lambda)&:= &\sum_{j < i-2} f_{\lambda}(j),
\end{eqnarray*}
(so $u_i(\lambda)$ is the total number of black balls in $d_{\lambda}$ strictly to the left of position $i-2$).
% Since $f_\lambda(i-2)=f_\lambda(i-1)=1$, we have 

%\Cath{$u_i(\lambda) \leq 2$.}

Assume that $f_\lambda(i-2-r)=\ldots =f_\lambda(i-3)=1$ and $f_\lambda(i-3-r)=0$ (so $u_i(\lambda) \geq r+2$). We prove the statement by double induction on $(u_i(\lambda),r)$.

First assume that $r=0$ (this is the base case: $u_i (\lambda) = 0$ 
implies $r=0$). Then
$$d_\lambda  = \xymatrix{&\underset{i-3}{\circ} &\underset{i-2}{\bullet}  &\underset{i-1}{\bullet}    }$$
Then consider $d_\nu$ obtained from 
$d_\lambda$ by moving
the black balls from positions $i-2,i-1$ to positions $i-3,i-2$ respectively:
$$d_\nu  = \xymatrix{&\underset{i-3}{\bullet} &\underset{i-2}{\bullet}  &\underset{i-1}{\circ}    }$$
Then we have $P(\lambda)\cong\Theta_{i-1}\Theta_i(P(\nu))$
and therefore 
$$\Theta_iP(\lambda)\cong\Theta_i\Theta_{i-1}\Theta_i(P(\nu))\cong\Theta_i(P(\nu))\cong P(\mu)$$
by Lemma~\ref{lem:simplemove}, where $$d_\mu  = \xymatrix{&\underset{i-3}{\bullet} &\underset{i-2}{\circ}  &\underset{i-1}{\bullet}    }$$

We now consider the case $r=1$:
$$d_\lambda  = \xymatrix{&\underset{i-4}{\circ} &\underset{i-3}{\bullet}  &\underset{i-2}{\bullet}  &\underset{i-1}{\bullet}    }$$
Then $P(\lambda)\cong\Theta_{i-2}(P(\kappa))$ due to Lemma~\ref{lem:simplemove} ($d_\kappa$ is obtained from $d_\lambda$ by moving a black ball from $i-3$ to $i-4$):
$$d_\kappa  = \xymatrix{&\underset{i-4}{\bullet} &\underset{i-3}{\circ}  &\underset{i-2}{\bullet}  &\underset{i-1}{\bullet}    }$$
Then $\Theta_i(P(\lambda))\cong\Theta_i \Theta_{i-2}(P(\kappa))\cong\Theta_{i-2}\Theta_i(P(\kappa))$ using that $\Theta_i$ and $\Theta_{i-2}$ commute due to Theorem~\ref{thrm:TL_action_categor}.
Applying the previous case ($r=0$) to $\Theta_i(P(\kappa))$, we obtain $\Theta_i(P(\kappa))\cong P(\tau)$ for 
$$ \tau  = \xymatrix{&\underset{i-4}{\bullet} &\underset{i-3}{\bullet}  &\underset{i-2}{\circ}  &\underset{i-1}{\bullet}    }$$
We can apply to $\Theta_{i-2}(P(\tau))$ the
induction hypothesis, since $u_{i-2}(\tau) = u_i(\lambda)-1 $. Set $\Theta_{i-2}(P(\tau))\cong P(\mu)$. Now $d_\tau$ is obtained from $d_\lambda$ by moving a black ball from $i-2$ to $i-4$,
and $d_\mu$ is obtained from $d_\tau$ by moving a black ball from $i-4$ to $j$ such that $r^+(i-3,j)=0$ in $d_{\tau}$, where $j$ is maximal with such property. Therefore, $\Theta_i(P(\lambda)) \cong P(\mu)$ satisfies the statement of the Lemma.

If $r\geq 2$, the argument is similar but easier. Let $p=i-1-r$. Then let $d_{\kappa}$ be the diagram obtained from $d_{\lambda}$ by moving a black ball from position $p-1$ to $p-2$. By Lemma~\ref{lem:simplemove}, $P(\lambda)\cong\Theta_p(\kappa)$, and  
$$\Theta_i(P(\lambda))\cong\Theta_i\Theta_{p}(P(\kappa))\cong\Theta_{p}\Theta_i(P(\kappa)).$$ Again, $\Theta_i$ and $\Theta_p$ commute by Theorem~\ref{thrm:TL_action_categor}. By induction hypothesis (on $r$), we have $\Theta_i(P(\kappa))\cong P(\tau)$ for some $\tau$. We calculate $$P(\mu)\cong\Theta_{p}(P(\tau)) \cong\Theta_i(P(\lambda))$$ using Lemma~\ref{lem:simplemove}.
To get $d_\mu$ from $d_\lambda$ we move a black ball from $i-2$ to $j$ and this $j$ satisfies: $r^+(i-1,j)=0$. Thus $\mu$ satisfies the statement of the Lemma.

The proof of \ref{itm:hardmove2}.) is almost symmetric, with minor differences. First of all, note that if for all $j \geq i$ we have $r^{-}(j,i-2)\neq 0$, then by Theorem~\ref{thrm:main-mult2}, $f_{\lambda'}(j) =0$ for any thick Kac component $\Delta(\lambda')$ of $P(\lambda)$ and any $j \geq i-2$; this implies $\Theta_i \Delta(\lambda') =0$ and thus $\Theta_i P(\lambda) = 0$. Hence, we will assume that $r^{-}(j,i-2)= 0$ for some $j \geq i$. Set 
$$u^-_i(\lambda) = \sum_{j \geq i} f_{\lambda}(j)$$ (this is the number of black balls at positions to the right of $i$, inclusive). By the assumption above, $u^-_i(\lambda) > 0$. Assume that $f_\lambda(i-1)=\cdots= f_\lambda(i-1+r)=0$ and $f_\lambda(i+r)=1$ (such $r$ exists since $u^-_i(\lambda) > 0$). 

We prove our statement by induction on $(u^-_i(\lambda), r)$. Consider the base case $r=0$:
$$d_\lambda  = \xymatrix{&\underset{i-2}{\circ} &\underset{i-1}{\circ}  &\underset{i}{\bullet}    }$$
Then consider $d_\nu$ obtained from 
$d_\lambda$ by moving
the black ball at position $i$ to position $i-2$:
$$d_\nu  = \xymatrix{&\underset{i-2}{\bullet} &\underset{i-1}{\circ}  &\underset{i}{\circ}    }$$
Then we have
$P(\lambda)\cong\Theta_{i+1}\Theta_i(P(\nu)),$  
and therefore by Lemma~\ref{lem:simplemove}
$$\Theta_i P (\lambda)\cong\Theta_i\Theta_{i+1}\Theta_i(P(\nu))\cong\Theta_i(P(\nu))\cong P(\mu)$$
where $$d_\mu  = \xymatrix{&\underset{i-2}{\circ} &\underset{i-1}{\bullet}  &\underset{i}{\circ}      }$$

Next, consider the case $r=1$, that means
$$d_\lambda  = \xymatrix{&\underset{i-2}{\circ} &\underset{i-1}{\circ} &\underset{i}{\circ}   &\underset{i+1}{\bullet}    }$$
Consider $d_\kappa$ obtained from 
$d_\lambda$ by moving
the black ball at position $i+1$ to position $i$:
$$d_\kappa  = \xymatrix{&\underset{i-2}{\circ} &\underset{i-1}{\circ} &\underset{i}{\bullet}  &\underset{i+1}{\circ}    }$$
Then we have $P(\lambda)\cong\Theta_{i+2}P(\kappa)$ by \ref{itm:hardmove1}.), and by Theorem~\ref{thrm:TL_action_categor}
$$\Theta_i (P (\lambda))\cong\Theta_i\Theta_{i+2}(P(\kappa))\cong\Theta_{i+2}\Theta_i(P(\kappa)).$$
We have already seen that $\Theta_i(P(\kappa)) \cong P(\tau)$ for $\tau$ of the form
$$d_\tau = \xymatrix{&\underset{i-2}{\circ} &\underset{i-1}{\bullet} &\underset{i}{\circ}  &\underset{i+1}{\circ}    } $$
(this is the case $r=0$). Then $\Theta_i (P (\lambda) )\cong \Theta_{i+2} (P(\tau) )$. To compute $\Theta_{i+2} (P(\tau) )$, we can apply the induction hypothesis, since $u^-_{\tau}(i+2) = u^-(\lambda) -1$, and obtain $$\Theta_i (P (\lambda) )\cong \Theta_{i+2} (P(\tau) ) \cong P(\mu)$$ where $d_\mu$ is obtained from $d_\tau$ by moving a black ball from $j$ to $i+1$ such that $r^-(j,i)=0$ in $d_{\tau}$, and $j$ is minimal with such property. Then we obtain $\Theta_i(P(\lambda))\cong P(\mu)$, where $\mu$ satisfies the statement of the Lemma.

Finally, the case $r > 1$ is very similar, but easier. Set $p:= i +r$. Let $d_\kappa$ be the diagram obtained from $d_{\lambda}$ by moving a black ball from position $p$ to $p-1$. By Lemma~\ref{lem:simplemove}, we have $P(\lambda)\cong\Theta_{p+1}(\kappa)$, and  
$$\Theta_i(P(\lambda))\cong\Theta_i\Theta_{p+1}(P(\kappa))\cong\Theta_{p+1}\Theta_i(P(\kappa)).$$ Again, $\Theta_i, \Theta_{p+1}$ here commute by Theorem~\ref{thrm:TL_action_categor}. By induction hypothesis (induction on $r$), we have $\Theta_i(P(\kappa))\cong P(\tau)$ for some $\tau$. Calculating $$P(\mu)\cong\Theta_{p}(P(\tau))\cong\Theta_i(P(\lambda))$$ using Lemma~\ref{lem:simplemove}, we obtain the desired $\mu$.
\end{proof}
Theorem~\ref{cor:indecomposable} follows now directly from Lemmas~\ref{lem:simplemove} and \ref{lem:hardmove}.

\section{Multiplicity-freeness results}\label{ssec:translation_proj_coroll}
As an application, we deduce now several crucial multiplicity-freeness results.  

\subsection{Hom spaces between indecomposable projectives}\label{sssec:proj_mult_free}
\begin{proposition}\label{cor:dim_hom_indec}
 Let $\lambda, \mu$ be two dominant weights. Then 
 $$ \dim \Hom_\gg(P(\lambda), P(\mu)) \leq 1$$
 \end{proposition}
\begin{proof}
Recall that $P(\lambda)$, $P(\mu)$ have filtrations by thick and thin Kac modules. In view of Lemma~\ref{lem:ext}, we have:
\begin{eqnarray*}
 \dim \Hom_\gg(P(\lambda), P(\mu)) &=& \dim \Hom_\gg(P(\mu+2\omega), P(\lambda))\\
 &=& \sum_{\tau} (P(\mu+2\omega): \Delta(\tau)) (P(\lambda): \nabla(\tau))\\
 &=& \lvert\blacktriangle(\mu +2\omega) \cap \blacktriangledown(\lambda +2\omega) \mid = \lvert\blacktriangle(\mu) \cap \blacktriangledown(\lambda) \mid
\end{eqnarray*}
Thus, our claim is that $\lvert\blacktriangle(\mu) \cap \blacktriangledown(\lambda) \mid \leq 1$. In case $\mu$ is typical, i.e., $\blacktriangle(\mu) = \{\mu\}$ this is obvious, since $\lvert\blacktriangle(\mu) \cap \blacktriangledown(\lambda) \mid \leq \lvert\blacktriangle(\mu)  \mid = 1$. We now show how to reduce the problem to the case when $\mu$ is typical.

Namely, by Corollary~\ref{cor:almosttrans}, $P(\mu)$ can be obtained from $P(\mu')$ for some typical weight $\mu'$ by a sequence $\Theta_{i_1}, \Theta_{i_2}, \ldots, \Theta_{i_k}$ of translation functors:
$$ P(\mu) \cong \Theta_{i_1}\circ \Theta_{i_2}\circ \ldots\circ \Theta_{i_k} P(\mu')$$

Then by the adjunctions from Proposition~\ref{itm:TL1} we obtain
\begin{eqnarray*}
 \dim \Hom_\gg(P(\lambda), P(\mu)) &= &\dim \Hom_\gg(P(\lambda), \Theta_{i_1}\circ \Theta_{i_2}\circ \ldots\circ \Theta_{i_k} P(\mu')) \\
 &=&\dim \Hom_\gg(\Theta_{i_k +1}\circ \ldots\circ \Theta_{i_2 +1}\circ \Theta_{i_1 +1} P(\lambda), P(\mu'))
\end{eqnarray*}
Since by Theorem~\ref{cor:indecomposable}, the module $\Theta_{i_k +1}\circ \ldots\circ \Theta_{i_2 +1}\circ \Theta_{i_1 +1} P(\lambda)$ is indecomposable, it is enough to consider the case when $\mu$ is a typical weight. Hence the claim follows.
\end{proof}
This immediately implies the following surprising fact:
\begin{theorem}
 Indecomposable projective objects in $\mathcal{F}_n$ are multiplicity-free, i.e., $$[P(\lambda):L(\mu)] \leq 1$$ for any dominant weights $\lambda, \mu$.
\end{theorem}
\begin{proof}
We have $[P(\lambda):L(\mu)] = \dim \Hom_\gg(P(\mu), P(\lambda)) \leq 1$ by Proposition~\ref{cor:dim_hom_indec}.
\end{proof}

\subsection{Multiplicities in translations of simples}
In contrast to the case of projectives, the images of simple modules under the action of translation functors is hard to describe.  Below we give some corollaries of Lemma~\ref{lem:hardmove} concerning the translation of simple modules.

Theorem~\ref{cor:indecomposable} however implies that for any $i$, $\Theta_i L(\lambda)$ is zero or  indecomposable: indeed if nonzero, it has a simple socle and a simple cosocle, since $$\Theta_i P(\lambda) \twoheadrightarrow \Theta_i L(\lambda) \hookrightarrow \Theta_i I(\lambda)$$ and we have seen that the image of each indecomposable projective/injective is again a projective/injective, with simple cosocle and socle.

\begin{corollary}
Let $\lambda \neq \lambda'$ be two distinct weights, and let $i$ be an integer.
\begin{enumerate}
 \item The module $\Theta_i L(\lambda)$ is multiplicity-free.
 \item The modules $\Theta_i L(\lambda)$ and $\Theta_i L(\lambda')$ do not have any common simple components.
\end{enumerate}
\end{corollary}
\begin{proof}
 For any $\mu$, $$[\Theta_i(L(\lambda)): L(\mu)] = \dim \Hom_\gg(P(\mu), \Theta_i L(\lambda)) = \dim \Hom_\gg(\Theta_{i+1} P(\mu), L(\lambda)).$$
Now $\Theta_{i+1} P(\mu)$ is either $0$ or isomorphic to $P(\nu)$ for some $\nu$, hence $[\Theta_i(L(\lambda)): L(\mu)] = \delta_{\nu, \lambda} \leq 1$, implying both statements.
\end{proof}

In fact, one can immediately give some sufficient conditions for $\Theta_i L(\lambda)$ to be zero:
\begin{corollary}
 We have $\Theta_i L(\lambda) = 0$ if $\lambda$ does not satisfy $f_\lambda(\InnaC{i-1})=0, f_\lambda(\InnaC{i})=1$. \InnaC{That is, $\Theta_i L(\lambda) \neq 0$ only if $d_\lambda$ looks locally as}
 \begin{eqnarray*}
 \xymatrix{  &\underset{\InnaC{i-1}}{\circ}  &\underset{\InnaC{i}}{\bullet} }
 \end{eqnarray*}
\end{corollary}
\begin{proof}
 Once again, for any $\mu$, it holds $[\Theta_i(L(\lambda)): L(\mu)] = \dim \Hom_\gg(P(\mu), \Theta_i L(\lambda)) = \dim \Hom_\gg(\Theta_{i+1} P(\mu), L(\lambda)) = \delta_{\Theta_{i+1} P(\mu), P(\lambda)}.$ To have $\Theta_i L(\lambda) \neq 0$, positions $\InnaC{i-1, i}$ of $d_{\lambda}$ have to be of the types obtained in Lemma~\ref{lem:hardmove}.
\end{proof}

\subsection{Socles and cosocles of Kac modules}\label{ssec:socles_Kac}

%\subsection{} 
In this subsection we compute explicitly the socle of standard and the cosocle of costandard modules.  As in the previous section we disregard parity.

\begin{proposition}\label{prop:socle_thick_kac}
 Let $\lambda \in \Lambda_n$.
\begin{enumerate}[1.)]
\item The cosocle of $\nabla(\lambda)$ is $L(\tau)$, where $d_{\lambda}$ is obtained from $d_{\tau}$ by transferring each black ball through the longest solid arrow originating in it.
\item The socle of $\Delta(\lambda)$ is $L(\tau')$, where $d_{\lambda}$ is obtained from $d_{\tau'}$ by transferring black balls through the maximal dashed arcs: i.e., for each empty position we choose the dashed arrow of maximal length originating in it, and transfer the corresponding black ball.
\item In addition, we have $\tau'  = \tau + 2\omega$.
\end{enumerate}
\end{proposition}

\begin{remark}
\label{interprsharp}
 This result gives another interpretation of the construction $(\cdot)^{\sharp}$ given in Section~\ref{ssec:duality_simple}: let $\mu := -w_0^{-1}(\lambda) - \tilde{\gamma}$. Then 
 $L(\tau+2\omega) = L(\mu^\sharp) = L(\mu)^* \hookrightarrow  \Delta(\mu)^*=  \Delta(-w_0\mu - \tilde{\gamma}) = \Delta(\lambda)$
 and thus $\tau = \left( -w_0^{-1}(\lambda) - \tilde{\gamma} \right)^{\sharp} -2\omega.$
\end{remark}

\begin{example}
\mbox{}
 \begin{enumerate}[1.)]
 \item Let $n = 3$, $\lambda =0$:
 $$d_\lambda = \xymatrix{       &\underset{-1}{\circ}  &\underset{0}{\bullet}  &\underset{1}{\bullet}  &\underset{2}{\bullet} &\underset{3}{\circ}   &\underset{4}{\circ}  &\underset{5}{\circ}&\underset{6}{\circ} } $$
 Then $\tau = 2\omega$: 
  % \vspace{1pt}
 $$d_\tau =  \xymatrix{      &\underset{-1}{\circ}  &\underset{0}{\circ}  &\underset{1}{\circ}  &\underset{2}{\bullet}  &\underset{3}{\bullet}  \ar@/_0.8pc/[luld] &\underset{4}{\bullet}\ar@/_1.4pc/[llulld] &\underset{5}{\circ} &\underset{6}{\circ} } $$
 and $\tau' = 4\omega$:
 $$d_{\tau'} =  \xymatrix{     &\underset{-1}{\circ}  &\underset{0}{\circ} \ar@{-->}@/_2.4pc/[rrrurrrd] &\underset{1}{\circ} \ar@{-->}@/_2pc/[rrurrd] &\underset{2}{\circ} \ar@{-->}@/_1.6pc/[rurd]  &\underset{3}{\circ}   &\underset{4}{\bullet} &\underset{5}{\bullet}  &\underset{6}{\bullet} } $$
 \vspace{1pt}
  \item Let $n=3$, $\lambda = \rho$. 
  $$d_\lambda = \xymatrix{ &\underset{-1}{\circ} &\underset{0}{\bullet}  &\underset{1}{\circ}  &\underset{2}{\bullet}  &\underset{3}{\circ} &\underset{4}{\bullet} &\underset{5}{\circ} &\underset{6}{\circ}  }$$
  This is a typical weight, and $\tau = \lambda = \rho$. On the other hand, $\tau' = \rho + 2\omega$:
  $$d_{\tau'} = \xymatrix{ &\underset{-1}{\circ} &\underset{0}{\circ} \ar@{-->}@/_2.6pc/[rrrurrrd] \ar@{-->}@/_2pc/[rrurrd] \ar@{-->}@/_1.4pc/[rurd]  &\underset{1}{\circ}  &\underset{2}{\bullet}  &\underset{3}{\circ} &\underset{4}{\bullet} &\underset{5}{\circ} &\underset{6}{\bullet}  }$$
 \end{enumerate}
 \vspace{0.7cm}
\end{example}

\begin{proof}[Proof of Proposition~\ref{prop:socle_thick_kac}]
Let us call an arrow (either solid or dashed) maximal, if it is the longest arrow originating from its source.

Let $L(\tau)$ be the cosocle of $\nabla(\lambda)$.  Let $\mu = 
-w_0\lambda +{\gamma}$ and let $\tau$ be the highest weight of the simple 
module $L(\mu)^*$. By Lemma~\ref{lem:dual_Kac}, the dual of the map 
$\Delta(\mu) \to \nabla(\mu)$ (whose image 
is $L(\mu)$) is a 
map $$\nabla(\mu)^* = 
\nabla(\lambda)\to\Delta(\mu)^* = \Delta(\lambda - 2 \omega)  
$$ whose image is the simple module $L(\tau) = L(\mu)^*$. Therefore $L(\tau)$ 
is the socle of $\Delta(\lambda - 2 \omega)$ (and thus $L(\tau + 2\omega)$ is 
the socle of $\Delta(\lambda)$).
This implies that $\lambda \in \blacktriangle(\tau) \cap \blacktriangledown(\tau + 2 \omega)$. We know that 
\begin{equation*}
 1= \dim\;\End_\gg (P(\tau)) = \sum_{\mu} (P(\tau): \Delta(\mu)) (P(\tau): \nabla(\mu))= \\
 = \lvert \blacktriangle(\tau) \cap \blacktriangledown(\tau + 2\omega) \mid.
\end{equation*}

Thus it is enough to check that the only element in $\blacktriangle(\tau) \cap \blacktriangledown(\tau + 2 \omega)$ is obtained by sliding black balls along the maximal solid arrows in $d_{\tau}$ (respectively, the maximal dashed arrows in $d_{\tau +2\omega}$).

Denote by $d_{\lambda'}$ the diagram obtained from $d_{\tau}$ by sliding black balls along the maximal solid arcs, and by $d_{\lambda''}$ the diagram obtained from $d_{\tau+2\omega}$ by sliding black balls along the maximal dashed arcs. We wish to show that $\lambda' = \lambda''$, which would imply that they coincide with $\lambda$.

Indeed, consider a position $i$. Assume that $f_{\lambda'}(i)  \neq f_{\lambda''}(i)$. We start by some observations, which will be the main tools used in our proof, and then arrive at a contradiction case-by-case.

\begin{enumerate}[(A)]
 \item\label{soc_Kac_item:1} By definition, $f_{\tau}(i) \neq f_{\lambda'}(i)$ if and only if $i$ is the source or the target of a maximal solid arrow in $d_{\tau}$.
 \begin{enumerate}
  \item This happens only if $f_{\tau}(i) =f_{\tau}(i-1)$.
  \item If $f_{\tau}(i) = 1$, then $f_{\tau}(i) \neq f_{\lambda'}(i)$ iff  $\overset{\leftarrow i}{\blacktriangle}(\tau) \neq \emptyset$ $\Leftrightarrow$ $f_{\tau}(i-1) = 1$ (a black ball is the source of some maximal solid arrow iff it is the source of any solid arc).
 \end{enumerate}
 \item\label{soc_Kac_item:2} Similarly, $f_{\tau}(i-2) = f_{\tau + 2\omega}(i) \neq f_{\lambda''}(i)$ if and only if $i-2$ is the source or the target of a maximal dashed arrow in $d_{\tau}$.
 \begin{enumerate}
  \item If $0= f_{\tau}(i-2) \neq f_{\lambda''}(i)$, then $f_{\tau}(i-1) =0 $. 
  \item If $f_{\tau}(i-2) = 1$, then $f_{\tau}(i-2) \neq f_{\lambda'}(i)$ if and only if $f_{\tau}(i-1) = 1$. This happens since  every black ball is the target of some dashed arrow (see Lemma~\ref{lem:comb_arc_diag2}), and so it is the target of a maximal dashed arrow if and only if $f_{\tau}(i-1) f_{\tau}(i-2) = 1$. 
  \item If $f_{\tau}(i-2) =0$, and $\underset{i-2\dashrightarrow}{\blacktriangledown}(\tau) \neq \emptyset$, then $f_{\tau}(i-2) \neq f_{\lambda'}(i)$ (an empty position is the source of some maximal dashed arrow iff it is the source of any dashed arc).
 \end{enumerate}
\item\label{soc_Kac_item:3} Assume $f_{\tau}( i-2) = f_{\tau}(i-1) =f_{\tau}(i)=0$.

If $\underset{i-2\dashrightarrow}{\blacktriangledown}(\tau) \neq \emptyset$, then setting $j:= \min (\underset{i-2\dashrightarrow}{\blacktriangledown}(\tau))$, we have: $i \in \overset{\leftarrow j}{\blacktriangle}(\tau)$, and the solid arrow from $j$ to $i$ is maximal. And vice versa: if 
$i \in \overset{\leftarrow j}{\blacktriangle}(\tau)$, then $j \in\underset{i-2\dashrightarrow}{\blacktriangledown}(\tau)$.
\vspace{1pt}
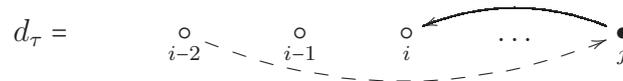
$$d_ \tau =  \xymatrix{   &\underset{i-2}{\circ}  \ar@{-->}@/_1.5pc/[rrurrd] &\underset{i-1}{\circ}  &\underset{i}{\circ} &{\ldots} &\underset{j}{\bullet}  \ar@/_0.8pc/[luld]}$$
\vspace{1pt}
\item\label{soc_Kac_item:4} Similarly, assume $f_{\tau}( i-2) =1$, $ f_{\tau}(i-1) =f_{\tau}(i)=0$. If $i \in \overset{\leftarrow j}{\blacktriangle}(\tau)$ for some $j$, $i-2 \in \underset{j_1\dashrightarrow}{\blacktriangledown}(\tau)$ for some $j_1$, then $j \in \underset{j_1\dashrightarrow}{\blacktriangledown}(\tau)$:
\vspace{1pt}
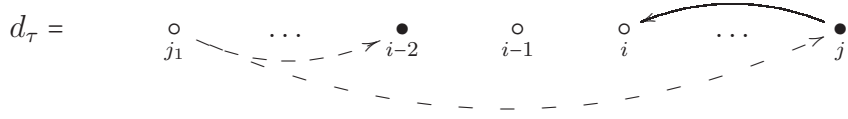
$$ d_\tau =  \xymatrix{  &\underset{j_1}{\circ}\ar@{-->}@/_1pc/[rurd] \ar@{-->}@/_2.5pc/[rrrurrrd] &{\ldots} &\underset{i-2}{\bullet}  &\underset{i-1}{\circ}  &\underset{i}{\circ} &{\ldots} &\underset{j}{\bullet}  \ar@/_0.8pc/[luld] \\ &{} &{} &{} &{} &{} &{} &{}}$$
% \vspace{20pt}
And vice versa: if $i-2 \in \underset{j_1\dashrightarrow}{\blacktriangledown}(\tau)$ for some $j_1$ and $j := \min(\underset{j_1\dashrightarrow}{\blacktriangledown}(\tau) \cap \mathbb{Z}_{ > i-2})$, then $i \in \overset{\leftarrow j}{\blacktriangle}(\tau)$ (and the arrow from $j$ to $i$ is maximal).
\end{enumerate}

We now consider the 8 possibilities for the values $$\vec{a}:=\left(f_{\lambda'}(i) - f_{\lambda''}(i), |f_{\lambda'}(i) - f_{\tau}(i)|, |f_{\lambda''}(i) - f_{\tau}(i-2)| \right) \in \{1, -1\} \times \{0, 1\}\times \{0, 1\}.$$

\begin{enumerate}[i.)]
\item Let $\vec{a} = (-1, 1, 1)$. Then we would have $f_{\tau}(i) =1 \neq f_{\lambda'}(i) =0  \neq f_{\lambda''}(i)=1 \neq f_{\tau}(i-2) =0$, which would imply, by \eqref{soc_Kac_item:1}, \eqref{soc_Kac_item:2} above $ f_{\tau}(i) =f_{\tau}(i-1) = f_{\tau}(i-2) =0$, leading to a contradiction.

\item Let $\vec{a} = (1, 1, 1)$. In this case, we have $f_{\tau}(i) =0 \neq f_{\lambda'}(i) =1  \neq f_{\lambda''}(i)=0 \neq f_{\tau}(i-2) =1$ and by the \eqref{soc_Kac_item:1}, $0 =f_{\tau}(i)=f_{\tau}(i-1)$, $ f_{\tau}(i-2) =1$:
$$d_\tau = \xymatrix{ &\underset{i-2}{\bullet}  &\underset{i-1}{\circ}  &\underset{i}{\circ}}$$
Since $f_{\lambda''}(i) \neq f_{\tau}(i-2)$, $i-2$ is the target of a maximal dashed arrow in $d_{\tau}$, originating in a position $j_1$. But \eqref{soc_Kac_item:4} would imply that $j \in \underset{j_1\dashrightarrow}{\blacktriangledown}(\tau)$, which means that the arrow from $j_1$ to $i-2$ is not maximal, leading to a contradiction.

\item Let  $\vec{a} = (-1, 0, 0)$. In this case $f_{\lambda'}(i) = 0 = f_{\tau}( i)$, $f_{\lambda''}(i) =1 =f_{\tau}( i-2)$. By \eqref{soc_Kac_item:2}, we conclude that $f_{\tau}(i-1) = 0$:
$$d_ \tau =  
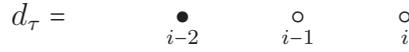
\xymatrix{      &\underset{i-2}{\bullet}  &\underset{i-1}{\circ}  &\underset{i}{\circ} }$$
Let $j_1$ be such that $i-2 \in\underset{j_1\dashrightarrow}{\blacktriangledown}(\tau)$ (such $j_2$ exists by Lemma~\ref{lem:comb_arc_diag2}). Since $f_{\lambda''}(i) =f_{\tau}( i-2)$, the dashed arrow from $j_1$ to $i-2$ is not maximal (see \eqref{soc_Kac_item:2}), so by \eqref{soc_Kac_item:4}, setting $j := \min(\underset{j_1\dashrightarrow}{\blacktriangledown}(\tau) \cap \mathbb{Z}_{ > i-2})$, we have: $i \in \overset{\leftarrow j}{\blacktriangle}(\tau)$ (and the arrow from $j$ to $i$ is maximal). But this would imply $f_{\lambda'}(i) \neq f_{\tau}( i)$ (see \eqref{soc_Kac_item:1}), leading to a contradiction.
\item If $f_{\tau}(i) = f_{\lambda'}(i) = 1$, then $f_{\tau}(i-1) =0 =f_{\lambda''}(i)$. Assume $f_{\lambda''}(i) = f_{\tau}(i-2) = 0$:
$$d_ \tau =  
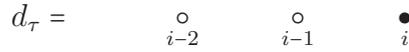
\xymatrix{      &\underset{i-2}{\circ}  &\underset{i-1}{\circ}  &\underset{i}{\bullet} }$$
By \eqref{soc_Kac_item:2}, this would imply that $\underset{i-2\dashrightarrow}{\blacktriangledown}(\tau) = \emptyset$, although $i \in \underset{i-2\dashrightarrow}{\blacktriangledown}(\tau)$, leading to a contradiction. Thus $f_{\lambda''}(i) \neq  f_{\tau}(i-2) = 1$:
$$ d_\tau =  
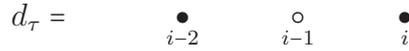
\xymatrix{      &\underset{i-2}{\bullet}  &\underset{i-1}{\circ}  &\underset{i}{\bullet} }$$ 
Then $1 =f_{\tau}(i-2) =f_{\tau}(i-1) =0$ by \eqref{soc_Kac_item:2}, which leads to a contradiction. 
Thus we have eliminated cases $\vec{a} = (1, 0, 0)$ and $\vec{a} = (1, 0, 1)$.
\item If $f_{\lambda'}(i) = 1 \neq f_{\tau}(i)$, and $f_{\lambda''}(i) =0$, then $i$ is the target of some maximal solid arrow starting in $j$, and thus $j \in \underset{i-2\dashrightarrow}{\blacktriangledown}(\tau)$ by \eqref{soc_Kac_item:3},  thus $0= f_{\tau}( i-2) \neq f_{\lambda''}(i) =0$, which is a contradiction.  On the other hand, if $f_{\lambda'}(i) = 0 = f_{\tau}(i)$, and $f_{\lambda''}(i) =1 \neq f_{\tau}( i-2)$, then $\underset{i-2\dashrightarrow}{\blacktriangledown}(\tau) \neq \emptyset$ (again, by \eqref{soc_Kac_item:3}), and so $i$ is the target of a maximal solid arc, implying $f_{\lambda'}(i) \neq f_{\tau}(i)$ which gives a contradiction. We have eliminated the cases $\vec{a} = (1, 1, 0)$ and $\vec{a} = (-1, 0, 1)$.

\item Let  $\vec{a} = (-1, 1, 0)$. In this case $f_{\lambda'}(i) = 0 \neq f_{\tau}( i)$, $f_{\lambda''}(i) =1 =f_{\tau}( i-2)$. By \eqref{soc_Kac_item:1}, the former implies $f_{\tau}( i-1) =f_{\tau}( i)=1$:
$$ d_\tau =  \xymatrix{      &\underset{i-2}{\bullet}  &\underset{i-1}{\bullet}  &\underset{i}{\bullet} }$$ 
But then, by \eqref{soc_Kac_item:2}, $f_{\tau}( i-1)=1$ would imply $f_{\tau}(i-2) \neq f_{\lambda'}(i)$, leading to a contradiction.
\end{enumerate}
This completes the proof of Proposition~\ref{prop:socle_thick_kac}.
\end{proof}

\subsection{Illustrating examples} 
We conclude this section by presenting some examples showing that the length of $\Theta_i L(\lambda)$ can be any number between $0$ and $n+1$.
\begin{example}
 We have already seen that $\Theta_i L(\lambda)$ is frequently zero. Here is an example when it has length $1$ (i.e., it is simple):
 $$d_\lambda =  \xymatrix{  &\underset{\InnaC{i-2}}{\circ} &\underset{i-1}{\circ} &\underset{i}{\bullet} &\underset{i+1}{\bullet}}$$
 Indeed, by Proposition~\ref{itm:TL1}, $[\Theta_i L(\lambda)\::\:L(\mu)]\not=0$  if and only if  $\Theta_{i+1}P(\mu)\cong P(\lambda)$. This is (by the Lemmas~\ref{lem:simplemove} and \ref{lem:hardmove}) only the case for $\mu$ such that $f_\mu(i+1)=f_\mu(i-1)=1$,  $f_\mu(i)=0$ and  $f_\mu(j)=f_\lambda(j)$ otherwise.
 \end{example}
 Next, we give a general example of a $\lambda$ such that $\Theta_i L(\lambda)$ has length greater than $2$. 
 \begin{example}
 Fix $k \in \{0, \ldots, n-1\}$. Let $\lambda$ be a dominant weight with black balls at positions $i, i+2, \ldots, i+2k$, and all other positions $\geq  i-2$ are empty. The rest of the black balls can be arranged arbitrarily at positions $<i-2$:
$$d_\lambda = \xymatrix{  &\underset{i-1}{\circ} &\underset{i}{\bullet} &\underset{i+1}{\circ} &\underset{i+2}{\bullet} &\underset{i+3}{\circ}  &{\ldots}&\underset{i+2k}{\bullet} &\underset{i+2k+1}{\circ} }$$
We claim that $\Theta_i L(\lambda)$ has length $k+2$, i.e., there exist $k+2$ dominant weights $\mu$ such that $\Theta_{i+1} P(\mu)\cong P(\lambda)$. Indeed, by Lemma~\ref{lem:hardmove}, such a weight $\mu$ has the form 
$$\mu_j = \xymatrix{  &\underset{i-1}{\circ} &\underset{i}{\circ} &\underset{i+1}{\circ} &\underset{i+2}{\bullet} &\underset{i+3}{\circ} &{\ldots} &\underset{i+2j}{\bullet} &\underset{i+2j+1}{\bullet} &{\ldots}}$$
for some $j \in \{-1, 0, 1, 2, \ldots, k\}$. Moreover, for any such $j$, $\Theta_{i+1} P(\mu_j)\cong P(\lambda)$.
\end{example}
\begin{example}
\label{weights}
 Let us consider a special case of the above construction. Let $n=3$ and consider the typical weight $\lambda$ with the diagram $$\xymatrix{  &\underset{-1}{\circ} &\underset{0}{\bullet} &\underset{1}{\circ} &\underset{2}{\bullet} &\underset{3}{\circ} &\underset{4}{\bullet} &\underset{5}{\circ}  &\underset{6}{\circ} }$$ there are $4$ possible weights $\mu_j$ as above:
 \begin{align*}
 &d_{\mu_1} = \xymatrix{   &\underset{-1}{\bullet} &\underset{0}{\circ} &\underset{1}{\circ} &\underset{2}{\bullet} &\underset{3}{\circ} &\underset{4}{\bullet} &\underset{5}{\circ} &\underset{6}{\circ}   }\\
 &d_{\mu_2} =\xymatrix{   &\underset{-1}{\circ} &\underset{0}{\circ} &\underset{1}{\bullet} &\underset{2}{\bullet} &\underset{3}{\circ} &\underset{4}{\bullet} &\underset{5}{\circ} &\underset{6}{\circ}  }\\
 &d_{\mu_3} =\xymatrix{   &\underset{-1}{\circ} &\underset{0}{\circ} &\underset{1}{\circ} &\underset{2}{\bullet} &\underset{3}{\bullet} &\underset{4}{\bullet} &\underset{5}{\circ} &\underset{6}{\circ}   }\\
 &d_{\mu_4} =\xymatrix{   &\underset{-1}{\circ} &\underset{0}{\circ} &\underset{1}{\circ} &\underset{2}{\bullet} &\underset{3}{\circ} &\underset{4}{\bullet} &\underset{5}{\bullet} &\underset{6}{\circ}   }
 \end{align*}
 
 In this case the Loewy length of $\Theta_0 L(\lambda)$ is also $4$, and the Loewy filtration has subquotients (from socle to cosocle): $$L(\mu_2); \Pi L(\mu_3); L(\mu_4); \Pi L(\mu_1) $$ To prove this statement, recall that
 \begin{enumerate}
  \item The weight $\lambda$ is typical, so we have an exact sequence
 $$ 0 \to \nabla(\mu_2) \longrightarrow \Theta_0 L(\lambda) = \Theta_0 \nabla(\lambda) \longrightarrow L(\mu_1) = \Pi \nabla(\mu_1) \to 0$$
 \item $\Theta_0 \Delta(\lambda) = \Pi \Delta(\mu_1)$, and has a simple cosocle $\Pi L(\mu_1)$. Thus $$ \Theta_0 \Delta(\lambda) = \Pi \Delta(\mu_1)= \Pi P(\mu_1) \twoheadrightarrow \Theta_0 L(\lambda)$$
 and $\Theta_0 L(\lambda)$ has cosocle $\Pi L(\mu_1)$, with radical $\nabla(\mu_2)$.
 \end{enumerate}
Now, $\nabla(\mu_2)$ has both a simple socle $L(\mu_2)$ and a simple cosocle $L(\mu_4)$, by Lemma~\ref{prop:socle_thick_kac}. 

\end{example}
% \begin{remark}
%  The surjective map $P(\mu_1) \twoheadrightarrow \Theta_0 L(\lambda)$ shows that $\op{Ext}^1_{\mathcal{F}_n}L((\mu_1), L(\mu_3))= \op{Ext}^1_{\mathcal{F}_n}(L(\mu_1), L(\mu_4))  = 0$.
% \end{remark}

% \Cath{I think I do understand now. I claim the old remark above is not correct and like to include the one below. }

\begin{remark}
In particular $ \op{Ext}^1_{\mathcal{F}_n}(L(\mu_4), L(\mu_3))\not=0$ and $ \op{Ext}^1_{\mathcal{F}_n}(L(\mu_3), L(\mu_2))\not=0$. Moreover, the surjective map $P(\mu_1) \twoheadrightarrow \Theta_0 L(\lambda)$ shows also that $ \op{Ext}^1_{\mathcal{F}_n}(L(\mu_1), L(\mu_4))\not=0$ . This fits with Corollary~\ref{cor:Exts}, since $\mu_4\in{\blacktriangledown}(\mu_3)$, $\mu_2\in{\blacktriangle}(\mu_3)$, and $\mu_1\in{\blacktriangledown}(\mu_4)$.
\end{remark}

In particular,  $\Theta_i L(\lambda)$  has Loewy length equal to $4$ and its socle filtration agrees with its radical filtration, hence the module is rigid.  

\begin{remark}
 Observe that the socle/radical filtration does not seem to be related to the order on the weights, since we have $\mu_4 \leq \mu_3 \leq \mu_2 \leq \mu_1.$
\end{remark}

\begin{remark}
The appearing large Loewy lengths are very much in contrast to the classical category $\mathcal{O}$ situation and the case of finite dimensional representations of $\op{GL}(m|n)$, where  it is a (nontrivial!) fact that  $\Theta_s L(\lambda)$, that is a through a wall translated simple module, is zero or of Loewy-length $3$. It always has  simple socle and simple cosocle, but with a possibly large semisimple middle layer which can be described via (parabolic) Kazdhan-Lusztig-Vogan polynomials, see e.g. \cite[8.10, 8.16]{Humphreys} and its translation to the $\op{GL}(m|n)$ situation via \cite[Theorem 1.1]{BS}, \cite[Theorem 1.1]{BS3}. In particular, these modules are always rigid. The rigidity can also be deduced via \cite[Proposition 2.4.1]{BGS} from the fact that $\Theta_i L(\lambda)$ has simple socle and simple cosocle invoking the again  a nontrivial fact that these categories are Koszul, see \cite{BGS}, \cite{BS}. In our periplectic situation we expect that $\Theta_i L(\lambda)$  is still always rigid for any $n$, $\lambda$ and $i$, but a proof would require new techniques. 
\end{remark}

\section{Blocks and action of translation functors}\label{sec:blocks}

In this section we finally determine the blocks of $\mathcal{F}_n $, and describe the action of translation functors on these blocks.
\subsection{Classification of blocks}
\begin{example}
In case $n=1$ the integral dominant weights are just given by integers $\lambda$. Since $\rho=0$ in this case, the $\rho$-shift is irrelevant. 
By Lemma~\ref{lem:typical} and Corollary~\ref{cor:BGG} we have $P(i)=\Delta(i)$ and $\nabla(i)=L(i)$, and so by Theorem~\ref{thrm:main-mult2} $P(i)$ fits into a 
non-split short exact sequence of the form 
$$0\to L(i+2)\rightarrow P(i)\rightarrow L(i)\to 0.$$ 
In particular, the category decomposes into four blocks, two of which are connected by a parity switch. 

Ignoring the parity, the two blocks are: one where the simple modules have an odd integer as their highest weight,  
and the one where the simple modules have an even integer as their highest weight.

Moreover, apart from the identity morphisms on indecomposable projectives, we have only the maps $\phi_i: P(i)\to P(i-2)$ for $i\in\mathbb{Z}$ 
(sending the cosocle to the socle) which satisfy $\phi_{i-2}\phi_{i}=0$. Hence, each block is equivalent to the category of 
finite dimensional complexes of vector spaces; in other words, it is equivalent to the category of finite-dimensional representations of the $A_\infty$-quiver
$$\cdots\diamond \rightarrow \diamond\rightarrow\diamond\rightarrow\cdots$$
with vertices $\diamond$ labeled by integers, and with the relation that the composition of two arrows is always zero.
\end{example}
To state the classification of blocks, we need some notation. For every $\lambda\in\Lambda_n$, set
\begin{eqnarray*}
\kappa(\lambda) = \sum_{i\in c_{\lambda}} (-1)^{i},&\quad\text{and}\quad&q(\lambda)=\begin{cases}0&\text{if  }\,\,|\lambda|\,\equiv 0,1\,\mod 4, \\ 1&\text{if  }\,\, |\lambda|\, \equiv 2,3\,\mod 4.\end{cases}
\end{eqnarray*}
with $c_\lambda$ as in \eqref{omegarho}. 

For instance the weights $\mu_1, \mu_2,\mu_3,\mu_4$ from Example~\ref{weights} have all the same $\kappa$-value $1$, and $q$-values $$q(\mu_1) = q(\mu_3) = 1, \;\;\; q(\mu_2)  = q(\mu_4) = 0.$$ 

\begin{theorem}\label{thrm:blocks}
 The category $\mathcal{F}_n$ has $2(n+1)$ blocks. There is a bijection between blocks and 
$\{ -n, -n+2 ,\ldots, n-2, n\}\times\{+,-\}$. 
We have a decomposition 
$$\mathcal{F}_n=\bigoplus_{p\in  \{ -n, -n+2 ,\ldots, n-2, n\}}\left( \mathcal{F}_n \right)^+_p\oplus \bigoplus_{p\in  \{ -n, -n+2 ,\ldots, n-2, n\}}\left( \mathcal{F}_n \right)^-_p,$$
where the block $\left( \mathcal{F}_n \right)_p^+$ (resp. $\left( \mathcal{F}_n \right)_p^-$) contains all simple modules $L(\lambda)$ with $\kappa(\lambda)=p$ and
with parity of the highest weight vector equal to $q(\lambda)$ (resp. $q(\lambda)+1$).
\end{theorem}
\begin{proof} Notice that the function $q:\Lambda_n\to\mathbb Z_2$ extends uniquely to the whole weight lattice so that $q(\lambda+\alpha)=q(\lambda)+q(\alpha)$
for any weight $\lambda$ and any root $\alpha$. Moreover, 
$q(\alpha)=0$ for any even root $\alpha$ and
$q(\alpha)=1$ for any odd root $\alpha$. We have a decomposition $ \mathcal{F}_n= \mathcal{F}_n^+\oplus  \mathcal{F}_n^-$, where
$ \mathcal{F}_n^+$ (resp. $ \mathcal{F}_n^-$) consists of all modules such that all weight vectors of weight $\mu$ have parity $q(\mu)$ (resp. $q(\mu)+1$).
Now we proceed to decomposing $ \mathcal{F}_n^\pm$ into the blocks. In the argument which follows we ignore the parity consideration.

We consider the minimal equivalence relation on the set of dominant weights such that $\lambda\sim\mu$ if $\mu$ is obtained from $\lambda$ by sliding a black ball 
via a solid or dashed arc. Proposition~\ref{prop:extension} and Theorem~\ref{thrm:main-mult2} imply that $L(\lambda)$ and $L(\mu)$ belong to the same block if and only
if $\lambda\sim\mu$. If we move a black ball via a solid to dashed arrow from position $i$ to position $j$, then $i\equiv j\mod 2$. Hence $\kappa$ is constant on 
every equivalence class. It remains to show that if $\kappa(\lambda)=\kappa(\mu)$, then $\lambda\sim\mu$.
We prove that by induction on $n$. The case $n=1$ is clear since moving via dashed arrow amount to moving the only black ball two positions to the left.
Now let $\nu(i)=\InnaB{\nu_{n-i+1}}$ denote the position of the $i$-th black 
ball counting from the left in the diagram of $\nu$.

Assume first that $\lambda(1)\equiv\mu(1) \mod 2$. Then moving the leftmost black ball in the both diagrams two position to the left several times
we can obtain $\lambda'\sim\lambda$ and $\mu'\sim\mu$ such that $\lambda'(1)=\mu'(1)$, $\lambda'(i)=\lambda(i)$ and $\mu'(i)=\mu(i)$ for all $i>1$,
and $\lambda'(2)-\lambda'(1),\mu'(2)-\mu'(1)>n^2$. Let $\tilde{\lambda}$ and $\tilde{\mu}$ be the diagram obtained from $\lambda'$ and $\mu'$
by removing the leftmost black ball. It is easy to see that the last condition implies that $\lambda'\sim\mu'$ if and only if $\tilde{\lambda}\sim\tilde{\mu}$.
Since $\kappa(\tilde{\lambda})=\kappa(\tilde{\mu})$ we have $\tilde{\lambda}\sim\tilde{\mu}$ by the induction hypothesis. Hence $\lambda\sim\mu$.

Now we assume that $\lambda(1)\equiv 1+\mu(1)\mod 2$. Without loss of generality we may assume that $\lambda(1)\equiv 1\mod 2$ and $\mu(1)\equiv 0 \mod 2$.
Note that in this case $\kappa(\lambda)=\kappa(\mu)\neq \pm n$. Let $r$ be the minimal index such that $\lambda(r)\equiv 0\mod 2$. Moving the $r-1$-st
black ball of $\lambda$ to the right (against dashed arrows), we obtain $\lambda'$ in the same equivalence class such that $\lambda'(r-1)+1=\lambda'(r)$. Now we can move
the $r$-th black ball of $\lambda'$ to the left via a solid arrow, so that the $r$-th black ball jumps over the $r-1$-st. In this way we obtain 
$\lambda''\sim\lambda'\sim\lambda$. Let $s$ be the minimal index such
that $\lambda''(s)\equiv 0\mod 2$. Then clearly $s<r$. Repeating this procedure several times, we will obtain a dominant weight $\nu$ which is equivalent to $\lambda$, and such that
$\nu(1)\equiv 0 \mod 2$. This reduces the situation to the previous case.
\end{proof}

\begin{remark} Note that the blocks $\left( \mathcal{F}_n \right)^{\pm}_n$ and $\left( \mathcal{F}_n \right)^{\pm}_{-n}$ are the only blocks in which all simple modules have typical weights (and thus coincide with thin Kac modules).
\end{remark}

\begin{example}
\begin{enumerate}[1.)]
 \item The trivial module $L(0) = \mathbb C$ lies in $\left( \mathcal{F}_n \right)^+_0$ if $n$ is even, and in $\left( \mathcal{F}_n \right)^+_1$ if $n$ is odd.
 \item The simple module $L(\rho) = \nabla(\rho)$ with a highest weight vector $\vec{v}$ such that $p(v) = 0$ lies in $\left( \mathcal{F}_n \right)^{\pm}_n$, where the sign is $+$ if $n \equiv -1, 0, 1, 2 (\mathrm{mod} \; 8)$, and $-$ otherwise.
\end{enumerate}

\end{example}
\subsection{Action of translation functors on blocks}
The following result describes the action of translation functors on blocks.
\begin{corollary} 
\label{cor:blocks}
Let $i \in \mathbb Z$, $p \in \{ -n, -n +2 , \ldots, n-2, n\}$. Then we have 
$$\Theta_i(\left( \mathcal{F}_n \right)^\pm_p)\subset \begin{cases}\left( \mathcal{F}_n \right)_{p+2}^\pm \,\,\text{if}\,\,i\,\, \text{is odd and}\,\,\tfrac{n-p}{2} \,\, \text{is even},\\
\left( \mathcal{F}_n \right)_{p+2}^\mp \,\,\text{if}\,\,i\,\, \text{is odd and}\,\,\tfrac{n-p}{2} \,\, \text{is odd},\\
\left( \mathcal{F}_n \right)_{p-2}^\pm \,\,\text{if}\,\,i\,\, \text{is even and}\,\,\tfrac{n-p}{2} \,\, \text{is even},\\
\left( \mathcal{F}_n \right)_{p-2}^\mp \,\,\text{if}\,\,i\,\, \text{is even and}\,\,\tfrac{n-p}{2} \,\, \text{is odd}.\end{cases}$$
\end{corollary}

\bibliographystyle{amsalpha}
\bibliography{references}

\newcommand{\etalchar}[1]{$^{#1}$}
\def\cprime{$'$}
\providecommand{\bysame}{\leavevmode\hbox to3em{\hrulefill}\thinspace}
\providecommand{\MR}{\relax\ifhmode\unskip\space\fi MR }
% \MRhref is called by the amsart/book/proc definition of \MR.
\providecommand{\MRhref}[2]{%
  \href{http://www.ams.org/mathscinet-getitem?mr=#1}{#2}
}
\providecommand{\href}[2]{#2}
\begin{thebibliography}{BDEA{\etalchar{+}}}

\bibitem[AMR06]{AMR}
S.~Ariki, A.~Mathas, and H.~Rui, \emph{Cyclotomic {N}azarov-{W}enzl algebras},
  Nagoya Math. J. \textbf{182} (2006), 47--134.

\bibitem[BDEA{\etalchar{+}}]{second}
M.~Balagovic, Z.~Daugherty, I.~Entova-Aizenbud, I.~Halacheva, J.~Hennig, M.~S.
  Im, G.~Letzter, E.~Norton, V.~Serganova, and C.~Stroppel, \emph{The affine
  {VW} supercategory}, arXiv:1801.04178, to appear in Selecta Math.

\bibitem[BGS96]{BGS}
A.~Beilinson, V.~Ginzburg, and W.~Soergel, \emph{Koszul duality patterns in
  representation theory}, J. Amer. Math. Soc. \textbf{9} (1996), no.~2,
  473--527.

\bibitem[BKN11]{BKN}
B.~D. Boe, J.~R. Kujawa, and D.~K. Nakano, \emph{Complexity and module
  varieties for classical {L}ie superalgebras}, Int. Math. Res. Not. IMRN
  \textbf{3} (2011), 696--724.

\bibitem[Bru03]{B03}
J.~Brundan, \emph{{Kazhdan-{L}usztig polynomials and character formulae for the
  {L}ie superalgebra {$\mathfrak{gl}(m\vert n)$}}}, {J. Amer. Math. Soc.}
  \textbf{{16}} ({2003}), no.~{1}, {185--231}.

\bibitem[Bru04]{B04}
J.~Brundan, \emph{Kazhdan-{L}usztig polynomials and character formulae for the
  {L}ie superalgebra {${\mathfrak q}(n)$}}, Adv. Math. \textbf{182} (2004),
  no.~1, 28--77.

\bibitem[BS11]{BS3}
J.~Brundan and C.~Stroppel, \emph{Highest weight categories arising from
  {K}hovanov's diagram algebra {III}: Category $\mathcal{O}$}, Repr. Theory
  \textbf{15} (2011), 170--243.

\bibitem[BS12]{BS}
\bysame, \emph{Highest weight categories arising from {K}hovanov's diagram
  algebra {IV}: the general linear supergroup}, J. Eur. Math. Soc. (JEMS)
  \textbf{14} (2012), no.~2, 373--419.

\bibitem[CE18]{CE}
K.~Coulembier and M.~Ehrig, \emph{The periplectic {B}rauer algebra {II}:
  {D}ecomposition multiplicities}, J. Comb. Algebra \textbf{2} (2018), no.~1,
  19--46.

\bibitem[Che15]{C}
C.-W. Chen, \emph{Finite-dimensional representations of periplectic {L}ie
  superalgebras}, J. Algebra \textbf{443} (2015), 99--125.

\bibitem[CK13]{CK}
B.~Cooper and V.~Krushkal, \emph{Handle slides and localizations of
  categories}, Int. Math. Res. Not. IMRN (2013), no.~10, 2179--2202.

\bibitem[CLW11]{CLW}
S.-J. Cheng, N.~Lam, and W.~Wang, \emph{Super duality and irreducible
  characters of ortho-symplectic {L}ie superalgebras}, Invent. Math.
  \textbf{183} (2011), no.~1, 189--224.

\bibitem[Cou18]{Col}
K.~Coulembier, \emph{The periplectic {B}rauer algebra}, Proc. Lon. Math. Soc.,
  2018, pp.~441--482.

\bibitem[CP18]{CP}
C.-W. Chen and Y.-N. Peng, \emph{Affine periplectic {B}rauer algebras}, J.
  Algebra \textbf{501} (2018), 345--372.

\bibitem[CPS94]{CPS}
E.~Cline, B.~Parshall, and L.~Scott, \emph{{The homological dual of a highest
  weight category}}, Proc. London Math. Soc. \textbf{68} (1994), 296--316.

\bibitem[Dri86]{Dri}
V.~G. Drinfeld, \emph{Degenerate affine {H}ecke algebras and {Y}angians},
  Funktsional. Anal. i Prilozhen. \textbf{20} (1986), no.~1, 69--70.

\bibitem[DRV13]{DRV}
Z.~Daugherty, A.~Ram, and R.~Virk, \emph{Affine and degenerate affine {BMW}
  algebras: actions on tensor space}, Selecta Math. (N.S.) \textbf{19} (2013),
  no.~2, 611--653.

\bibitem[EQ16]{EQ}
A.~Ellis and Y.~Qi, \emph{The differential graded odd nil{H}ecke algebra},
  Comm. Math. Phys. \textbf{344} (2016), no.~1, 275--331.

\bibitem[ES17]{ES1}
M.~Ehrig and C.~Stroppel, \emph{On the category of finite-dimensional
  representations of {${\rm OSp}(r|2n)$}: {P}art {I}}, Representation
  theory---current trends and perspectives, Eur. Math. Soc., Z\"urich, 2017,
  pp.~109--170.

\bibitem[ES18]{ES}
\bysame, \emph{Nazarov-{W}enzl algebras, coideal subalgebras and categorified
  skew {H}owe duality}, Adv. Math. \textbf{331} (2018), 58--142.

\bibitem[Fuk86]{Fuks}
D.~B. Fuks (ed.), \emph{Cohomology of infinite-dimensional {L}ie algebras},
  translated from the russian by a. b. sosinski\u\i ed., Contemporary Soviet
  Mathematics, Consultants Bureau, New York, 1986.

\bibitem[Gor01]{G}
M.~Gorelik, \emph{The center of a simple {$P$}-type {L}ie superalgebra}, J.
  Algebra \textbf{246} (2001), no.~1, 414--428.

\bibitem[GS10]{GS}
C.~Gruson and V.~Serganova, \emph{Cohomology of generalized supergrassmannians
  and character formulae for basic classical {L}ie superalgebras}, Proc. Lond.
  Math. Soc. (3) \textbf{101} (2010), no.~3, 852--892.

\bibitem[Hum08]{Humphreys}
J.~E. Humphreys, \emph{Representations of semisimple {L}ie algebras in the
  {BGG} category {$\mathcal{O}$}}, Graduate Studies in Mathematics, vol.~94,
  American Mathematical Society, Providence, RI, 2008.

\bibitem[Kac77]{K77}
V.~G. Kac, \emph{Lie superalgebras}, Advances in Math. \textbf{26} (1977),
  no.~1, 8--96.

\bibitem[Kac78]{K78}
\bysame, \emph{Representations of classical {L}ie superalgebras}, Differential
  geometrical methods in mathematical physics, {II} ({P}roc. {C}onf., {U}niv.
  {B}onn, {B}onn, 1977), Lecture Notes in Math., vol. 676, Springer, 1978,
  pp.~597--626.

\bibitem[Kas95]{Kassel}
C.~Kassel, \emph{Quantum groups}, Graduate Texts in Mathematics, vol. 155,
  Springer, 1995.

\bibitem[KT17]{KT}
J.~Kujawa and B.~Tharp, \emph{The marked {B}rauer category}, J. Lond. Math.
  Soc. (2) \textbf{95} (2017), no.~2, 393--413.

\bibitem[Mac95]{Macd}
I.~G. Macdonald, \emph{Symmetric functions and {H}all polynomials}, second ed.,
  Oxford Mathematical Monographs, The Clarendon Press, Oxford University Press,
  1995.

\bibitem[Moo03]{M}
D.~Moon, \emph{Tensor product representations of the {L}ie superalgebra
  {${\mathfrak p}(n)$} and their centralizers}, Comm. Algebra \textbf{31}
  (2003), no.~5, 2095--2140.

\bibitem[Naz96]{Nazarov}
M.~Nazarov, \emph{Young's orthogonal form for {B}rauer's centralizer algebra},
  J. Algebra \textbf{182} (1996), no.~3, 664--693.

\bibitem[PS94]{PS94}
I.~Penkov and V.~Serganova, \emph{Generic irreducible representations of
  finite-dimension {L}ie superalgebras}, Internat. J. Math. \textbf{5} (1994),
  no.~3, 389--419.

\bibitem[PS97]{PS}
\bysame, \emph{Characters of irreducible {$G$}-modules and cohomology of
  {$G/P$} for the {L}ie supergroup {$G=Q(N)$}}, J. Math. Sci. \textbf{84}
  (1997), no.~5, 1382--1412.

\bibitem[Ser96]{Vsel}
V.~Serganova, \emph{Kazhdan-{L}usztig polynomials and character formula for the
  {L}ie superalgebra {${\mathfrak gl}(m|n)$}}, Selecta Math. (N.S.) \textbf{2}
  (1996), no.~4, 607--651.

\bibitem[Ser02]{Ser}
\bysame, \emph{On representations of the {L}ie superalgebra {$p(n)$}}, J.
  Algebra \textbf{258} (2002), no.~2, 615--630.

\bibitem[Ser11]{Serqr}
\bysame, \emph{Quasireductive supergroups}, New developments in {L}ie theory
  and its applications, Contemp. Math., vol. 544, Amer. Math. Soc., Providence,
  RI, 2011, pp.~141--159.

\bibitem[Str05]{Str}
C.~Stroppel, \emph{Categorification of the {T}emperley-{L}ieb category,
  tangles, and cobordisms via projective functors}, Duke Math. J. \textbf{126}
  (2005), no.~3, 547--596.

\bibitem[Wey03]{W}
J.~Weyman, \emph{Cohomology of vector bundles and syzygies}, Cambridge Tracts
  in Mathematics, no. 149, Cambridge University Press, 2003.

\end{thebibliography}
\end{document}